\documentclass[11pt]{article}

\usepackage[T1]{fontenc}
\usepackage[utf8]{inputenc}
\usepackage{lmodern}

\usepackage{parskip}
\usepackage{amsmath, amssymb, amsthm}
\usepackage{mathrsfs}
\usepackage{subcaption}
\usepackage{booktabs}
\usepackage{tabularx}
\newcolumntype{Y}{>{\centering\arraybackslash}X}

\usepackage[numbers]{natbib}
\let\parencite\citep

\usepackage{geometry}
\usepackage{thmtools}

\declaretheoremstyle[
  headfont=\bfseries\itshape,
  bodyfont=\normalfont,
  spaceabove=16pt,
  spacebelow=14pt,
  headpunct={.},
  postheadspace=0.5em,
  mdframed={ hidealllines=true, innerleftmargin=10pt, innerrightmargin=10pt }
]{elegant}

\declaretheoremstyle[
  headfont=\bfseries\itshape,
  bodyfont=\normalfont,
  spaceabove=12pt,
  spacebelow=12pt,
  headpunct={.},
  postheadspace=0.5em,
  mdframed={ hidealllines=true, innerleftmargin=10pt, innerrightmargin=10pt }
]{definitionstyle}

\declaretheoremstyle[
  headfont=\itshape,
  bodyfont=\itshape,
  spaceabove=10pt,
  spacebelow=10pt,
  headpunct={.},
  postheadspace=0.5em,
]{remarkstyle}

\declaretheorem[style=elegant, numberwithin=section, name=Theorem]{theorem}
\declaretheorem[style=elegant, sibling=theorem, name=Proposition]{proposition}
\declaretheorem[style=elegant, sibling=theorem, name=Lemma]{lemma}
\declaretheorem[style=elegant, sibling=theorem, name=Corollary]{corollary}
\declaretheorem[style=elegant, sibling=theorem, name=Assumption]{assumption}
\declaretheorem[style=elegant, sibling=theorem, name=Principle]{principle}
\declaretheorem[style=definitionstyle, sibling=theorem, name=Definition]{definition}
\declaretheorem[style=definitionstyle, sibling=theorem, name=Example]{example}
\declaretheorem[style=remarkstyle, sibling=theorem, name=Remark]{remark}

\usepackage{hyperref}
\usepackage{cleveref}

\usepackage{mathtools}
\DeclareMathOperator*{\argmin}{argmin}

\usepackage{tikz}
\usetikzlibrary{calc, shapes, arrows, decorations.pathmorphing}

\usepackage{stmaryrd}

\title{A Wasserstein Geometric Framework for Hebbian Plasticity}

\author{Ulrich Tan \\ORCID: https://orcid.org/0009-0001-2907-501X}

\date{April 11, 2025}

\begin{document}

\maketitle

\begin{abstract}
  We introduce the \emph{Tan--HWG framework} (Hebbian--Wasserstein--Geometry), a geometric theory of Hebbian plasticity in which memory states are modeled as probability measures evolving through Wasserstein minimizing movements. Hebbian learning rules are formalized as \emph{Hebbian energies} satisfying a sequential stability condition, ensuring well-posed fiberwise JKO updates, optimal-transport realizations, and an energy descent inequality.

  This variational structure induces a fundamental separation between \emph{internal} and \emph{observable} dynamics. Internal memory states evolve along Wasserstein geodesics in a latent curved space, while observable quantities---such as effective synaptic weights---arise through geometric projection maps into external spaces. Simplicial projections recover classical affine schemes (including exponential moving averages and mirror descent), while revealing synaptic competition and pruning as geometric consequences of mass redistribution. Hilbertian projections provide a geometric account of phase alignment and multi-scale coherence.

  Classical neural networks appear as flat projections of this curved dynamics, while the framework naturally accommodates richer distributional representations, including structural weights and embedding memories, and their spectral extensions in complex internal spaces.

  Under mild Lipschitz regularity assumptions, including a quasi-stationary ``sleep-mode'' regime, we establish the existence of continuous-time limit curves. This yields a variational formulation of memory consolidation as a perturbed Wasserstein gradient flow. The framework thus provides a unified geometric foundation for synaptic plasticity, representation dynamics, and context-dependent computation.
\end{abstract}

\section*{Introduction}

Hebbian plasticity~\parencite{hebb} remains one of the central principles for understanding how neural systems adapt to experience. Classical formulations describe how synaptic strengths change as a function of co-activation, while modern approaches increasingly seek geometric or variational interpretations of synaptic dynamics~\parencite{gerstner2014neuronal, dayanabbott, sutton2018reinforcement}. Yet, despite the diversity of existing models, a unified geometric formulation of Hebbian learning has remained elusive.

Optimal transport and Wasserstein geometry provide a natural language for describing the evolution of probability distributions~\parencite{Villani2003, santambrogio}. The JKO scheme~\parencite{jko1998}, originally introduced to discretize gradient flows in Wasserstein space, offers a principled variational framework for modeling distributional dynamics. This makes Wasserstein geometry a compelling candidate for describing synaptic plasticity, although its relevance to Hebbian learning has not been systematically explored.

This work introduces the \emph{Tan--HWG framework} (Hebbian--Wasserstein--Geometry), a geometric theory of Hebbian plasticity in which memory states are represented as probability measures evolving through fiberwise Wasserstein proximal steps. Plasticity unfolds in a latent curved space, while a global contextual signal encodes the functional state of the system and drives the local updates. This formulation reveals a structural separation between \emph{internal} dynamics in Wasserstein space and \emph{observable} dynamics obtained through geometric projection maps.

Within this framework, classical neural networks appear as flat projections of a richer distributional geometry. The probabilistic representation naturally induces two complementary levels of structure: \emph{structural weights}, encoding the allocation of synaptic influence, and \emph{embedding memories}, encoding internal synaptic states. In complex internal spaces, these embeddings give rise to \emph{spectral memories} that capture phase information. These phenomena emerge from the geometry itself rather than from additional modeling assumptions.

The main contributions of this work are as follows:
\begin{enumerate}
  \item \emph{A geometric formulation of Hebbian plasticity.}
  We identify a broad class of Hebbian energies for which the minimizing-movement scheme is well posed, admits fiberwise optimal-transport realizations, and satisfies an energy descent inequality.

  \item \emph{A dual internal/observable dynamics.}
  Internal memory states evolve along geodesics in the Wasserstein space, while observable quantities arise through projections onto metric external spaces (e.g.\ simplex or Hilbert spaces). This duality provides geometric interpretations of synaptic competition, pruning, multi-scale coherence, neuronal assemblies, and phase alignment.

  \item \emph{A variational account of memory consolidation.}
  Under mild Lipschitz regularity assumptions, including a quasi-stationary ``sleep-mode'' regime, we establish the existence of continuous-time limit curves, yielding a perturbed Wasserstein gradient-flow formulation of memory consolidation.
\end{enumerate}

Overall, the framework provides a unified geometric foundation for synaptic plasticity, representation dynamics, and context-dependent computation, revealing how observable learning rules arise as projections of an underlying Wasserstein geometry.

\section{Mathematical foundations}\label{sec:foundations}

\subsection{Internal state space}

Let $(X,\mu)$ be a probability space, or more generally a finite-measure space, representing the set of (postsynaptic) neurons or readout units, $\mu$ serving as a reference measure on $X$.
Let $(Y,d_Y)$ be a complete, separable, geodesic metric space (i.e.\ a geodesic Polish space), the \emph{internal state space}, representing synaptic internal states.
This assumption ensures that $\mathcal P_2(Y)$ is itself a geodesic Polish space, which will be crucial for the variational dynamics introduced later.

\begin{example}[Compatible internal state spaces]
  Typical admissible examples of $(Y,d_Y)$ include Euclidean spaces $\mathbb R^d$, and complete geodesic metric graphs, namely graphs whose edges are finite-length closed segments (ensuring geodesicity) and which have no open ends, missing vertices, or incomplete rays (ensuring completeness). More generally, CAT(0) spaces~\parencite{BridsonHaefliger1999}, which are also Polish, are admissible internal state space.
\end{example}

\subsection{Memory fields}

For each $x\in X$, the local memory state is a probability measure $\rho_x \in \mathcal{P}_2(Y)$, and we denote by
\[
\mathcal{X} := \big\{ \rho: x\mapsto\rho_x\in\mathcal{P}_2(Y) \text{ Borel measurable} \big\}
\]
the space of memory fields, $\mathcal{P}_2(Y)$ being naturally equipped with Wasserstein--2 metric and its induced Borel $\sigma$-algebra.

\begin{definition}[$\mu$-compatible fields]
  A field $\rho\in\mathcal X$ is said to be \emph{$\mu$-compatible} if
  \[
  \int_X\int_Y d_Y(y,y_0)^2\,\mathrm d\rho_x(y)\,\mathrm d\mu(x)<+\infty
  \]
  for some (hence any) $y_0\in Y$.
  We denote by $\mathcal X_\mu(X, Y)$, or simply $\mathcal X_\mu$, the set of all such $\mu$-compatible fields.
\end{definition}

By definition of $\mathcal P_2(Y)$, for $\rho\in\mathcal X$, for each $x$, we know that the second moment
\[
\int_Y d_Y(y,y_0)^2\,\mathrm d\rho_x(y)<+\infty
\]
is finite. The $\mu$-compatibility for a memory field hence states that the $\mu$-average of all fiber moments of the fields is finite.

Measure-compatibility ensures tightness and compactness in the product space $\mathcal X$ with respect to the base measure $\mu$. This property will be crucial later in the paper, in particular for the compactness of the discrete trajectories and the passage to the continuous-time limit in the minimizing-movement scheme.

\begin{remark}[Geometry induced by the memory field]
  Although not used in the developments below, the memory field $\rho=(\rho_x)_{x\in X}$ canonically induces a pseudo‑metric on $X$ via
  \[
  d_\rho(x,y)=W_2(\rho_x,\rho_y),
  \]
  endowing $X$ with a data‑dependent geometry. This viewpoint highlights that the similarity structure between units is not fixed a priori but emerges from their local memory states. One may further equip $(X,d_\rho)$ with its Hausdorff measure $\mu_\rho$, yielding a complete, separable induced metric‑measure structure $(X,d_\rho,\mu_\rho)$. While this evolving geometry will not be exploited in the sequel, it provides a useful conceptual interpretation of the framework as defining a dynamic cognitive geometry shaped by the memory field.
\end{remark}

\subsection{Example: synaptic relative weight representation}
\label{ex:example0}

Let $X$ denote a set of neurons. We consider two finite subsets:
\[
X_N := \{x_1,\dots,x_N\} \subset X,
\qquad
X'_M := \{x'_1,\dots,x'_M\} \subset X,
\]
representing two populations of neurons. We assume that each neuron in $X_N$ is connected to each neuron in $X'_M$, as in a simple two-layer feed-forward architecture (if $X_N \cap X'_M \neq \emptyset$, the network no longer has a feed-forward topology, but the construction below still applies).

To represent the space of possible synaptic configurations from $X'_M$ to $X_N$, we introduce a complete geodesic metric graph $Y$ with the following structure:
\begin{itemize}
  \item a distinguished subset of leaves
  \[
  Y_M := \{y_1,\dots,y_M\} \subset Y
  \]
  corresponding to the neurons $X'_M$;
  \item for each $y_i$, an edge connecting $y_i$ to a central hub $y^o \in Y$, so that $Y$ is a star-shaped tree with center $y^o \in Y$ (not associated with any neuron) and leaves $Y_M$;
  \item the metric on $Y$ is the natural path metric induced by the tree structure.
\end{itemize}

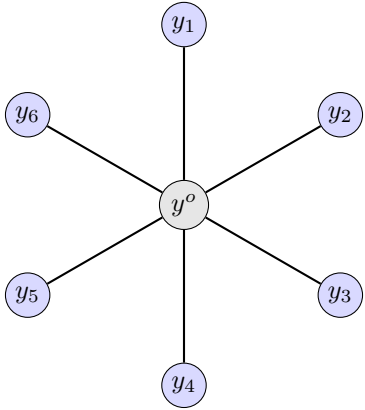
\begin{figure}[!hbtp]
  \centering
  \begin{tikzpicture}[scale=1.2, every node/.style={font=\small}]
    \node[circle, draw, fill=gray!20, inner sep=2pt] (c) at (0,0) {$y^o$};
    \node[circle, draw, fill=blue!15, inner sep=2pt] (y1) at (90:2) {$y_1$};
    \node[circle, draw, fill=blue!15, inner sep=2pt] (y2) at (30:2) {$y_2$};
    \node[circle, draw, fill=blue!15, inner sep=2pt] (y3) at (-30:2) {$y_3$};
    \node[circle, draw, fill=blue!15, inner sep=2pt] (y4) at (-90:2) {$y_4$};
    \node[circle, draw, fill=blue!15, inner sep=2pt] (y5) at (210:2) {$y_5$};
    \node[circle, draw, fill=blue!15, inner sep=2pt] (y6) at (150:2) {$y_6$};
    \foreach \i in {1,2,3,4,5,6}
      \draw[thick] (c) -- (y\i);
    \node at (0,-2.6) {$Y$ is a star-shaped geodesic tree with center $y^o$ and leaves $Y_M=\{y_1,\dots,y_M\}$};
  \end{tikzpicture}
  \caption{Star-shaped metric tree $Y$ representing the synaptic targets $X'_M$.}
  \label{fig:star-tree}
\end{figure}

$X_N$ is naturally equipped with the normalized measure
\[
\mu = \frac{1}{N}\sum_{j=1}^N \delta_{x_j},
\]
and for each postsynaptic neuron $x_j \in X_N$, we represent its synaptic profile by a probability measure
\[
\rho_{x_j} = \sum_{i=1}^M p_{i,j}\,\delta_{y_i} \in P_2(Y),
\]
supported on the leaves $Y_M$, where $\rho_{x_j}(\{y_i\})=p_{i,j} \in [0,1]$ denotes the \emph{relative} synaptic strength from $x'_i$ (or $y_i$) to $x_j$, with the constraint
\[
\sum_{i=1}^M p_{i,j} = 1.
\]

This representation captures only the relative profile of synaptic weights, not their absolute magnitudes. In particular, two nontrivial positive measures are considered equivalent whenever they differ by a global scaling factor:
\[
\forall\, \nu,\nu' \in \mathcal M_+(Y),
\qquad
\nu \sim \nu'
\quad\Longleftrightarrow\quad
\exists\, c>0,\ \nu = c\,\nu'.
\]
Thus, the Wasserstein geometry acts on equivalence classes of synaptic weight profiles, reflecting the fact that only the relative distribution of synaptic strengths is modeled here.
This normalization is not intended as a biological assumption, but arises naturally from the Wasserstein formalism: the space $\mathcal P_2(Y)$ provides a geometrically well-structured setting in which synaptic configurations can evolve along continuous geodesics.

The Wasserstein distance $W_2(\rho_{x_j},\rho_{x_j}')$ quantifies the minimal transport cost required to transform one synaptic profile into another. Because $Y$ is a geodesic metric tree, this transport is uniquely defined and proceeds along the edges of the star-shaped graph. Intermediate points along such a Wasserstein geodesic between two profiles correspond to ``virtual'' synaptic configurations in which mass may temporarily reside along the edges of the tree or at the central hub $y^o$. These intermediate states do not correspond to actual neurons, but they arise naturally from the requirement that synaptic evolution be continuous in $\mathcal P_2(Y)$.

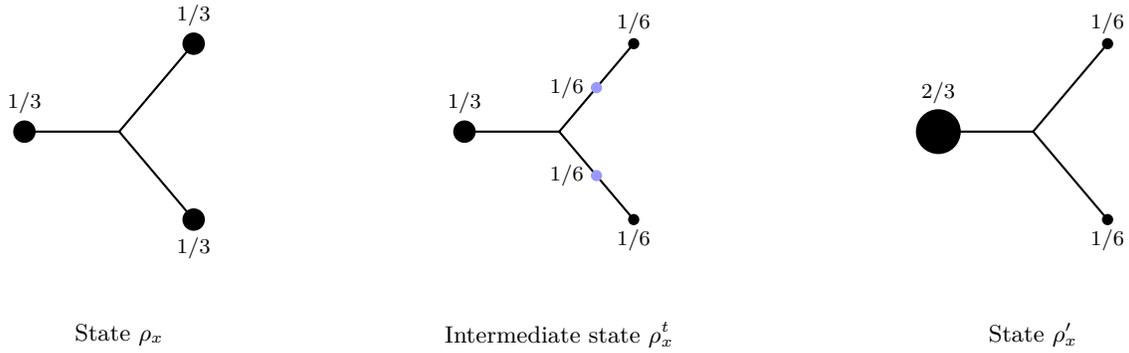
\begin{figure}[!hbtp]
  \centering
  \begin{tikzpicture}[scale=0.9,>=stealth]
    \begin{scope}[shift={(-9.5,0)}]
      \coordinate (C1) at (0,0);
      \coordinate (Y1a) at (1.1,1.3);
      \coordinate (Y1b) at (1.1,-1.3);
      \coordinate (Y1c) at (-1.4,0);
      \draw[thick] (C1)--(Y1a) (C1)--(Y1b) (C1)--(Y1c);
      \foreach \p in {Y1a,Y1b,Y1c} \node[circle,fill=black,inner sep=3pt] at (\p) {};
      \node at (Y1a) [above=3pt] {\scriptsize $1/3$};
      \node at (Y1b) [below=3pt] {\scriptsize $1/3$};
      \node at (Y1c) [above=3pt] {\scriptsize $1/3$};
      \node at (0,-3) {\footnotesize State $\rho_x$};
    \end{scope}
    \begin{scope}[shift={(-3,0)}]
      \coordinate (C2) at (0,0);
      \coordinate (Y2a) at (1.1,1.3);
      \coordinate (Y2b) at (1.1,-1.3);
      \coordinate (Y2c) at (-1.4,0);
      \draw[thick] (C2)--(Y2a) (C2)--(Y2b) (C2)--(Y2c);
      \foreach \p in {Y2a,Y2b} \node[circle,fill=black,inner sep=1.5pt] at (\p) {};
      \node[circle,fill=black,inner sep=3pt] at (Y2c) {};
      \node at (Y2a) [above] {\scriptsize $1/6$};
      \node at (Y2b) [below] {\scriptsize $1/6$};
      \node at (Y2c) [above=3pt] {\scriptsize $1/3$};
      \node[circle,fill=blue!40,inner sep=1.5pt] at ($(C2)!0.5!(Y2a)$) {};
      \node at ($(C2)!0.5!(Y2a)$) [left=1pt] {\scriptsize $1/6$};
      \node[circle,fill=blue!40,inner sep=1.5pt] at ($(C2)!0.5!(Y2b)$) {};
      \node at ($(C2)!0.5!(Y2b)$) [left=1pt] {\scriptsize $1/6$};
      \node at (0,-3) {\footnotesize Intermediate state $\rho^t_x$};
    \end{scope}
    \begin{scope}[shift={(4,0)}]
      \coordinate (C3) at (0,0);
      \coordinate (Y3a) at (1.1,1.3);
      \coordinate (Y3b) at (1.1,-1.3);
      \coordinate (Y3c) at (-1.4,0);
      \draw[thick] (C3)--(Y3a) (C3)--(Y3b) (C3)--(Y3c);
      \foreach \p in {Y3a,Y3b} \node[circle,fill=black,inner sep=1.5pt] at (\p) {};
      \node[circle,fill=black,inner sep=6pt] at (Y3c) {};
      \node at (Y3a) [above] {\scriptsize $1/6$};
      \node at (Y3b) [below] {\scriptsize $1/6$};
      \node at (Y3c) [above=7pt] {\scriptsize $2/3$};
      \node at (0,-3) {\footnotesize State $\rho'_x$};
    \end{scope}
  \end{tikzpicture}
  \caption{Illustration of the Wasserstein transport between two discrete synaptic profiles. Mass is moved from $\rho_x$ (left) to $\rho'_x$ (right) along the edges of the graph. The intermediate configuration $\rho_x^t$ (center) shows partial transport along each branch (blue points), representing the geodesic interpolation in the Wasserstein space, with $t=0.25$.}
\end{figure}

Thus, the representation
\[
x_j \longmapsto \rho_{x_j} \in \mathcal P_2(Y)
\]
provides a geometrically coherent model of synaptic relative weights, in which plasticity can be described as a continuous trajectory in Wasserstein space, while biological observations correspond to projections onto the discrete set of leaves $Y_M$.
In particular, we can consider that along the geodesic $(\rho^t_{x_j})_{t\in[0,1]}$ between $\rho^0_{x_j}=\rho_{x_j}$ and $\rho^1_{x_j}=\rho_{x_j}'$, the network is in a mixed state: with probability $(1-t)$ one observes the original profile $\rho_{x_j}$, and with probability $t$ one observes the updated profile $\rho_{x_j}'$. Averaging over multiple observations yields
\[
\mathbb{E}\!\left[\rho^t_{x_j}\right]
= (1-t)\,\rho_{x_j} + t\,\rho_{x_j}'.
\]
Equivalently, since each distribution is supported on the leaves $Y_M$, the expected weights satisfy
\[
\mathbb{E}\!\left[p^t_{i,j}\right]
= (1-t)\,p_{i,j} + t\,p'_{i,j}.
\]
Thus, the expected synaptic profile follows a straight line segment in the $(M\!-\!1)$-dimensional simplex between the points
\[
p_j = (p_{1,j},\dots,p_{M,j})
\qquad\text{and}\qquad
p'_j = (p'_{1,j},\dots,p'_{M,j}).
\]
Remarkably, this affine evolution is entirely independent of the metric structure of the underlying graph $Y$: although the internal Wasserstein geodesic depends on the lengths of the edges, the observable mixed state produces a universal barycentric interpolation in the simplex.

\subsection{Geometry on the space of memory fields}

We now introduce a natural geometry on the space of memory fields $\mathcal{X}$. This will allow us to treat the evolution $t \mapsto \rho(t)$ as a trajectory in a metric space.

Given two memory fields $\rho, \rho' \in \mathcal{X}_\mu$, we define their distance by
\[
\mathcal{W}(\rho,\rho')
:=\left( \int_X W_2^2(\rho_x,\rho'_x) \, \mathrm{d}\mu(x) \right)^{1/2}.
\]
This is the standard product Wasserstein metric on the bundle $\mathcal{P}_2(Y)^X$, weighted by the reference measure $\mu$ on $X$ (see \parencite{villani,ags}). The space $(\mathcal{X}_\mu,\mathcal{W})$ is complete and separable whenever $(Y,d_Y)$ is complete and separable. This equips $\mathcal{X}_\mu$ with a canonical metric structure in which memory fields can be compared, interpolated, and evolved.

\begin{remark}[Domain of the metric.]
  Although each fibre $\rho_x$ belongs to $\mathcal P_2(Y)$, the product Wasserstein metric
  \[
  \mathcal W(\rho,\sigma)^2
  = \int_X W_2^2(\rho_x,\sigma_x)\,\mathrm d\mu(x)
  \]
  is finite only when the $\mu$-average of the second moments is finite. For this reason, the natural domain of the metric $\mathcal W$ is the subspace
  \[
  \mathcal X_\mu
  := \Big\{ \rho\in\mathcal X :
  \int_X \int_Y d_Y(y,y_0)^2\,\mathrm d\rho_x(y)\,\mathrm d\mu(x) < \infty \Big\},
  \]
  on which $\mathcal W$ is a well-defined finite metric. Working in $\mathcal X_\mu$ ensures tightness, compactness of discrete trajectories, and stability of the Wasserstein geometry under limits.
\end{remark}

\begin{remark}[Cognitive interpretation]
  Geometrically, $\mathcal{W}(\rho,\rho')$ measures the total amount of transport required to transform the entire memory field $\rho$ into $\rho'$, aggregating the fiberwise Wasserstein costs across $X$. Two global states are close when, for most units $x \in X$, the corresponding internal distributions encode similar information. Thus $\mathcal{W}$ measures the dissimilarity between two ``memory states'' in a geometrically meaningful way.
\end{remark}

Moreover, $\mu$-compatibility of memory fields is actually stable under Wasserstein geometry. More precisely, we have the following proposition:

\begin{proposition}[Geodesic stability of $\mu$-compatible fields]
  \label{prop:mu-compatible-geodesic}
  Assume $\rho^0,\rho^1\in\mathcal X_\mu$ are $\mu$-compatible fields, i.e.
  \[
  \int_X\int_Y d_Y(y,y_0)^2\,\mathrm d\rho^i_x(y)\,\mathrm d\mu(x)<\infty,
  \qquad i=0,1,
  \]
  for some (hence any) $y_0\in Y$.
  Let $(\rho^t)_{t\in[0,1]}$ be a constant-speed geodesic between $\rho^0$ and $\rho^1$ in $(\mathcal X,\mathcal W)$.
  Then, for every $t\in[0,1]$, the field $\rho^t$ is $\mu$-compatible.

  In particular, the set of $\mu$-compatible fields $\mathcal X_\mu$ is geodesically convex in $(\mathcal X,\mathcal W)$.
\end{proposition}

\begin{proof}
  By the product structure of $(\mathcal X,\mathcal W)$, a constant-speed geodesic $(\rho^t)_{t\in[0,1]}$ between $\rho^0$ and $\rho^1$ is given fiberwise by Wasserstein geodesics: for $\mu$-a.e.\ $x\in X$, $(\rho^t_x)_{t\in[0,1]}$ is a $W_2$-geodesic in $\mathcal P_2(Y)$ between $\rho^0_x$ and $\rho^1_x$.

  Fix $y_0\in Y$. In $\mathcal P_2(Y)$, along a $W_2$-geodesic $(\nu^t)_{t\in[0,1]}$ we have
  \[
  \int_Y d_Y(y,y_0)^2\,\mathrm d\nu^t(y)
  \le (1-t)\int_Y d_Y(y,y_0)^2\,\mathrm d\nu^0(y)
  + t\int_Y d_Y(y,y_0)^2\,\mathrm d\nu^1(y),
  \qquad t\in[0,1].
  \]
  Applying this with $\nu^t=\rho^t_x$ for $\mu$-a.e.\ $x$ and integrating with respect to $\mu$ yields
  \begin{align*}
    \int_X\int_Y d_Y(y,y_0)^2\,d\rho^t_x(y)\,\mathrm d\mu(x) \,
    &\,\le (1-t)\int_X\int_Y d_Y(y,y_0)^2\,\mathrm d\rho^0_x(y)\,\mathrm d\mu(x)\\
    &\qquad + t\int_X\int_Y d_Y(y,y_0)^2\,\mathrm d\rho^1_x(y)\,\mathrm d\mu(x)
  \end{align*}
  Since the right-hand side is finite by $\mu$-compatibility of $\rho^0$ and $\rho^1$, the left-hand side is finite for all $t\in[0,1]$, so each $\rho^t$ is $\mu$-compatible. This proves the claim.
\end{proof}

Between two $\mu$-compatible memory fields, it will be useful to consider a $\mu$-measurable family of optimal transport couplings between local (fiberwise) memory states. Hence the following lemma.

\begin{lemma}[Measurable selection of optimal couplings]
  \label{lem:couplings-selection}
  Let $\rho,\rho'\in\mathcal X_\mu$. For $\mu$-a.e.\ $x\in X$, let $\Pi(\rho_x,\rho'_x)$ denote the set of optimal couplings between $\rho_x$ and $\rho'_x$.
  Then, the set-valued map $x\mapsto \Pi(\rho_x,\rho'_x)$ admits a measurable selection.
\end{lemma}

\begin{proof}
  By standard optimal transport theory, this set is nonempty and closed in $\mathcal P_2(Y\times Y)$. Moreover, since $x\mapsto\rho_x$ and $x\mapsto\rho'_x$ are measurable (by definition of $\mathcal X_\mu \subset \mathcal X$), the set-valued map $x\mapsto \Pi(\rho_x,\rho'_x)$ admits a measurable selection by the Kuratowski--Ryll-Nardzewski theorem.
\end{proof}

\subsection{The context space}

We now introduce the notion of \emph{context}, which encodes the functional state of the system.
Let $(\Gamma, \mathrm m)$ be a measured space representing the \emph{internal activation space}. A \emph{context} is encoded by a $\mu$-measurable family of finite complex measures $\phi=\lbrace \phi_x\rbrace_{x\in X} \in \mathcal{M}(\Gamma)^X$, not necessarily normalized.

We consider the Banach space of complex measures
\[
  (\mathcal M(\Gamma),\|\cdot\|),
\]
where $\|\cdot\|$ is the total variation norm. We say that a context $\phi$ is \emph{$\mu$-compatible} if
\[
  \int_X \|\phi_x\|^2\,\mathrm d\mu(x)<+\infty,
\]
and we denote by $\mathcal C_\mu$ the set of $\mu$-compatible contexts i.e.
\[
 \mathcal C_\mu:=\left\lbrace
  \phi \in \mathcal{M}(\Gamma)^X \,\big|\,
  x\mapsto\phi_x \text{ is measurable},
  \int_X \|\phi_x\|^2\,\mathrm d\mu(x)<+\infty
  \right\rbrace.
\]
We endow $\mathcal C_\mu$ with the following distance
\[
  d_C(\phi, \phi'):=\left(\int_X \|\phi_x-\phi'_x\|^2\,\mathrm d\mu(x)\right)^{1/2},
  \qquad \phi,\phi'\in\mathcal C_\mu.
\]

\begin{remark}[Interpretation of the internal activation space and complex-valued contexts.]
  The internal activation space represents the state of the world, at least as perceived by the agent. When the world is perceived solely through neuronal activation, it is natural to take the internal activation space to coincide with the neuronal representation, i.e.\ $X$ and $Y$. This is analogous to a human perceiving the world only through sensory pathways translated into neural encodings, or to an artificial system encoding only what its sensors capture (visual signals, kinetic sensors etc.).

  In this view, $\|\phi_x\|$ can be seen as the ``intensity'' of context on $x$: it is not normalized. The complex-valued nature of contexts will reveal itself as natural in the case of generalized neural networks (Section~\ref{sec:networks}), where complex-valued synaptic weights allow convenient representations of phase synchronization and alignment mechanisms. Moreover, it endows the context space with a natural metric (via total variation norm), which provides a clean setting to define regularity assumptions relative to the contexts (Section~\ref{sec:continuous-limit}).

  In Section~\ref{sec:mirror-descent} we will see an example where $\Gamma=Y\times Y$, and in Section~\ref{sec:networks}, an example where $\Gamma=Y\times X$.
\end{remark}

\subsection{External representation space}

Let $(Z, d_Z)$ be a metric space, representing the external state space. $Z$ is the space where observables evolve. $Z$ is endowed with its borelian measure. A context $\phi$ is projected to $Z$ via a global signal (or signal) $S(\rho,\phi)$, where $S:\mathcal X_\mu \times \mathcal M(\Gamma)^X\to Z^X$ couples a memory field $\rho\in\mathcal X_\mu$ with the context.

\section{Hebbian--Wasserstein compatible energies}\label{sec:energies}

The intuition behind the framework is that Hebbian memory updates behave like a Wasserstein minimizing-movement dynamics. We therefore introduce a class of energies whose fiberwise structure and stability properties ensure that a JKO-type scheme is well posed.

\subsection{Hebbian energies: structural definition}
We interpret Hebbian principle~\parencite{hebb} as follows: each unit updates its local memory state based solely on its own activity and on a global co-activation signal summarising the collective configuration. In our formalism, the ``weight'' of $x'$ on $x$ is represented by the measure $\rho_x$ around $x'$. Thus the energy driving the update of $\rho_x$ should depend only on the pair $(\rho_x, S_x(\rho,\phi))$, where $S_x(\rho,\phi)$ encodes the global context.

\begin{definition}[Hebbian energy]\label{def:hebbian}
  An energy functional
  \[
  E : \mathcal X_\mu \times \mathcal M(\Gamma)^X \to \mathbb R\cup\{+\infty\}
  \]
  is \emph{Hebbian} if there exist:
  \begin{itemize}
    \item a metric space $(Z,d_Z)$, the \emph{external state space}, equiped with its Borel $\sigma$-algebra;
    \item a family of maps $S=\lbrace S_x \rbrace_{x\in X}$, the \emph{global signal}, with each
    \[
    S_x:\mathcal X_\mu \times \mathcal M(\Gamma)^X \to Z,
    \]
    such that for every $(\rho, \phi)\in\mathcal X_\mu \times \mathcal M(\Gamma)^X$,
    \[
    x\mapsto S_x(\rho, \phi)
    \]
    is Borel measurable as a map from $X$ to $Z$;
    \item a family of maps $\mathcal E = \lbrace \mathcal E_x \rbrace_{x\in X}$, the \emph{local energy density}, with each
    \[
    \mathcal E_x : \mathcal P_2(Y) \times Z \to \mathbb R \cup \{+\infty\},
    \]
    such that the map
    \[
    (x,\eta,z)\mapsto \mathcal E_x(\eta,z)
    \]
    is Borel measurable on $X\times\mathcal P_2(Y)\times Z$;
  \end{itemize}
  such that for all $(\rho,\phi)$:
  \[
  E(\rho,\phi)
  =\int_X \mathcal E_x\big(\rho_x,\,S_x(\rho,\phi)\big)\,\mathrm d\mu(x).
  \]
\end{definition}

This definition captures our structural Hebbian requirement: the energy decouples across fibres, and each fibre (i.e. each unit) interacts with the rest of the system only through the global signal.

\subsection{Sequential stability: analytic admissibility}

When the global signal is held fixed, each unit should respond in a stable and well-behaved manner to this signal. This reflects the sequential nature of synaptic plasticity, where updates occur locally and asynchronously under a slowly varying global influence.

\begin{definition}[Sequential stability]
  \label{def:seq-stab}
  A Hebbian energy $E$, with local energy density $\mathcal E$,
  is \emph{sequentially stable} if for every reference configuration $(\phi^{\mathrm{ref}}, \rho^{\mathrm{ref}}) \in \mathcal M(\Gamma)^X \times \mathcal{X}_\mu$, the \emph{frozen energy}, defined by
  \[
  E_{\phi^{\mathrm{ref}},\rho^{\mathrm{ref}}}(\rho)
  := \int_X \mathcal{E}_x\big(\rho_x, h_x\big)\, \mathrm{d}\mu(x),
  \qquad\text{where } h=S(\rho^{\mathrm{ref}},\phi^{\mathrm{ref}})\text{ is fixed},
  \]
  satisfies the AGS assumptions, namely:
  \begin{enumerate}
    \item \textbf{Properness:} $E_{\phi^{\mathrm{ref}},\rho^{\mathrm{ref}}}$ is not identically $+\infty$ and never takes the value $-\infty$.
    \item \textbf{Lower semicontinuity:} $E_{\phi^{\mathrm{ref}},\rho^{\mathrm{ref}}}$ is lower semicontinuous with respect to $\mathcal W$-convergence on $\mathcal X_\mu$.
    \item \textbf{Geodesic convexity:} $E_{\phi^{\mathrm{ref}},\rho^{\mathrm{ref}}}$ is $\lambda$-geodesically convex in $(\mathcal X_\mu,\mathcal W)$, $\lambda>-\infty$.
  \end{enumerate}
\end{definition}

\emph{Notations.} At a fixed reference configuration $(\phi^{\mathrm{ref}}, \rho^{\mathrm{ref}})$, we denote by $S_{\phi^{\mathrm{ref}}, \rho^{\mathrm{ref}}}:=S(\rho^{\mathrm{ref}},\phi^{\mathrm{ref}})$ the \emph{frozen global signal} (or simply \emph{frozen signal}) and by
\[
E_{\phi^{\mathrm{ref}},\rho^{\mathrm{ref}}}(\rho)
:= \int_X \mathcal{E}_x\big(\rho_x, S_{\phi^{\mathrm{ref}},\rho^{\mathrm{ref}},\, x}\big)\, \mathrm{d}\mu(x)
\]
the \emph{frozen energy}.

The AGS assumptions~\parencite{ags} guarantee the well-posedness of minimizing-movement schemes with the frozen energy as we will see later. Nevertheless, in practice, the definition of a Hebbian energy is naturally given by defining its local energy density. In such a case, establishing sequential stability is more natural at the level of the local energy itself. Hence the following definition.

\begin{definition}[Local sequential stability]
  \label{def:loc-seq-stab}
  A Hebbian energy $E$ is \emph{locally sequentially stable} if for every reference configuration $(\phi^{\mathrm{ref}}, \rho^{\mathrm{ref}}) \in \mathcal M(\Gamma)^X \times \mathcal{X}_\mu$, for $\mu$-a.e.\ $x\in X$, when fixing the global signal $z=S_x(\rho^{\mathrm{ref}},\phi^{\mathrm{ref}})\in Z$, the local energy density
  \[
  \mathcal E_x(\cdot,z) : \mathcal P_2(Y) \to \mathbb R \cup \{+\infty\}
  \]
  satisfies:
  \begin{enumerate}
    \item \textbf{Properness:} $\mathcal E_x(\cdot,z)$ is not identically $+\infty$ and never takes the value $-\infty$.
    \item \textbf{Lower semicontinuity:} $\mathcal E_x(\cdot,z)$ is lower semicontinuous with respect to $W_2$-convergence on $\mathcal P_2(Y)$.
    \item \textbf{Geodesic convexity:} $\mathcal E_x(\cdot,z)$ is $\lambda$-geodesically convex in $(\mathcal P_2(Y),W_2)$, $\lambda>-\infty$.
    \item \textbf{Lower bound:} there exists a function $c:X\times Z\to \mathbb R \cup \{+\infty\}$ such that for all $z\in Z$:
    \[
    c(\cdot,z) \in \mathrm L^1(X,\mu),
    \quad\text{and}\quad
    \forall x\in X, \forall \eta\in \mathcal P_2(Y), \quad \mathcal E_x(\eta,z) \ge c(x,z).
    \]
    \item \textbf{Existence of a finite-energy field:} for every reference configuration $(\phi^{\mathrm{ref}}, \rho^{\mathrm{ref}})$, there exists at least one $\rho\in\mathcal X_\mu$ such that
    \[
    \int_X \mathcal E_x(\rho_x,S_x(\rho^{\mathrm{ref}},\phi^{\mathrm{ref}}))\,\mathrm d\mu(x)
    <+\infty.
    \]
  \end{enumerate}
\end{definition}

The lower bound property and the existence of a finite-energy field property are key to ensure properness of the total energy, so that we can have the following proposition:

\begin{proposition}[Local sequential stability implies sequential stability]
  \label{prop:loc-seq-stab}
  A locally sequentially stable Hebbian energy is sequentially stable.
\end{proposition}

\begin{proof}
  Fix a reference configuration $(\phi^{\mathrm{ref}},\rho^{\mathrm{ref}})\in\mathcal M(\Gamma)^X\times\mathcal X_\mu$ and set $h:=S(\rho^{\mathrm{ref}},\phi^{\mathrm{ref}})\in Z^X$.

  Define the frozen energy
  \[
  E_{\phi^{\mathrm{ref}},\rho^{\mathrm{ref}}}(\rho)
  := \int_X \mathcal E_x(\rho_x,h_x)\,\mathrm d\mu(x),
  \qquad \rho\in\mathcal X_\mu.
  \]

  \emph{Properness.}
  The lower bound implies
  \[
  E_{\phi^{\mathrm{ref}},\rho^{\mathrm{ref}}}(\rho)\ge \int_X c(x,h_x)\,\mathrm d\mu(x)>-\infty
  \]
  for all $\rho$, so the functional never takes the value $-\infty$. Moreover, the existence of a finite-energy field property ensures that the frozen energy is not identically $+\infty$. Hence the properness of the frozen energy.

  \emph{Lower semicontinuity.}
  Let $(\rho^n)_{n\in\mathbb N}\subset\mathcal X_\mu$ converge to $\rho$ in $(\mathcal X_\mu,\mathcal W)$.
  Up to a subsequence, we may assume $\rho^n_x\to\rho_x$ in $W_2$ for $\mu$-a.e.\ $x$. By fiberwise lower semicontinuity,
  \[
  \mathcal E_x(\rho_x,h_x)
  \le \liminf_{n\to\infty} \mathcal E_x(\rho^n_x,h_x)
  \qquad\text{for $\mu$-a.e.\ }x.
  \]
  Subtracting the integrable lower bound $c(x,h_x)$ and applying Fatou’s lemma to the nonnegative functions $\mathcal E_x(\rho^n_x,h_x)-c(x,h_x)$ yields
  \[
  \int_X \mathcal E_x(\rho_x,h_x)\,\mathrm d\mu(x)
  \;\le\;
  \liminf_{n\to\infty} \int_X \mathcal E_x(\rho^n_x,h_x)\,\mathrm d\mu(x),
  \]
  that is,
  \[
  E_{\phi^{\mathrm{ref}},\rho^{\mathrm{ref}}}(\rho)
  \;\le\;
  \liminf_{n\to\infty} E_{\phi^{\mathrm{ref}},\rho^{\mathrm{ref}}}(\rho^n),
  \]
  so $E_{\phi^{\mathrm{ref}},\rho^{\mathrm{ref}}}$ is $\mathcal W$–lower semicontinuous on $\mathcal X_\mu$.

  \emph{Geodesic convexity.}
  By the product structure of $(\mathcal X,\mathcal W)$, a constant-speed geodesic $(\rho^t)_{t\in[0,1]}$ between $\rho^0$ and $\rho^1$ is given fiberwise by Wasserstein geodesics:
  for $\mu$-a.e.\ $x\in X$, $(\rho^t_x)_{t\in[0,1]}$ is a $W_2$-geodesic in $\mathcal P_2(Y)$ between $\rho^0_x$ and $\rho^1_x$.
  By fiberwise $\lambda$-geodesic convexity ($\lambda>-\infty$), for every $t\in[0,1]$ and $\mu$-a.e.\ $x$,
  \[
  \mathcal E_x(\rho^t_x,h_x)
  \;\le\; (1-t)\,\mathcal E_x(\rho^0_x,h_x) + t\,\mathcal E_x(\rho^1_x,h_x)
  - \frac{\lambda}{2}t(1-t)W_2^2(\rho_x^0, \rho_x^1).
  \]
  Integrating over $X$ gives
  \[
  E_{\phi^{\mathrm{ref}},\rho^{\mathrm{ref}}}(\rho^t)
  \;\le\;
  (1-t)\,E_{\phi^{\mathrm{ref}},\rho^{\mathrm{ref}}}(\rho^0)
  + t\,E_{\phi^{\mathrm{ref}},\rho^{\mathrm{ref}}}(\rho^1)
  - \frac{\lambda}{2}t(1-t)\mathcal W^2(\rho^0, \rho^1),
  \]
  showing $\lambda$-geodesic convexity in $(\mathcal X_\mu,\mathcal W)$ ($\lambda>-\infty$).

  All AGS conditions are therefore satisfied, and the frozen energy is sequentially stable.
\end{proof}

\subsection{Tan--HWG energy class}

\emph{Notations.} To avoid confusion with existing vocabulary, we will regularly use the tag ``Tan--HWG'', standing for Hebbian--Wasserstein--Geometry, to denote specific properties of our framework.

This leads to the definition of the class of energies that will be considered in our framework.

\begin{definition}[Tan--HWG energy]
  \label{def:tan-hwg-en}
  A Hebbian energy satisfying sequential stability is called a \emph{Tan--HWG energy}.
\end{definition}

Section~\ref{sec:quadratic} provides examples of Tan--HWG energies. Beyond biological motivations (i.e.\ sequential nature of the updates), freezing the signal at each step improves computational practicality, while preserving enough stability for a well-posed dynamics. The next section details and justifies the update mechanism.

\section{Discrete Tan--HWG dynamics}\label{sec:dynamics}

We now construct the dynamics associated with Tan--HWG energies. The key idea is that at each time step, the global signal is frozen, and each fibre performs a Wasserstein proximal step with respect to the frozen energy.

\subsection{The Tan--HWG minimizing-movement scheme}
\label{sec:jko-tan-hwg}

Given a Tan--HWG energy and an initial configuration $\rho^0 \in \mathcal{X}_\mu$, the discrete-time Tan--HWG evolution is defined by the JKO-type scheme \parencite{jko1998,ags}
\begin{equation}
  \rho^{n+1}_\tau \in \argmin_{\rho\, \in\, \mathcal{X}_\mu}
  \left\{
  E_{\phi^n,\rho^n_\tau}(\rho)
  + \frac{1}{2\tau} \int_X W_2^2(\rho_x,\rho^n_{\tau,x})\, \mathrm{d}\mu(x)
  \right\}.
  \label{eq:jko-tan-hwg}
\end{equation}
for a timestep $\tau>0$.

The sequential stability conditions of Definition~\ref{def:seq-stab} ensure that, at each step, the minimisation problem is well-posed. In particular, every step produces a stable Hebbian update of the memory field.

This motivates the definition of the $\tau$-Tan--HWG update operator
\[
T_{\tau}^E(\rho \,|\, \phi)
:= \argmin_{\rho'\, \in\, \mathcal{X}_\mu}
\left\{
E_{\phi,\rho}(\rho')
+ \frac{1}{2\tau} \int_X W_2^2(\rho'_x,\rho_x)\, \mathrm{d}\mu(x)
\right\}.
\]
As is standard in minimizing-movement schemes, the operator may be multivalued. We write $T_{\tau}^E(\rho \,|\, \phi)$ to emphasize that the Tan--HWG update acts on the memory field $\rho$ while the context $\phi$ is frozen and plays the role of an external parameter.

\subsection{Well-posedness of Tan–HWG minimizing-movement scheme}

The following Theorem formally states that Tan--HWG energies---that is, Hebbian and sequentially stable energies---generate a well-posed Tan--HWG scheme, with an energy descent inequality (EDI).

\begin{theorem}[Discrete Tan--HWG dynamics]
  \label{thm:tan-hwg}
  Let $\tau>0$ be a fixed time step. Let $E$ be a Tan--HWG energy with local energy density $\mathcal E$ and global signal $S$. Let $T_{\tau}^E$ be the $\tau$-Tan--HWG update operator defined by
  \[
  T_{\tau}^E(\rho\,|\,\phi)
  := \argmin_{\rho'\, \in\, \mathcal{X}_\mu}
  \left\{
  E_{\phi,\rho}(\rho')
  + \frac{1}{2\tau} \mathcal{W}^2(\rho',\rho)
  \right\}.
  \]
  Let $(\phi^n_\tau)_{n\ge 0}\in\left(\mathcal M(\Gamma)^X\right)^{\mathbb N}$ be a set of contexts.

  Then for every $\rho^0_\tau \in \mathcal{X}_\mu$, the sequence
  \[
  \rho^{n+1}_\tau \in T_{\tau}^E(\rho^n_\tau \,|\, \phi^n_\tau), \qquad \text{for all } n \ge 0,
  \]
  is well-defined. Moreover:
  \begin{itemize}
    \item the update is fiberwise: for $\mu$-a.e.\ $x$, $\rho^{n+1}_{\tau,x}$ solves an independent convex minimisation problem;
    \item the scheme satisfies the energy descent inequality (EDI)
    \[
    E_{\phi^n_\tau,\rho^n_\tau}(\rho^{n+1}_\tau) + \frac{1}{2\tau}\mathcal W^2(\rho^{n+1}_\tau,\rho^n_\tau)
    \le
    E_{\phi^n_\tau,\rho^n_\tau}(\rho^n_\tau).
    \]
  \end{itemize}
  Thus the Tan--HWG scheme defines a well-posed discrete-time dynamics on $(\mathcal X_\mu, \mathcal W)$, in the sense of a minimizing-movement scheme.
\end{theorem}

\begin{proof}
  Fix $\rho^0_\tau \in \mathcal{X}_\mu$ and $n \ge 0$. For $\rho\in\mathcal X_\mu$, define
  \[
  \mathcal{F}_n(\rho)
  := E_{\phi^n_\tau,\rho^n_\tau}(\rho)
  + \frac{1}{2\tau} \mathcal{W}^2(\rho,\rho^n_\tau)
  = \int_X \mathcal F_{n,x}(\rho_x)\, \mathrm d\mu(x),
  \]
  where
  \[
  \mathcal F_{n,x}(\eta)
  := \mathcal E_x(\eta, S_{\phi^n_\tau,\rho^n_\tau,\,x}) + \frac{1}{2\tau} W_2^2(\eta,\rho^n_{\tau,x}).
  \]

  \emph{Step 1: existence of the update.}
  By sequential stability, the frozen energy map
  \[
  \rho \mapsto E_{\phi^n_\tau,\rho^n_\tau}(\rho)
  \]
  satisfies the AGS assumptions on $(\mathcal X_\mu,\mathcal W)$. Hence each $\mathcal F_n$ admits at least one minimizer $\rho^{n+1}_\tau\in\mathcal X_\mu$. Hence the sequence $\rho^{n+1}_\tau \in T_{\tau}^E(\rho^n_\tau \,|\, \phi^n_\tau)$ for all $n \ge 0$ is well-defined.

  \emph{Step 2: fiberwise reduction.}
  Since $\mathcal F_n$ is an integral of decoupled fiberwise functionals, any minimizer $\rho$ of $\mathcal F_n$ must satisfy, for $\mu$-a.e.\ $x$,
  \[
  \rho_x \in \argmin_{\eta\in\mathcal P_2(Y)} \mathcal F_{n,x}(\eta).
  \]
  Otherwise, replacing $\rho_x$ on a set of positive $\mu$-measure by a measurable selection of fiberwise minimizers would strictly decrease $\mathcal F_n$, while preserving $\mu$--compatibility, a contradiction.

  \emph{Step 3: energy descent inequality.}
  Since $\rho^{n+1}_\tau$ minimises $\mathcal{F}_n$, we have
  \[
  \mathcal{F}_n(\rho^{n+1}_\tau) \le \mathcal{F}_n(\rho^n_\tau),
  \]
  which expands as
  \[
  E_{\phi^n_\tau,\rho^n_\tau}(\rho^{n+1}_\tau)
  + \frac{1}{2\tau} \mathcal{W}^2(\rho^{n+1}_\tau,\rho^n_\tau)
  \le
  E_{\phi^n_\tau,\rho^n_\tau}(\rho^n_\tau).
  \]
  This is exactly the stated energy descent inequality.

  Thus the Tan--HWG update is well defined, with fiberwise minimisation and EDI, and the discrete dynamics is well posed.
\end{proof}

\begin{remark}[First variation]
  For a frozen signal $S_{\phi, \rho}$, the frozen energy admits a metric subdifferential which decomposes fiberwise:
  \[
  \xi\in\partial E_{\phi,\rho}(\rho')
  \quad\Longleftrightarrow\quad
  \xi_x\in\partial_1\mathcal E_x(\rho'_x,S_{\phi, \rho, x})
  \quad\text{for $\mu$-a.e.\ }x.
  \]
  Whenever $E$ is a Tan--HWG energy, we write
  \[
  \frac{\delta E}{\delta\rho_x}(\rho,\phi)
  \in \partial_1\mathcal E_x\big(\rho_x,S_x(\rho,\phi)\big)
  \]
  for any choice of fiberwise subgradient. The formal identity
  \[
  \frac{\delta E}{\delta\rho_x}(\rho,\phi)
  = \frac{\partial\mathcal E_x}{\partial\rho}\big(\rho_x,S_x(\rho,\phi)\big)
  \]
  is therefore understood in this subdifferential sense, with the global signal frozen.
  In other words, the first variation of $E$ is defined relative to the subdifferential of the frozen energy.
\end{remark}

\begin{remark}[Global signal and continuous Tan--HWG dynamics]
  Since the global signal is frozen during each fixed time step $\tau>0$, no additional regularity assumption on the signal $S$ is required for the well-posedness of the discrete Tan--HWG dynamics.

  In contrast, the identification of a vanishing time-step limit $\tau\to0$ as a Wasserstein gradient flow requires further assumptions, in particular on the regularity, stability, and coercivity properties of the signal $S$ and of the associated energy.
  Even for fixed context $\phi$, the energy
  \[
  E(\rho,\phi) = \int_X \mathcal E_x\big(\rho_x, S_x(\rho, \phi)\big)\,\mathrm{d}\mu(x)
  \]
  must satisfy the structural conditions of the AGS minimizing-movement theory~\parencite{ags}.

  In this work, we focus on discrete Tan--HWG dynamics, which are both mathematically well-posed and computationally practical (e.g.\ see Section~\ref{sec:closed-form}).
  Nevertheless, Section~\ref{sec:continuous-limit} will explore a continuous-time limit of the Tan--HWG scheme under the special case of a quadratic energy with Lipschitz global signal and compatible set of contexts.
\end{remark}

\subsection{Brenier and metric graph realizations of Tan--HWG updates}

In the Euclidean case, the fiberwise JKO update admits a classical optimal-transport representation. In the special case $Y=\mathbb R^d$ and the $\rho^n_{\tau,x}$ admit densities, we can deduce a Brenier realization of the update:

\begin{corollary}[Brenier realization of Tan--HWG updates]
  \label{cor:brenier}
  Assume the setting of the $\tau$-Tan--HWG scheme, with $Y = \mathbb{R}^d$. Suppose moreover that for each $n\ge 0$ and $\mu$-a.e.\ $x\in X$, the measures $\rho^n_{\tau,x}$ and $\rho^{n+1}_{\tau,x}$ admit densities with respect to the Lebesgue measure on $\mathbb{R}^d$.

  Then for $\mu$-a.e.\ $x\in X$ and each $n\ge 0$:
  \begin{enumerate}
    \item There exists a convex function $\Psi_x^n : \mathbb{R}^d \to \mathbb{R}$ such that the optimal transport map
    \[
    T_x^n := \nabla \Psi_x^n
    \]
    pushes $\rho^n_{\tau,x}$ forward to $\rho^{n+1}_{\tau,x}$, i.e.
    \[
    \rho^{n+1}_{\tau,x} = T_{x,\#}^n \rho^n_{\tau,x}.
    \]
    \item The Wasserstein geodesic between $\rho^n_{\tau,x}$ and $\rho^{n+1}_{\tau,x}$ is given by
    \[
    \rho^n_{\tau,x}(t)
    := \big((1-t)\mathrm{Id} + t T_x^n\big)_\# \rho^n_{\tau,x},
    \quad t\in[0,1],
    \]
    and satisfies the continuity equation
    \[
    \partial_t \rho^n_{\tau,x}(t) + \nabla\cdot\big(\rho^n_{\tau,x}(t)\, v_x^n(t)\big) = 0,
    \]
    for a suitable velocity field $v_x^n(t)$ derived from $T_x^n$.
  \end{enumerate}
  In particular, each discrete Tan--HWG update $\rho^n_\tau \mapsto \rho^{n+1}_\tau$ admits a realization as a Brenier optimal transport in each fibre, and the Hebbian dynamics can be interpreted as a potential-driven mass flow in the internal space.
\end{corollary}

\begin{proof}
  Fix $n\ge 0$ and $x\in X$ such that $\rho^n_{\tau,x}$ and $\rho^{n+1}_{\tau,x}$ admit densities with respect to the Lebesgue measure on $\mathbb{R}^d$.
  Since $\rho^n_{\tau,x} \in \mathcal{P}_2(\mathbb{R}^d)$ is absolutely continuous, Brenier's theorem applies and yields the claims.
  The conclusion follows for $\mu$-a.e.\ $x$ and all $n$.
\end{proof}

This is the standard McCann interpolation. Additional mathematical details on fixed points and stability of Tan--HWG dynamics are provided in Appendix~\ref{sec:fixed_points}.

While the Brenier–McCann representation provides a natural description of Tan--HWG updates in Euclidean internal spaces, the applications developed in this work rely on internal representations whose geometry is non-Euclidean. In particular, several of our examples use internal spaces $Y$ that are geodesic metric graphs. In this setting, optimal transport remains well defined, but the Brenier map is replaced by a geodesic interpolation along the branches of the graph.

\begin{corollary}[Geodesic realization of Tan--HWG updates on metric graphs]
  \label{cor:graph}
  Assume the setting of the $\tau$-Tan--HWG scheme, and let $Y$ be a compact geodesic metric graph endowed with its shortest-path distance. For each $n\ge 0$ and $\mu$-a.e.\ $x\in X$, let $\rho^n_{\tau,x},\rho^{n+1}_{\tau,x}\in\mathcal P_2(Y)$ be the fiberwise minimizers produced by the Tan--HWG update.

  Then for $\mu$-a.e.\ $x\in X$ and each $n\ge 0$:
  \begin{enumerate}
    \item There exists an optimal transport plan
    \[
      \pi_x^n \in \mathrm{Opt}\big(\rho^n_{\tau,x},\rho^{n+1}_{\tau,x}\big)
    \]
    for the quadratic cost on $Y$.
    \item The constant-speed Wasserstein geodesic between $\rho^n_{\tau,x}$ and $\rho^{n+1}_{\tau,x}$ is given by
    \[
      \rho^n_{\tau,x}(t)
      := (e_t)_\# \pi_x^n,
      \qquad t\in[0,1],
    \]
    where $e_t : \mathrm{Geo}(Y)\to Y$ evaluates a geodesic at time $t$.
    \item This interpolation satisfies the continuity equation on the graph,
    \[
      \partial_t \rho^n_{\tau,x}(t)
      + \nabla_{\!Y}\cdot\big(\rho^n_{\tau,x}(t)\, v_x^n(t)\big)=0,
    \]
    where $v_x^n(t)$ is the metric velocity field induced by the geodesics in the support of $\pi_x^n$.
  \end{enumerate}

  In particular, each discrete Tan--HWG update admits a realization as a geodesic mass flow on the metric graph $Y$, and the Hebbian dynamics can be interpreted as a redistribution of mass along shortest paths in the internal space.
\end{corollary}

\begin{proof}
  Since $Y$ is a compact geodesic metric space, $(\mathcal P_2(Y),W_2)$ is geodesic and every pair of measures admits at least one optimal transport plan for the quadratic cost. Let $\pi_x^n$ be such a plan between $\rho^n_{\tau,x}$ and $\rho^{n+1}_{\tau,x}$.

  The geodesic interpolation
  \[
    \rho^n_{\tau,x}(t) := (e_t)_\# \pi_x^n
  \]
  is the standard McCann interpolation on geodesic spaces. It is well known that such interpolations satisfy a continuity equation with a velocity field given by the metric derivative of the underlying geodesics. The conclusion follows for $\mu$-a.e.\ $x$.
\end{proof}

\begin{remark}[Interpretation]
  The theorem guarantees the existence of the discrete sequence $(\rho^n)_n$, but it does not describe the mechanism by which one passes from $\rho^n$ to $\rho^{n+1}$. The corollary shows that this update is not a mere jump between two abstract measures: it can always be realized as a continuous displacement of mass along a Wasserstein geodesic. This continuous interpolation is essential for interpreting the Tan--HWG dynamics as an internal evolution of mass on the graph $Y$, rather than as a purely discrete update rule.
\end{remark}

\begin{remark}[Euclidean vs.\ graph geometry]
  In the Euclidean case $Y=\mathbb{R}^d$, optimal transport between absolutely continuous measures is induced by a unique map $\nabla\Psi$ for a convex potential $\Psi$ (Brenier's theorem). The Wasserstein geodesic is therefore obtained by pushing mass along straight lines in $\mathbb{R}^d$.

  In contrast, when $Y$ is a geodesic metric graph, optimal transport is generally not induced by a map but by an optimal plan supported on geodesics of the graph. Mass travels along shortest paths, which may branch or merge, and the Wasserstein geodesic is obtained by interpolating these paths. Thus, while the Euclidean case yields a ``potential flow'', the graph case yields a ``geodesic redistribution of mass'' along the combinatorial structure of $Y$.
\end{remark}

\paragraph{Mass on the edges and motivation for a leaf-level representation.}
In the Tan--HWG scheme on a metric graph $Y$, each fiberwise update $\rho^n_x \mapsto \rho^{n+1}_x$ is obtained by minimizing over the whole space $\mathcal P_2(Y)$. In particular, there is no reason for the iterates $(\rho^n_x)_n$ to remain supported on the leaves of $Y$: whenever transporting mass through the interior edges reduces the total cost, the optimal states $\rho^n_x$ may (and typically will) place positive mass on the segments of the graph.
By optimality of transport plans between discrete distributions, each state remains a discrete measure, but the \emph{movement} of mass occurs continuously along edges.
This phenomenon is intrinsic to the Wasserstein geometry on $Y$ and reflects the fact that the Tan--HWG update corresponds to a genuine displacement of mass in the internal space, rather than a purely discrete jump between leaf configurations.

For this reason, it is natural to introduce a \emph{leaf-level representation} of the dynamics, obtained by projecting the local memory states onto distributions on the set of leaves. The resulting projected evolution lives in the simplex $\mathcal P(Y_M)$ and provides a convenient observable description of the model. The next section formalises this idea in a general setting, leading to the geodesic projection principle in Principle~\ref{def:projection}.

\begin{remark}
  We write $\mathcal P_2(Y)$ for the $W_2$–Wasserstein space over $Y$, and $\mathcal P(Y_M)$ for the Euclidean simplex of probability measures over the finite observable space $Y_M$.
\end{remark}

\section{Extension of the dynamics to finite internal support systems}\label{sec:projection}

\paragraph{Unstability of Tan--HWG dynamics on finite support.}
Finite spaces $Y_M=\{y_i, \dots, y_M\}$ cannot be admissible as internal state spaces since they are not geodesic. A natural way to circumvent this is to lift $Y_M$ into an admissible space $Y$ i.e. $Y_M\subset Y$.
In this case, consider a memory set $(\rho^n)_{n\ge 0}$ following a Tan--HWG minimizing-movement scheme. Assume a state $\rho^n$ has all its supports on $Y_M$ i.e. for $\mu$-a.e.\ $x$, $\rho_x^n$ is a probability distribution with support exclusively on $Y_M$. In the general setting, nothing prevents the following states $(\rho^{n'})_{n' \ge n}$ to have supports beyond $Y_M$: the induced distribution dynamics is not stable on the finite support.

\subsection{Finite internal support systems}
We start by generalizing a simplex representation for distributions that have a finite support on an arbitrary internal state spaces $Y$.

\begin{definition}
  \label{def:finite-system}
  We say that a neural system is a finite neural system with $M$ internal degrees of freedom on $Y$, if, for each $x\in X$, its local memory state can be represented as a finitely supported probability measure
  \[
  \rho_x = \sum_{i=1}^M w_{i,x}\,\delta_{y_i},
  \]
  for some fixed family $\lbrace y_i \rbrace_{i\in \llbracket1,M\rrbracket}\subset Y$ and weights $w_x=(w_{1,x}, \dots w_{M,x})\in\Delta^{M-1}$.
\end{definition}

In such settings, memory fields $\rho$ admits a natural simplicial representation $w=\lbrace w_x\rbrace_{x\in X}$, with
\[
w_x=(w_{1,x}, \dots w_{M,x})\in\Delta^{M-1}.
\]

\begin{remark}[Intuition]
  Assume a finite neural system evolves according to a discrete $\tau$-Tan--HWG minimizing-movement scheme, where each $\rho^n$ admits a representation on the simplex. For a fiber $x$, the simplicial representation of the trajectory $(\rho^n_x)_{n\ge 0}$ would be a discontinuous set of points on the simplex (due to the finite timestep $\tau$ of the scheme). A continuous simplicial curve representation occurs when a continuous limit curve of the trajectory exists at the limit $\tau\to 0$. Two challenges arise:
  \begin{enumerate}
    \item first, representing $\rho^n$ on the simplex is not trivial when $\rho^n$ has not all its support on $Y_M$, which is generally the case;
    \item second, the existence of a limit curve when $\tau\to 0$ for a Tan--HWG minimizing-movement scheme is not guaranteed.
  \end{enumerate}

  The geodesic projection principle introduced below provides a way to overcome the first difficulty by defining a canonical projection of local memory states onto measures on a finite set.
\end{remark}

We present a projection principle based on geodesics in the next section.

\subsection{Geodesic projection principle}
\label{sec:geodesic-projection-principle}

We first define the notion of observable space and observable states on this space to represent accessible information about the internal state space $Y$.

\begin{definition}[Observable space and observable states]
  \label{def:observable}
  An \emph{observable space} is a finite (or countable) set of internal states $\mathcal S\subset Y$. An \emph{observable state} relative to $\mathcal S$ is a probability measure $\hat\rho \in \mathcal P(\mathcal S)$, that is, a probability distribution over the elements of $\mathcal S$. Such a state represents the coarse, externally accessible information about an internal distribution on $Y$.

  We denote by $\mathcal P_2(\mathcal S)$ the space of observable states relative to $\mathcal S$ when such observables are lifted to $\mathcal P_2(Y)$.
\end{definition}

This represents the idea that an agent has only access to a limited set of internal states. A neural network for instance, cannot ``see'' synaptic weights that are not tied to an actual synapse. They can only be influenced by accessible weights. When a probability measure in $\mathcal P_2(Y)$ (i.e. a synaptic strength profile) attributes weights outside $\mathcal S$, those weights remain inaccessible to the agent. A probability measure in $\mathcal P_2(Y)$ has to be projected into a probability measure in $\mathcal P(\mathcal S)$ (a simplex).

\begin{principle}[Stochastic geodesic projection principle]
  \label{def:projection}
  Let $\hat\rho,\hat\nu\in\mathcal P(\mathcal S)$ be two observable states relative to the observable space $\mathcal S$, and let $(\rho^t)_{t\in[0,1]}$ denote the constant-speed Wasserstein geodesic between their internal lifts $\rho,\nu\in\mathcal P_2(Y)$.

  For each $t\in[0,1]$, we define a random observable state $\hat\rho^t$ taking values in $\{\hat\rho,\hat\nu\}$ by
  \[
    \mathbb P(\hat\rho^t=\hat\rho)=1-t,
    \qquad
    \mathbb P(\hat\rho^t=\hat\nu)=t.
  \]
  The \emph{stochastic geodesic projection} of the local memory state $\rho^t$ onto the observable space $\mathcal S$ is then defined as the expectation
  \[
    \Pi_{\hat\rho\to\hat\nu}(\rho^t)
    := \mathbb E[\hat\rho^t].
  \]
\end{principle}

This is a stochastic projection. It expresses the fact that a mass particle travelling along the internal geodesic from $\rho$ to $\nu$ is observed at geodesic parameter $t$ as belonging to the origin state with probability $1-t$ and to the target state with probability $t$.

\begin{remark}[Affine form]
  Since $\hat\rho^t$ takes only the two values $\hat\rho$ and $\hat\nu$, its expectation is simply the convex combination
  \[
    \Pi_{\hat\rho\to\hat\nu}(\rho^t)
    = (1-t)\,\hat\rho + t\,\hat\nu.
  \]
  Thus the linear interpolation between observable states arises as the expectation of a binary random observation of the internal geodesic state.
\end{remark}

This projection acts as a bridge between the internal Wasserstein geodesic structure on $\mathcal P_2(Y)$ and the linear geodesic structure of the simplex $\mathcal P(\mathcal S)$ endowed with the Euclidean metric. In particular, it is the only admissible map that sends the internal constant-speed geodesic $(\rho^t)_{t\in[0,1]}$ to a constant-speed geodesic in $\mathcal P(\mathcal S)$. Indeed, since geodesics in the Euclidean simplex are exactly affine segments, any geodesicity-preserving projection must coincide with the linear interpolation $t\mapsto (1-t)\hat\rho + t\hat\nu$.

\begin{remark}[Relativity of the stochastic projection]
  The stochastic geodesic projection is not a global map $\mathcal P_2(Y)\to\mathcal P(\mathcal S)$. It is defined relatively to a pair of observable states $(\hat\rho,\hat\nu)$, via the Wasserstein geodesic between their internal lifts $(\rho,\nu)$. In other words, the projection is contextualized by $(\hat\rho,\hat\nu)$, which can be seen as a global signal (e.g.\ $S(\cdot, \phi):=\phi$ with context $\phi:=(\hat \rho, \hat \nu)$).
\end{remark}

\begin{remark}[Extension to combinatorial graphs]
  Any finite combinatorial graph $G=(V,E)$ can be canonically thickened into a metric graph $Y$ by replacing each edge $(i,j)\in E$ with an interval of length $\ell_{ij}>0$. This construction endows $Y$ with a geodesic metric, and therefore with a well-defined Wasserstein space $(\mathcal P_2(Y),W_2)$ on which the Tan--HWG update is formulated.

  The Tan--HWG minimization is then performed over $\mathcal P_2(Y)$, and the resulting iterates $(\rho^n)_{n\ge0}$ naturally live on the whole metric graph, not only on the vertices.

  In Section~\ref{sec:isotropic-en}, we will see an illustration where the signal plays the role of a target distribution at each step, stabilized by its frozen nature. This will lead to an update $\rho^{n+1}$ located on the geodesic between the previous state $\rho^n$ and the target frozen signal distribution $\nu^n$. The stochastic geodesic projection then gives a natural representation of observable weights on the combinatorial graph, given by a recursive relation detailed later.
\end{remark}

\begin{remark}[Perspectives]
  Beyond its conceptual interest, the framework developed here may have applications in areas where probability distributions supported on a finite set evolve dynamically. A rapidly growing line of work in machine learning studies continuous-time flows on the categorical simplex for generative modeling, distillation, or accelerated sampling \parencite{lipman2023flow,gat2024discreteflowmatching,liu2024sfm,hoogeboom2021argmax}. These approaches typically operate directly in the simplex equipped with a Riemannian or information-geometric structure, without modeling any latent geometry underlying the discrete states. In contrast, our construction provides a principled way to introduce such a latent geometry: a distribution evolves internally on a metric graph according to a Wasserstein gradient-flow scheme, and the observable categorical distribution arises through stochastic geodesic projection.

  This viewpoint could be useful in several ways. First, it offers a geometrically structured latent space for categorical flows, potentially improving interpretability or stability. Second, it provides a natural regularisation mechanism when the discrete states possess semantic or hierarchical relations representable as a graph. Finally, it suggests a new class of generative models in which the internal dynamics is Wasserstein-geodesic while the observable dynamics remains linear on the simplex. Investigating these directions lies beyond the scope of the present work but constitutes a promising avenue for future research.
\end{remark}

The next section will show how, in Tan--HWG minimizing-movement schemes, the signal plays the role of a ``target'' measure at each step, placing the update on the geodesic between the previous state and the target frozen signal.

\section{Quadratic energies and constant contraction update rule}\label{sec:quadratic}

In this section, we consider the case where the external state space is $Z=\mathcal{P}_2(Y)$, so that both local states and the global signal fiber components live in the same Wasserstein space. This setting preserves the full optimal-transport geometry and provides a natural benchmark for the framework. In this regime, we show how the global signal acts as an ``attractor'' in the sense of Principle~\ref{pcp:iso_attractor}, and that a stable, state-independent contraction update rule emerges (Theorem~\ref{thm:closed_form}). We apply the stochastic geodesic projection principle (Section~\ref{sec:projector}) and recover classical update rules as special cases of the framework (Sections~\ref{sec:closed-form} and~\ref{sec:mirror-descent}).

\subsection{$W_2$--isotropic energy and global signal as an attractor}
\label{sec:isotropic-en}

We show the attractor role of the global signal in the case of isotropic (or radial) energies.

\begin{definition}[$W_2$--isotropic Tan--HWG energy]
  \label{def:isotropic}
  Consider a local energy density depending only on the Wasserstein distance to the global signal fiber component $z\in \mathcal{P}_2(Y)$ i.e. a local energy density of the general form
  \[
  \mathcal E_x(\cdot,z) := \mathcal F_x\bigl(W_2(\cdot,z)\bigr).
  \]
  Assume $\mathcal E_x(\cdot,S_x(\rho,\phi))$ is such that it defines a Tan--HWG energy $E$ with global signal $S$. We say that such energy is a \emph{$W_2$--isotropic Tan--HWG energy}.
\end{definition}

\begin{remark}[Signal-driven isotropy]
  Isotropy is defined relative to the global signal, which serves as the reference point of the geometry. The energy treats all directions pointing toward the signal equivalently, so that the signal behaves as an attractor. This reflects the fact that isotropy is not an intrinsic property of the space, but a property induced by the signal that organises the geometry of the dynamics.
\end{remark}

When $\mathcal F_x$ is strictly increasing for $\mu$-a.e.\ $x$, decreasing the energy $E(\rho,\phi)$ forces the Wasserstein distances $W_2(\rho_x,S_x(\rho, \phi))$ to decrease fiberwise, so that local state $\rho_x$ is geometrically attracted towards the fiber component $S_x(\rho,\phi)$ of the global signal in the Wasserstein space for $\mu$-a.e.\ $x$. This yields the following principle.

\begin{principle}[Geometric attraction in the isotropic case]
  \label{pcp:iso_attractor}
  Assume $E$ is a $W_2$--isotropic Tan--HWG energy, with global signal $S$ and local energy density $\mathcal E$, with $\mathcal E_x(\rho_x, S_x(\rho, \phi))=\mathcal F_x\bigl(W_2(\rho_x,S_x(\rho, \phi))\bigr)$, with $\mathcal F_x$ strictly increasing for $\mu$-a.e.\ $x$. Then, along the Tan--HWG minimizing-movement scheme dynamics, for $\mu$-a.e.\ $x$, the local memory $\rho_x$ is attracted toward the global signal fiber component $S_x(\rho, \phi)$, in the sense that whenever the energy decrease is strict, one has
  \[
  E_{\phi^n, \rho^n}(\rho^{n+1}) < E_{\phi^n, \rho^n}(\rho^{n}) \quad\Longrightarrow\quad W_2(\rho^{n+1}_x, S_{\phi^n, \rho^n,\,x}) < W_2(\rho^{n}_x, S_{\phi^n,\rho^n,\,x})
  \]
  for $\mu$-a.e.\ $x\in X$.
\end{principle}

In Section~\ref{sec:W2quadratic}, we give an explicit example of such isotropic Tan--HWG energy i.e.\ satisfying conditions of Definition~\ref{def:tan-hwg-en}.

\subsection{$W_2$--anisotropic internal geometries}

If the internal space $Y$ is equipped with another distance $d_A$ that is bi-Lipschitz equivalent to the reference metric $d_Y$, then the associated Wasserstein distance $W_A$ is equivalent to $W_2$. In this case, an energy that is isotropic with respect to $W_A$ can be interpreted as a $W_2$--anisotropic energy. The anisotropy is therefore not intrinsic: it arises from expressing the Tan--HWG dynamics in a different but equivalent internal geometry.

In this sense, the ``anisotropic'' case is simply the Tan--HWG dynamics written with respect to a bi-Lipschitz deformation of the internal metric. The underlying geometric structure is unchanged; only the coordinate representation of the dynamics differs.

\begin{remark}
  Anisotropy therefore reflects the choice of internal coordinates rather than a change in the underlying variational structure. This viewpoint is consistent with the interpretation of observable deviations from affine dynamics as signatures of curvature in the internal geometry.
\end{remark}

\subsection{$W_2$--quadratic energy}\label{sec:W2quadratic}

Whenever $Y$ is CAT(0), the Wasserstein space $(\mathcal P_2(Y),W_2)$ is CAT(0) (see \parencite{Sturm2003I}), and the map $\eta\mapsto W_2^2(\eta,z)$ is geodesically convex for $z\in\mathcal P_2(Y)$. In particular, for any convex nondecreasing function $f_x:\mathbb R_+\to\mathbb R$, the functional
\[
\mathcal E_x(\cdot,z) := f_x\bigl(W_2^2(\cdot,z)\bigr)
\]
is geodesically convex along $W_2$-geodesics. This is the most natural case among isotropic energies, where we can guarantee such convexity.

\begin{proposition}
  \label{prop:compo}
  Assume that $(Y,d_Y)$ is a Polish CAT(0) space.
  Let $f:x\mapsto f_x$ be a measurable map with, for all $x\in X$, $f_x:\mathbb{R}^+\to\mathbb{R}\cup\{+\infty\}$ proper, lower semicontinuous, and convex nondecreasing, and such that
  \[
  f_x(0)<+\infty, \qquad \text{for $\mu$-a.e.\ $x$ and} \qquad (x\mapsto f_x(0))\in\mathrm L^1(X,\mu).
  \]
  Define the local energy density $\mathcal E$ by
  \[
  \mathcal{E}_x(\rho_x,z):=f_x\big(W_2^2(\rho_x,z)\big),
  \qquad z\in\mathcal{P}_2(Y).
  \]
  Then,
  for any measurable global signal $S:\mathcal X_\mu\times\mathcal M(\Gamma)^X\to\mathcal X_\mu$, the energy
  \[
  E(\rho,\phi)
  :=\int_X \mathcal{E}_x\big(\rho_x,S_x(\rho,\phi)\big)\,\mathrm d\mu(x)
  \]
  is a locally sequentially stable Hebbian energy, hence a Tan--HWG energy. We call such energy a \emph{$W_2$--quadratic Tan--HWG energy}.
\end{proposition}

\begin{proof}
  $E$ is Hebbian in the sense of Definition~\ref{def:hebbian}. We prove that $E$ is locally sequentially stable (Definition~\ref{def:loc-seq-stab}):
  \begin{itemize}
    \item the local energy density is proper and lower semicontinuous since $f_x$ is proper, l.s.c.;
    \item geodesic convexity of the local energy density comes from the composition of the squared Wasserstein distance in the CAT(0) space with $f_x$ convex nondecreasing;
    \item the lower bound is ensured by taking $c(x, z)=f_x(0)$ and by monotony of $x\mapsto f_x\big(W_2^2(\rho_x,z)\big)$;
    \item Existence of a finite-energy field: for every reference configuration $(\phi^{\mathrm{ref}},\rho^{\mathrm{ref}})$, by taking $\rho=S(\rho^{\mathrm{ref}},\phi^{\mathrm{ref}})\in\mathcal X_\mu$ we have
    \[
    \int_X \mathcal E_x(\rho_x, S_x(\rho^{\mathrm{ref}},\phi^{\mathrm{ref}}))\,\mathrm d\mu(x)
    = \int_X f_x(0)\,\mathrm d\mu(x) < +\infty.
    \]
  \end{itemize}
  Then, by Proposition~\ref{prop:loc-seq-stab}, $E$ is sequentially stable, hence is a Tan--HWG energy.
\end{proof}

\begin{remark}[$CD(0,\infty)$]
  More generally, when $Y$ is equiped with a reference mesure $\mathfrak m_Y$, and $(Y, d_Y, \mathfrak m_Y)$ is a Polish space satisfying the synthetic curvature–dimension condition $CD(0,\infty)$ in the sense of Lott–Sturm–Villani~\parencite{LottVillani2009}, then $(\mathcal P_2(Y),W_2)$ is a geodesic space and the squared Wasserstein distance is geodesically convex along $W_2$-geodesics. In this case, Proposition~\ref{prop:compo} still holds. To our knowledge, this is the most general case where we can guarantee convexity of the isotropic energy. Other forms of $W_2$-isotropic Tan--HWG energies may exist: we leave the study of those for future work.
\end{remark}

\begin{remark}
  If each $f_x$ is $\lambda$-convex and nondecreasing on $\mathbb R_+$, with $\lambda\in(-\infty, 0)$ independent of $x$, then on any bounded subset $\{\rho : W_2(\rho,z)\le R\}$ the map $\rho\mapsto f_x(W_2^2(\rho,z))$ is $\lambda_R$--geodesically convex along $W_2$--geodesics, with $\lambda_R = 4R^2\lambda<0$. A global quantitative $\lambda$--geodesic convexity estimate, independent of $R$, would require stronger structural assumptions and is beyond the scope of this work.

  Nevertheless, in practice, it is enough to have such local $\lambda_R$--geodesic convexity on a bounded subset that is invariant under the JKO scheme. In that case, each minimisation step is well-posed, and the discrete-time dynamics is well-defined.
\end{remark}

\subsection{State-independent contraction factor}
\label{sec:closed_form}
Within the class of $W_2$--isotropic Tan--HWG energies, the following result shows that the quadratic case is the \emph{unique} situation in which the update reduces to a universal Wasserstein geodesic interpolation.

\begin{theorem}[State-independent contraction factor]
  \label{thm:closed_form}
  Let $Y$ be a CAT(0) Polish space, and $E$ a $W_2$--isotropic Tan--HWG energy in the sense of Definition~\ref{def:isotropic}, with signal $S:\mathcal X_\mu\times\mathcal M(\Gamma)^X\to\mathcal X_\mu$,
  \[
  E(\rho,\phi)
  = \int_X \mathcal{F}_x\big(W_2(\rho_x,S_x(\rho,\phi))\big)\,\mathrm{d}\mu(x).
  \]
  Assume $\mathcal{F}_x$ is a continuously differentiable and not constant for $\mu$-a.e.\ $x\in X$.
  Let $(\phi^n)_{n\ge 0}$ be a countable set of contexts.
  Let $\tau>0$.
  Assume the Tan--HWG scheme:
  \[
  \rho^{n+1}
  \in T^E_\tau(\rho^n\,|\,\phi^n)
  = \argmin_{\rho\in\mathcal{X}_\mu}
  \left\{
  \int_X \mathcal{F}_x\big(W_2(\rho_x,h^n_x)\big)\,\mathrm{d}\mu(x)
  + \frac{1}{2\tau} \int_X W_2^2(\rho_x,\rho_x^n)\,\mathrm{d}\mu(x)
  \right\},
  \tag{*}
  \]
  where at step $n$, $h^n:=S_{\phi^n,\rho^n}=S(\rho^n,\phi^n)$ is the frozen global signal.

  Then the Tan--HWG scheme (*) admits a geodesic update of the form
  \[
  \rho_x^{n+1}=\rho_{x,t_{\tau, x}}^n,
  \qquad W_2(\rho_{x,t_{\tau, x}}^n,\rho_x^n)=t_{\tau, x}\,W_2(\rho_x^n,h^n_x),
  \]
  for $\mu$-a.e.\ $x\in X$, where
  \begin{itemize}
    \item $(\rho_{x,t}^n)_{t\in[0,1]}$ is the $W_2$--geodesic from $\rho_x^n$ to $h^n_x$,
    \item and $t_{\tau, x}\in(0,1)$ is a \emph{state-independent contraction factor} i.e. independent of $\rho_x^n$ and $h_x^n$,
  \end{itemize}
  if and only if
  \[
  \mathcal{F}_x(r)=\frac{\alpha_x}{2}r^2+\mathrm{const},
  \qquad \text{with fiberwise constant }\alpha_x>0.
  \]
  We say that $E$ is an \emph{purely quadratic energy} with parameter $\alpha=\lbrace\alpha_x\rbrace_{x\in X}$.
  In particular, we have
  \[
  t_{\tau, x} = \frac{\alpha_x\tau}{1+\alpha_x\tau},
  \]
  and the purely quadratic Tan--HWG scheme is unique given an initial condition $\rho^0\in\mathcal X_\mu$.

  Among $W_2$--isotropic Tan--HWG energies, the purely quadratic case is the unique one that yields a state-independent Tan--HWG contraction factor on each fibre. In other words, only $W_2$-quadratic energies with local energy density affine in $W_2^2$ yield such fiberwise state-independent contraction factor $t_{\tau,x}$.
\end{theorem}

\begin{proof}
  Fix $n$, and set $h^n:=S_{\phi^n,\rho^n}=S(\rho^n, \phi^n)\in\mathcal X_\mu$ the frozen signal. By the Hebbian structure and the definition of $E_{\phi^n,\rho^n}$, the JKO functional at step $n$ reads
  \[
  \mathcal{J}(\rho)
  = \int_X \left[
  \mathcal{F}_x\big(W_2(\rho_x,h^n_x)\big)
  + \frac{1}{2\tau}W_2^2(\rho_x,\rho_x^n)
  \right]\mathrm{d}\mu(x).
  \]
  The integral is a sum of independent fiberwise contributions, so the minimization decouples: for $\mu$-a.e.\ $x$, $\rho_x^{n+1}$ minimizes
  \[
  \mathcal J_x(\rho_x)
  :=\frac{1}{2\tau}W_2^2(\rho_x,\rho_x^n)
  + \mathcal{F}_x\big(W_2(\rho_x,h^n_x)\big)
  \]
  over $\mathcal{P}_2(Y)$.

  Assume
  \[
  \rho_x^{n+1}=\rho_{x,t_{\tau,x}}^n,
  \qquad
  W_2(\rho_{x,t_{\tau,x}}^n,\rho_x^n)=t_{\tau,x}\,W_2(\rho_x^n,h^n_x),
  \]
  for $\mu$-a.e.\ $x\in X$, where $(\rho_{x,t}^n)_{t\in[0,1]}$ is the $W_2$--geodesic from $\rho_x^n$ to $h^n_x$, and with a coefficient $t_{\tau,x}\in(0,1)$ \emph{independent of $\rho_x^n$}.
  Fix such $x$ and denote
  \[
  D := W_2(\rho_x^n,h^n_x),
  \]
  On the geodesic, we have $W_2(\rho_x^{n+1},\rho_x^n)=W_2(\rho_{x,t_{\tau,x}}^n,\rho_x^n)=t_{\tau,x} \,D$, and for all $t\in\left[0, 1\right]$
  \[
  W_2(\rho_{x,t}^n,\rho_x^n)=t\,D
  \quad \text{and }
  W_2(\rho_{x,t}^n,h^n_x)=(1-t)\,D.
  \]
  Since $\rho_x^{n+1}$ minimizes $\mathcal J_x$, it minimizes it along the geodesic $(\rho_{x,t}^n)_{t\in[0,1]}$ i.e. it minimizes
  \[
  f(t)
  :=\mathcal J_x(\rho_{x,t}^n)
  =\frac{1}{2\tau}W_2^2(\rho_{x,t}^n,\rho_x^n)
  + \mathcal{F}_x\big(W_2(\rho_{x,t}^n,h^n_x)\big)
  =\frac{1}{2\tau}t^2\,D^2
  + \mathcal{F}_x\big((1-t)\,D).
  \]
  The minimum is reached at $t=t_{\tau,x}$ where $\rho_{x,t}^n=\rho_x^{n+1}$, so, with $\mathcal F_x$ continuously differentiable, we have
  \[
  0=f'(t_{\tau,x})
  =\frac{t_{\tau,x}}{\tau}\,D^2
  -D \mathcal{F}_x'\big((1-t_{\tau,x})\,D).
  \]
  Since the coefficient $t_{\tau,x}\in(0,1)$ is \emph{independent of $\rho_x^n$}, this holds true for any $D= W_2(\rho_x^n,h^n_x)$.
  If $W_2(\rho_x^n,h^n_x)=0$ for all $x$ and all $n$, then $\rho_x^n=h^n_x$ for all $n$ and the trajectory is constant. In this trivial case the statement is empty, so we may assume that there exists at least one $x$ and one step $n$ such that $D := W_2(\rho_x^n,h^n_x) > 0$.
  And, for all $D>0$,
  \[
  \frac{t_{\tau,x}}{\tau}=\frac{\mathcal{F}_x'\big((1-t_{\tau,x})\,D)}{D}.
  \]
  This has to hold true for any $\rho^n$ and $h^n$, therefore for any distance $D=W_2(\nu,\nu')$ with $\nu\neq\nu'$ arbitrary, as well as any distribution on the geodesic between $\nu$ and $\nu'$, hence for any $\widetilde D:=t\,D$, $t\in[0,1]$.
  Thus $\frac{t_{\tau,x}}{\tau}-\frac{\mathcal{F}_x'\big((1-t_{\tau,x})\,D)}{D}=0$ on any segment $[0,D]$ for any $D$. Thus $\mathcal{F}_x'\big((1-t_{\tau,x})\,D)$ is linear in $D$ i.e. of the form
  \[
  \mathcal F_x'(r) = \alpha_x r,
  \]
  with $\alpha_x>0$ since $\mathcal F_x$ is convex non constant. We can compute $\alpha_x$ from
  \[
  \frac{t_{\tau,x}}{\tau}=\alpha_x (1-t_{\tau,x})
  \]
  yielding
  \[
  \alpha_x = \frac{t_{\tau,x}}{\tau(1-t_{\tau,x})}.
  \]
  Then integrating gives
  \[
  \mathcal{F}_x(r)=\frac{\alpha_x}{2}r^2+\mathrm{const}.
  \]

  Conversely, if $\mathcal{F}_x(r)=\frac{\alpha_x}{2}r^2+\mathrm{const}$, with $\alpha_x>0$, the JKO functional at step $n$ reads
  \[
  \mathcal{J}(\rho)
  = \int_X \left[
  \frac{\alpha_x}{2}\,W_2^2(\rho_x,h^n_x)
  + \frac{1}{2\tau}W_2^2(\rho_x,\rho_x^n)
  \right]\mathrm{d}\mu(x).
  \]
  The integral is a sum of independent fiberwise contributions, so the minimization decouples: for $\mu$-a.e.\ $x$, $\rho_x^{n+1}$ minimizes
  \[
  \mathcal J_x(\rho_x)
  :=\frac{1}{2\tau}W_2^2(\rho_x,\rho_x^n)
  + \frac{\alpha_x}{2}\,W_2^2(\rho_x,h^n_x)
  \]
  over $\mathcal{P}_2(Y)$.
  Fix $x$ and denote $D := W_2(\rho_x^n,h^n_x)$. Let $(\rho_{x,t}^n)_{t\in[0,1]}$ be the $W_2$--geodesic from $\rho_x^n$ to $h^n_x$. Along this geodesic we have
  \[
  W_2(\rho_{x,t}^n,\rho_x^n)=tD,
  \qquad
  W_2(\rho_{x,t}^n,h^n_x)=(1-t)D,
  \]
  so that
  \[
  f(t)
  :=\mathcal J_x(\rho_{x,t}^n)
  =\frac{1}{2\tau}t^2D^2+\frac{\alpha_x}{2}(1-t)^2D^2
  \]
  is a strictly convex quadratic polynomial in $t$. Its unique minimizer is
  \[
  t_{\tau,x}=\frac{\alpha_x\tau}{1+\alpha_x\tau},
  \]
  and the corresponding value is
  \[
  \min_{t\in[0,1]} f(t)
  = \mathcal J_x(\rho_{x,t_{\tau,x}}^n).
  \]
  We now show that this minimizer on the geodesic is actually a minimizer on all $\mathcal P_2(Y)$. For any $\rho_x$ with $r:=W_2(\rho_x,h^n_x)$, the triangle inequality gives
  \[
  W_2(\rho_x,\rho_x^n)\ge |D-r|,
  \]
  hence
  \[
  \mathcal J_x(\rho_x)
  \ge \frac{1}{2\tau}(D-r)^2+\frac{\alpha_x}{2}r^2
  =: \Psi_D(r).
  \]
  A direct computation shows that $\Psi_D$ is strictly convex in $r$ and attains its unique minimum at $r^\ast=(1-t_{\tau,x})D$, with
  \[
  \Psi_D(r^\ast)=\mathcal J_x(\rho_{x,t_{\tau,x}}^n).
  \]
  Therefore, for any $\rho_x$,
  \[
  \mathcal J_x(\rho_x)\ge \Psi_D(r)\ge \Psi_D(r^\ast)
  =\mathcal J_x(\rho_{x,t_{\tau,x}}^n).
  \]
  Hence,
  \[
  \min_{\rho_x\in\mathcal P_2(Y)}{\mathcal J_x(\rho_x)} \ge \mathcal J_x(\rho_{x,t_{\tau,x}}^n),
  \]
  so
  \[
  \rho_{x,t_{\tau,x}}^n \in \argmin_{\rho_x\in\mathcal P_2(Y)}{\mathcal J_x(\rho_x)}.
  \]
  Thus, the Tan--HWG scheme admits a solution of the form
  \[
  \rho_x^{n+1}=\rho_{x,t_{\tau,x}}^n,
  \qquad W_2(\rho_{x,t_{\tau,x}}^n,\rho_x^n)=t_{\tau,x}\,W_2(\rho_x^n,h^n_x),
  \]
  for $\mu$-a.e.\ $x\in X$, where $(\rho_{x,t}^n)_{t\in[0,1]}$ is the $W_2$--geodesic from $\rho_x^n$ to $h^n_x$, and with a coefficient
  \[
  t_{\tau,x} = \frac{\alpha_x\tau}{1+\alpha_x\tau} \in(0,1) \text{ independent of } \rho_x^n.
  \]
  This proves the converse.

  Given an initial condition $\rho^0\in\mathcal X_\mu$, the scheme is uniquely defined, since for each $x$ the functional $\mathcal J_x$ is strictly convex on $\mathcal P_2(Y)$ when $Y$ is CAT(0), hence admits a unique minimizer.

  This completes the proof.
\end{proof}

\begin{remark}[Geometric intuition]
  Quadratic energies are the only ones whose Wasserstein gradient flow naturally follows geodesics. Indeed, for a functional of the form $\rho\mapsto \tfrac12 W_2^2(\rho,z)$, the Wasserstein gradient is exactly the initial velocity of the geodesic from $\rho$ to $z$. As a consequence, the JKO update balances two geodesic velocities---one pointing toward the signal and one toward the previous state---and the minimizer lies at a fixed proportion along the geodesic. For any non-quadratic energy, the gradient is no longer aligned with the geodesic direction, and the update cannot reduce to a universal geodesic interpolation.
\end{remark}

\emph{Notation.}
For each fiber $x$, the parameter $\alpha_x$ is a constant depending only on $x$. We may write $\alpha:x\mapsto\alpha_x$. Similarly, the contraction factor $t_{\tau,x}$ depends only on $x$ and on the timestep $\tau$, and we may write $t_\tau(x)$.

\begin{remark}[Clustering and reduction to a constant parameter]
  \label{rmk:cluster}
  The Hebbian structure ensures that the scheme decouples fiberwise. We may therefore group together fibers that share the same local energy density. Such a group, or \emph{clusters}, is a subset $X'\subset X$ on which all $\mathcal F_x$ coincide. On each cluster, the parameter $\alpha_x$ is constant, say $\alpha>0$, and so is the contraction factor $t_{\tau,x}$, which we denote simply by $t_\tau$:
  \[
  t_\tau=\frac{\alpha\tau}{1+\alpha\tau}.
  \]
  In this case, the update can be written compactly at the field level as
  \[
  \rho^{n+1}=\rho^n_{t_\tau}.
  \]
  Unless otherwise stated, we will work within a single cluster, which allows us to simplify notation without loss of generality.
\end{remark}

Table~\ref{tab:quadratic-en} summarizes what we have seen so far in the present section, with the following inclusions:
\[
\{\text{Purely quadratic energies}\} \subset \{W_2\text{--quadratic energies}\} \subset \{W_2\text{--isotropic energies}\}.
\]
Purely quadratic energies form the rigid core of the isotropic class: they are the only ones whose gradient flow follows Wasserstein geodesics with a state-independent speed. Quadratic energies in $W_2^2$ retain convexity but lose this rigidity, while general isotropic energies preserve only the attractor structure.

\begin{table}[!hbtp]
  \centering
  \begin{tabularx}{\textwidth}{Y Y Y}
    \toprule
    \textbf{Energy class}
    & \textbf{Local density form}
    & \textbf{Main property} \\
    \midrule
    $W_2$--isotropic Tan--HWG
    & $\mathcal F_x \circ W_2$
    & $S$ is an attractor \\
    $W_2$--quadratic
    & $f_x\circ W_2^2$
    & Tan--HWG energy \\
    Purely quadratic
    & $\frac{\alpha}{2}\,W_2^2 + const$
    & $\rho^{n+1}=\rho^n_{t_\tau}$ \& uniqueness\\
    \bottomrule
  \end{tabularx}
  \caption{Summary of main properties by isotropic energy class.}
  \label{tab:quadratic-en}
\end{table}

\subsection{Stochastic Tan--HWG geodesic projector}\label{sec:projector}

In this section, we apply the stochastic geodesic projection principle seen in Section~\ref{sec:geodesic-projection-principle} along the trajectory of a purely quadratic Tan--HWG scheme.

Let $Y$ be a CAT(0) Polish space. Consider a finite neural system with $M$ internal degrees of freedom on $Y$. We denote by $Y_M\subset Y$ the family of internal states supporting the local memory states $\rho_x$ of the system.
Consider a family $(\rho^n)_{n\in\mathbb N}$ following a purely quadratic Tan--HWG scheme, with parameter $\alpha$, time step $\tau>0$, signal $S$ and associated frozen signal family $(h^n)_{n\in\mathbb N}\in\mathcal X_\mu^{\mathbb N}$ (for a countable set of contexts $(\phi^n)_{n\ge 0}$).
Assume that
\[
h^0_x\in \mathcal P_2(Y_M),
\qquad\text{and}\qquad
\rho^0_x=\rho_x\in \mathcal P_2(Y_M),
\]
i.e.\ the initial memory states and initial global signal are observable states relative to the observable space $Y_M$ in the sense of Definition~\ref{def:observable}.
Assume that for all $(\eta, \psi) \in \mathcal X_\mu\times\mathcal M(\Gamma)^X$
\[
\operatorname{supp}(S_x(\eta, \psi))\subset Y_M,
\]
i.e.\ the signal has values in the space of observable states $\mathcal P_2(Y_M)$, hence $h^n_x\in \mathcal P_2(Y_M)$ for all $n$.

\paragraph{Intuition.}
Theorem~\ref{thm:closed_form} yields that the update $\rho^{n+1}_x$ is on the constant-speed geodesic between $\rho^n_x$ and $h^n_x$, at parameter $t_\tau$ (we reduce to the study of clusters in the sense of Remark~\ref{rmk:cluster}). We apply stochastic geodesic projection principle to $\rho^{n+1}_x$ lying on the geodesic between $\rho^n_x$ and $h^n_x$. Then, by recursively applying such stochastic geodesic projections, we obtain a projected observable state relative to $Y_M$.

\bigskip
More precisely: for all $x\in X$, we denote by $(\mathcal S_{x,\tau}^n)_{n\in\mathbb N}$ the family of finite subsets of $Y$ defined recursively by
\[
\mathcal S_{x,\tau}^{n+1}:=\operatorname{supp}(\rho_x^{n+1}) \cup \mathcal S_{x,\tau}^n,
\qquad\text{and}\qquad
\mathcal S_{x,\tau}^0=Y_M.
\]
In particular,
\[
\operatorname{supp}(h_x^n)\subset\mathcal S_{x,\tau}^n,
\]
by stability of support under the signal map $S$.

Fix $x\in X$ and $n\in\mathbb N$. Let $\{y_1,\dots,y_{N_n}\}:=\mathcal S_{x,\tau}^n$ and let $\pi_x^n=(\pi_{kl}^n)_{1\le k,l\le {N_n}}$ be an optimal transport plan between $\rho_x^n$ and $h_x^n$.
For each pair $(y_k,y_l)$ with $\pi_{kl}^n>0$, we denote by
\[
  \gamma_{kl}(t) \qquad (t\in[0,1])
\]
the $W_2$--geodesic interpolation between $y_k$ and $y_l$.

\paragraph{Step 1. Couples generating a point.}
For any point
\[
  y \in \mathcal S_{x,\tau}^{n+1}\setminus \mathcal S_{x,\tau}^n,
\]
we define the set of generating couples, i.e.\ the pairs $(k,l)$ whose geodesic interpolation at time $t_\tau$ produces the new point $y$,
\[
  \zeta_{x,\tau}^n(y)
  := \big\{ (k,l)
      \;\big|\;
      \gamma_{kl}(t_\tau)=y
     \big\},
\]
as well as the sets of such couples originating from, and arriving to, a given $i$:
\[
  \mathcal T_i(y):=\big\{ l \;\big|\; \gamma_{il}(t_\tau)=y \big\},
  \qquad
  \mathcal O_i(y):=\big\{ k \;\big|\; \gamma_{ki}(t_\tau)=y \big\}.
\]

\begin{lemma}\label{lem:gen-couples}
  For any point $y \in \mathcal S_{x,\tau}^{n+1}\setminus \mathcal S_{x,\tau}^n$, and for all couple $(k,l)\in\llbracket 1, N_n \rrbracket^2$:
  \[
  (k,l)\in\zeta_{x,\tau}^n(y) \Longrightarrow k\neq l,
  \]
  And we have
  \[
    \bigcup_{y\in\mathcal S_{x,\tau}^{n+1}\setminus\mathcal S_{x,\tau}^n}
      \zeta_{x,\tau}^n(y)
    =
    \lbrace(k,l)\in\llbracket 1, N_n\rrbracket^2 \mid k\neq l\rbrace.
  \]
\end{lemma}
\begin{proof}
  Assume the first claim is not true, then $y=\gamma_{kk}(t_\tau)=y_k\in\mathcal S_{x,\tau}^n$, contradicting $y\notin\mathcal S_{x,\tau}^n$.

  Now fix $(k,l)\in \bigcup_{y\in \mathcal S_{x,\tau}^{n+1}\setminus\mathcal S_{x,\tau}^n} \zeta_{x,\tau}^n(y)$, then there exists $y\in \mathcal S_{x,\tau}^{n+1}\setminus\mathcal S_{x,\tau}^n$ such that $y=\gamma_{kl}(t_\tau)$ and by the first claim, $k\neq l$.

  Conversely, for any $k\neq l$ in $\llbracket 1, N_n\rrbracket$, taking $y=\gamma_{kl}(t_\tau)$ suffices to prove the converse inclusion.
\end{proof}

\paragraph{Step 2. Stochastic backward projector.}
Let $\nu\in\mathcal P_2(\mathcal S_{x,\tau}^{n+1})$.
We define a random measure $\hat\nu$ with values in $\mathcal P_2(\mathcal S_{x,\tau}^n)$, via a random partition of $\mathcal S_{x,\tau}^{n+1}$:
\[
\mathcal S_{x,\tau}^{n+1}=\bigsqcup_{i=1}^{N_n} \mathcal Y_i, \qquad \mathcal Y_i \cap \mathcal Y_j = \emptyset \text{ if } i\neq j,
\]
such that, $\hat\nu(\{y_i\})=\nu(\mathcal Y_i)$, and for all $y\in \mathcal S_{x,\tau}^{n+1}$:
\begin{itemize}
  \item if $y=y_i\in\mathcal S_{x,\tau}^n$, the mass $\nu(\{y_i\})$ stays at $y_i$ i.e.
  \[
  \mathbb P\left( y_i \in \mathcal Y_i \right)=1;
  \]
  \item if $y\in\mathcal S_{x,\tau}^{n+1}\setminus\mathcal S_{x,\tau}^n$, then:
  \[
  \mathbb P\left(y\in\mathcal Y_i\mid y\in\mathcal S_{x,\tau}^{n+1}\setminus\mathcal S_{x,\tau}^n\right)
  = (1-t_\tau)\,
  \displaystyle\sum_{l\in\mathcal{T}_i(y)}
    \frac{\pi_{il}^n}{\rho^{n+1}_x(\{y\})}
  + t_\tau\,
  \displaystyle\sum_{k\in\mathcal{O}_i(y)}
    \frac{\pi_{ki}^n}{\rho^{n+1}_x(\{y\})}.
  \]
\end{itemize}
The second item is well defined since $\rho^{n+1}_x(\{y\})\neq 0$ with $y\in\mathcal S_{x,\tau}^{n+1}\setminus\mathcal S_{x,\tau}^n\subset\operatorname{supp}(\rho_x^{n+1})$, and it corresponds to the following draws:
\begin{enumerate}
  \item we choose a pair $(k,l)\in\zeta_{x,\tau}^n(y)$ with probability
  \[
    \mathbb P\big((k,l) \text{ is chosen} \mid y\big)
    = \frac{\pi_{kl}^n}{\rho^{n+1}_x(\{y\})},
  \]
  i.e.\ the probability of selecting $(k,l)$ is the portion of mass transported along the geodesic $\gamma_{kl}$ that contributes to $y$;
  \item and we send the mass $\nu(\{y\})$ to
  \[
    y_k \text{ with probability } (1-t_\tau),
    \qquad
    y_l \text{ with probability } t_\tau,
  \]
  which corresponds to the stochastic geodesic principle.
\end{enumerate}

We call the map $\Pi^n_{x,\tau}:\nu\mapsto\hat\nu$, the \emph{stochastic backward projector}.

\paragraph{Step 3. Expectation operator.}
The deterministic \emph{expectation operator}
\[
  \mathbb E_{x,\tau}^n : \mathcal P_2(\mathcal S_{x,\tau}^{n+1})
  \longrightarrow \mathcal P_2(\mathcal S_{x,\tau}^n)
\]
is defined by
\[
  \mathbb E_{x,\tau}^n(\nu)
  := \mathbb E[\hat\nu].
\]

\begin{lemma}\label{lem:expectation-op}
  For every $y_i\in\mathcal S_{x,\tau}^n$, one has
  \[
  \mathbb E_{x,\tau}^n(\nu)(\{y_i\})
  = \nu(\{y_i\})
  + \sum_{y \in \mathcal S_{x,\tau}^{n+1}\setminus\mathcal S_{x,\tau}^n}
      \nu(\{y\})
      \sum_{(k,l)\in \zeta_{x,\tau}^n(y)}
        \Big[
          (1-t_\tau)\mathbf 1_{\{y_k=y_i\}}
          + t_\tau\mathbf 1_{\{y_l=y_i\}}
        \Big]
        \frac{\pi_{kl}^n}{\rho^{n+1}_x(\{y\})}.
  \]
  In particular $\mathbb E_{x,\tau}^n$ is a linear operator from $\mathcal P_2(\mathcal S_{x,\tau}^{n+1})$ to $\mathcal P_2(\mathcal S_{x,\tau}^n)$. For all $t\in[0,1]$, and all $(\nu, \nu')\in\mathcal P_2(\mathcal S_{x,\tau}^{n+1})\times\mathcal P_2(\mathcal S_{x,\tau}^{n+1})$,
  \[
    \mathbb E_{x,\tau}^n((1-t)\nu + t\nu')=(1-t)\mathbb E_{x,\tau}^n(\nu)+t\mathbb E_{x,\tau}^n(\nu').
  \]
  Moreover, for all $n$, $\mathbb E_{x,\tau}^n$ acts as the identity on $\mathcal P_2(Y_M)$, and
  \[
    \mathbb E_{x,\tau}^n(h_x^k)=h_x^k, \quad\text{for all } k,
  \]
  and
  \[
    \mathbb E_{x,\tau}^n(\rho_x^{n+1})
    = (1-t_\tau)\,\rho^n_x + t_\tau\,h^n_x.
  \]
\end{lemma}
\begin{proof}
  The first claim is a direct calculation of the expected value of $\hat\nu$, and the second is a direct consequence.

  For all $\eta\in\mathcal P_2(Y_M)$, $\eta$ has support in $Y_M\subset\mathcal S_{x,\tau}^n$, we have $\eta(\{y\})=0$ for $y\notin\mathcal S_{x,\tau}^n$, hence $\mathbb E_{x,\tau}^n(\eta)=\eta$ by direct application of the first claim. In particular, since $h^k_x\in\mathcal P_2(Y_M)$ for all $k$, $\mathbb E_{x,\tau}^n(h_x^k)=h_x^k$.

  For every $y_i\in\mathcal S_{x,\tau}^n$,
  \begin{align*}
    \mathbb E_{x,\tau}^n(\rho^{n+1}_x)(\{y_i\})
    &= \rho^{n+1}_x(\{y_i\})
    + \sum_{\mathcal S_{x,\tau}^{n+1}\setminus\mathcal S_{x,\tau}^n}
        \frac{\rho^{n+1}_x(\{y\})}{\rho^{n+1}_x(\{y\})}
        \sum_{\zeta_{x,\tau}^n(y)}
          \Big[
            (1-t_\tau)\mathbf 1_{\{y_k=y_i\}}
            + t_\tau\mathbf 1_{\{y_l=y_i\}}
          \Big]
          \pi_{kl}^n\\
    &= \rho^{n+1}_x(\{y_i\})
    + \sum_{y \in \mathcal S_{x,\tau}^{n+1}\setminus\mathcal S_{x,\tau}^n}
        \sum_{(k,l)\in \zeta_{x,\tau}^n(y)}
          \Big[
            (1-t_\tau)\mathbf 1_{\{y_k=y_i\}}
            + t_\tau\mathbf 1_{\{y_l=y_i\}}
          \Big]
          \pi_{kl}^n.
  \end{align*}

  Since Lemma~\ref{lem:gen-couples} yields
  \[
    \bigcup_{y\in \mathcal S_{x,\tau}^{n+1}\setminus\mathcal S_{x,\tau}^n}
      \zeta_{x,\tau}^n(y)
    =
    \lbrace(k,l)\in\llbracket 1, N_n\rrbracket^2 \mid k\neq l\rbrace,
  \]
  the previous equality reduces to
  \begin{align*}
    \mathbb E_{x,\tau}^n(\rho^{n+1}_x)(\{y_i\})
    &= \rho^{n+1}_x(\{y_i\})
     + \sum_{k \neq l}
        \Big[
          (1-t_\tau)\mathbf 1_{\{y_k=y_i\}}
          + t_\tau\mathbf 1_{\{y_l=y_i\}}
        \Big]
        \pi_{kl}^n\\
    &= \rho^{n+1}_x(\{y_i\})
    + \sum_{l \neq i} (1-t_\tau)\,\pi_{il}^n
    + \sum_{k \neq i} t_\tau\,\pi_{ki}^n.
  \end{align*}

  By construction of $\pi_x^n$ and $\gamma_{(\cdot, \cdot)}$, and by Theorem~\ref{thm:closed_form}:
  \[
  \rho^n_x = \sum_{k=1}^{N_n}\sum_{l=1}^{N_n} \pi_{kl}^n\,\delta_{y_k},
  \qquad
  h^n_x = \sum_{l=1}^{N_n}\sum_{k=1}^{N_n} \pi_{kl}^n\,\delta_{y_l},
  \qquad
  \rho_x^{n+1}=\sum_{k=1}^{N_n}\sum_{l=1}^{N_n} \pi_{kl}^n\,\delta_{\gamma_{kl}(t_\tau)}.
  \]
  In particular, evaluating at $y_i$ yields
  \[
  \rho^n_x(\{y_i\}) = \sum_{l=1}^{N_n} \pi_{il}^n,
  \qquad
  h^n_x(\{y_i\}) = \sum_{k=1}^{N_n} \pi_{ki}^n,
  \qquad
  \rho^{n+1}_x(\{y_i\}) = \pi^n_{ii},
  \]
  since $\delta_{\gamma_{kl}(t_\tau)}(y_i)=0$ whenever $k\neq i$ and $l\neq i$, because $\gamma_{kl}(t_\tau)$ cannot coincide with $y_i$ unless $k=l=i$.
  Thus
  \begin{align*}
    (1-t_\tau)\,\rho^n_x(\{y_i\}) + t_\tau\,h^n_x(\{y_i\})
    &=(1-t_\tau)\,\sum_{l=1}^{N_n} \pi_{il}^n + t_\tau\,\sum_{k=1}^{N_n} \pi_{ki}^n\\
    &=\pi^n_{ii}
      + \sum_{l \neq i} (1-t_\tau)\,\pi_{il}^n
      + \sum_{k \neq i} t_\tau\,\pi_{ki}^n.
  \end{align*}
  Comparing with the previous evaluation of $\mathbb E_{x,\tau}^n(\rho^{n+1}_x)(\{y_i\})$ yields the final claim.
\end{proof}

Note that the restriction of each expectation operator $\mathbb E_{x,\tau}^n$ on $\mathcal P_2(Y_M)$ is the identity.
This technical construction ensures the well-posedness of the following definition.

\begin{definition}[Stochastic Tan--HWG geodesic projector]
  Consider the setting described above. The projector $\mathrm{Proj}_{x,\tau}^n: \mathcal P_2(\mathcal S_{x,\tau}^{n}) \longrightarrow \mathcal P_2(Y_M)$ defined for all $(x,n)\in X\times \mathbb N$ by
  \[
    \mathrm{Proj}_{x,\tau}^{n+1}:=\mathrm{Proj}_{x,\tau}^n\circ\mathbb E_{x,\tau}^n,
    \qquad \mathrm{Proj}_{x,\tau}^0 = \mathrm{Id}_{\mathcal P_2(Y_M)},
  \]
  is called a \emph{stochastic Tan--HWG geodesic projector} for the purely quadratic Tan--HWG scheme. We also denote
  \[
  \mathrm{Proj}_\tau:=(\lbrace \mathrm{Proj}_{x,\tau}^n \rbrace_{x\in X})_{n\in\mathbb N}.
  \]
\end{definition}

\begin{remark}[Relativity of projector and non uniqueness]
  A stochastic Tan--HWG geodesic projector is a family of maps $\mathcal P_2(\mathcal S_{x,\tau}^n) \rightarrow \mathcal P_2(Y_M)$ defined subsequently to a purely quadratic Tan--HWG scheme.

  The construction above show that such projector always exists for a purely quadratic Tan--HWG scheme. The stochastic Tan--HWG projector is unique exactly when both the optimal transport plan $\pi_x^n$ and all induced geodesics $\gamma_{kl}$ are unique. This is automatic in Euclidean spaces with quadratic cost when $\rho_x^n$ is absolutely continuous (Brenier), but not in general CAT(0) spaces. For the purpose of this paper, we do not require uniqueness of the projector. Since uniqueness is not guaranteed, we may write $p \in \mathrm{Proj}_\tau$ to denote the choice of such a projector.
\end{remark}

We have
\[
  \mathrm{Proj}_{x,\tau}^{n+1}=\mathbb E_{x,\tau}^0\circ \mathbb E_{x,\tau}^1\dots\circ\mathbb E_{x,\tau}^n,
\]
and
\[
\mathrm{Proj}_{x,\tau}^{n+1}(\rho^{n+1}_x)=(1-t_\tau)\mathrm{Proj}_{x,\tau}^{n}(\rho^n_x)+t_\tau\,h^n_x,
\]
since $\mathrm{Proj}_{x,\tau}^{n}(h^n_x)=h^n_x$ for $h^n_x\in \mathcal P_2(Y_M)$ with support in $Y_M\subset \mathcal S_{x,\tau}^n$. This is the key identity showing that the projected dynamics follows an exponential moving average of the signal (see Section~\ref{sec:closed-form}).

\subsection{Closed-form EMA observable dynamics}
\label{sec:closed-form}

The projector mechanism of the previous section induces a closed-form observable dynamics.

\begin{theorem}[Observable dynamics induced by stochastic geodesic projection]
  \label{thm:observable-dynamics}
  Let $Y$ be a CAT(0) Polish space, and $Y_M\subset Y$ the support of a finite neural system with $M$ internal degrees of freedom on $Y$. Consider a purely quadratic Tan--HWG energy $E$ with parameter $\alpha$, and signal $S:\mathcal X_\mu\times\mathcal M(\Gamma)^X\to\mathcal X_\mu$.

  Let $\tau>0$.
  Let $(\phi^n_\tau)_{n\ge 0}$ be a countable set of contexts.

  Assume $S_x(\mathcal X_\mu,\phi(t))\subset \mathcal P_2(Y_M)$ for all $t\ge 0$, and consider an initial state $\rho^0$ such that for all $x\in X$, $\rho^0_x\in \mathcal P_2(Y_M)$ (observability of signal and initial state).

  Let $(\rho^n_\tau)_{n\in\mathbb N}$ follow the $\tau$-Tan--HWG scheme relative to $E$ with initial condition $\rho^0$, under context $\phi$:
  \[
  \rho^{n+1}_\tau \in T^E_\tau(\rho^n_\tau\,|\,\phi^n_\tau), \qquad n\in\mathbb N.
  \]
  We denote $(h^n_\tau)_{n\in\mathbb N}$ the associated frozen signal family, and $\mathrm{Proj}_\tau=(\lbrace \mathrm{Proj}_{x,\tau}^n \rbrace_{x\in X})_{n\in\mathbb N}$ an associated stochastic Tan--HWG geodesic projector.

  The observable sequence $(\hat\rho^n_\tau)_{n\in \mathbb N}$ defined by $\hat\rho^{n}_{\tau, x} := \mathrm{Proj}_{x,\tau}^{n}(\rho^{n}_\tau)$ satisfies the affine recurrence
  \begin{equation}
    \label{eq:observable-recurrence}
    \hat\rho^{n+1}_\tau
    = (1-t_\tau)\,\hat\rho^n_\tau + t_\tau\,h^n_\tau.
  \end{equation}

  Consequently, the observable sequence admits the explicit closed-form expression
  \begin{equation}
    \label{eq:observable-closed-form}
    \hat\rho^n_\tau
    = (1-t_\tau)^n \hat\rho^0
    + t_\tau \sum_{k=0}^{n-1} (1-t_\tau)^{n-1-k} h^k_\tau.
  \end{equation}

  Assume $\rho^n_\tau$ and $\phi^n_\tau$ admit limit curves $\rho:t\to\rho(t)$ and $\phi:t\to\phi(t)$ when $\tau\to 0$ with $n\tau\to t$. In the continuous-time limit $\tau\to0$ with $n\tau\to t$, the observable converges to the solution of the ODE
  \begin{equation}
    \label{eq:observable-ODE}
    \frac{\mathrm d}{\mathrm dt}\,\hat\rho(t)
    = -\alpha\,\hat\rho(t) + \alpha\,h(t),
  \end{equation}
  namely
  \begin{equation}
    \label{eq:observable-ODE-solution}
    \hat\rho(t)
    = e^{-\alpha t}\hat\rho(0)
    + \alpha\int_0^t e^{-\alpha(t-s)} h(s)\,\mathrm ds,
  \end{equation}
  where $h(t)=S(\rho(t),\phi(t))$.
\end{theorem}

\begin{proof}
  Applying recursively Lemma~\ref{lem:expectation-op} yields \eqref{eq:observable-recurrence}. Iterating the recurrence gives \eqref{eq:observable-closed-form}.

  The continuous-time limit follows from the standard convergence of the discrete exponential filter to the ODE \eqref{eq:observable-ODE}, whose solution is \eqref{eq:observable-ODE-solution}.
\end{proof}

Thus, in the purely quadratic case, the observable dynamics is an exponential filtering of the signal. In the general setting, this result holds by homogeneous cluster with corresponding parameters $\alpha(x)$ and contraction factors $t_\tau(x)$ (see Remark~\ref{rmk:cluster}).

\begin{remark}[Bridge OT-emerging stochastic process-exponential filtering]
  The observable dynamics \eqref{eq:observable-recurrence}--\eqref{eq:observable-ODE} can thus be interpreted as the macroscopic effect of a particle-wise stochastic geodesic projection along the optimal transport plans generated by the Tan--HWG update.

  At the observable level, this induces an exponential moving average (EMA) of the past signal values, with decay factor $(1-t_\tau)$.

  Conversely, any EMA dynamics on a simplex arises as the observable dynamics of the stochastic Tan--HWG geodesic projection applied to a purely deterministic Tan--HWG dynamics (convexity of the purely quadratic energy yields uniqueness of the minimizer at each update\footnote{This determinism is specific to the purely quadratic case: the strict convexity of the quadratic energy ensures uniqueness of the minimizer at each JKO step. For general $W_2$--isotropic energies, only weak or local convexity is available, and the JKO update may admit multiple minimizers, preventing a deterministic underlying dynamics.}). Since exponential moving averages are ubiquitous in signal processing, machine learning, and computational neuroscience, this geometric derivation reveals a surprisingly universal mechanism underlying a widely used update rule.

  To our knowledge, this is the first framework in which an OT-induced stochastic process yields an exact exponential moving average at the observable level.
\end{remark}

\begin{remark}[On the structural role of the contraction factor]
  A notable feature of the Tan--HWG update is that the contraction factor
  \[
  t_\tau=\frac{\alpha\tau}{1+\alpha\tau}
  \]
  is not an approximation of a continuous-time dynamics, but the \emph{exact} minimizer of the discrete variational problem associated with the quadratic Wasserstein energy
  \[
  \mathcal E(\rho_x,h_x)=\frac{\alpha}{2}W_2^2(\rho_x,h_x).
  \]
  As a consequence, the observable update
  \[
  \hat\rho^{n+1}_x=(1-t_\tau)\hat\rho^n_x+t_\tau h^n_x
  \]
  is \emph{exactly linear} in $\hat\rho_x$, with a contraction factor that is constant and independent of the state. This linearity is the direct reflection of the state-independent geodesic contraction established in Theorem~\ref{thm:closed_form}. Such a fixed-ratio interpolation along Wasserstein geodesics can only arise from a \emph{purely quadratic} isotropic energy, as shown in Theorem~\ref{thm:closed_form}.

  In the continuous-time limit, this yields the exponential filtering equation
  \[
  \frac{\mathrm d}{\mathrm dt}\,\tilde\rho_x=-\alpha(\tilde\rho_x-h_x),
  \]
  which therefore appears not as an external modeling assumption, but as the \emph{unique} macroscopic law compatible with the geometric structure of the purely quadratic Tan--HWG dynamics.

  Moreover, the exact expression of $t_\tau$ contrasts with the usual Euler-type approximations, where one replaces $t_\tau$ by $\alpha\tau$. Such approximations may lose stability (e.g.\ when $\alpha\tau>1$), whereas the Tan--HWG factor always satisfies $0<t_\tau<1$ for all $\tau>0$, ensuring \emph{unconditional stability} of the discrete dynamics. The Tan--HWG scheme is therefore not a numerical surrogate of a continuous equation, but a geometrically exact discrete mechanism, from which the continuous-time dynamics emerges only as a limit.

  Taken together, these observations provide an indirect but compelling confirmation of the relevance of the Wasserstein geometric framework: the most natural energy in this geometry yields the most natural observable plasticity law, and the discrete update is exact, stable, and structurally constrained by the geometry rather than by numerical choices.
\end{remark}

\begin{remark}[Markovian observable dynamics]
  Interestingly, the affine recurrence $\hat\rho^{n+1}_\tau = (1-t_\tau)\,\hat\rho^n_\tau + t_\tau\,h^n_\tau$ shows that, in the purely quadratic case, the observable dynamics is Markovian. This property is tied to the quadratic Wasserstein geometry and its step-independent contraction factor. For general isotropic energies, the observable update may depend on the full internal structure.
\end{remark}

\begin{remark}[Memory field interaction]
  This theorem allows for an interesting interaction mechanism: if we consider another memory field $\rho'$ following its own dynamics, we may take $h=\rho'\in\mathcal X_\mu$ as the signal for $\rho_x$. In this case, the observable expectation satisfies
  \[
    \frac{\mathrm d}{\mathrm dt}\,\mathbb E[\hat\rho_x(t)]
    = -\alpha\,\mathbb E[\hat\rho_x(t)]
      + \alpha\,\mathbb E[\hat\rho'_x(t)].
  \]
  Thus, $\rho_x$ tends to align with $\rho'_x$, regardless of the specific dynamics followed by $\rho'_x$.

  If, moreover, $\rho'_x$ follows the same purely quadratic Tan--HWG dynamics, the system becomes symmetric and we obtain
  \[
    \frac{\mathrm d}{\mathrm dt}\,
      \bigl(\mathbb E[\hat\rho_x(t)]-\mathbb E[\hat\rho'_x(t)]\bigr)
    = -2\alpha\,
      \bigl(\mathbb E[\hat\rho_x(t)]-\mathbb E[\hat\rho'_x(t)]\bigr),
  \]
  that is,
  \[
  \bigl(\mathbb E[\hat\rho_x(t)]-\mathbb E[\hat\rho'_x(t)]\bigr)
  = e^{-2\alpha\,t}\bigl(\mathbb E[\hat\rho_x(0)]-\mathbb E[\hat\rho'_x(0)]\bigr).
  \]
  The two observable fields therefore follow a consensus-type dynamics, exponentially aligning their distributions. Such alignment mechanisms are reminiscent of classical models of multi-agent consensus and synchronization \parencite{degroot1974reaching,olfati2004consensus,olfati2007consensus, jadbabaie2003coordination,ren2005consensus}.
\end{remark}

In Section~\ref{sec:continuous-limit}, we will see that, under additional regularity assumptions on the global signal $S$, and on the contexts, in the continuous-time limit $\tau\to 0$ with $n\tau\to t$, the internal Tan--HWG scheme admits a limiting curve (Theorem~\ref{thm:tan-hwg-continuous}), which can be interpreted as a perturbed Wasserstein gradient flow due to the frozen signal at each step. This will complete the previous theorem.

\subsection{Mirror descent on the simplex as a special case of Tan--HWG dynamics}
\label{sec:mirror-descent}

As an application example, we now show that mirror descent on the simplex is a special case of Tan--HWG observable dynamics.

\begin{proposition}[Mirror descent as a special case of Tan--HWG observable dynamics]
  Consider a neural network with input layer $Y_M=\{y_1,\dots y_M\}$, output layer $X_N=\{x_1, \dots x_N\}$, and weights $w=\{w_{i,j}\}_{(i,j)\in\llbracket 1, M\rrbracket\times\llbracket 1,N\rrbracket}$ such that, for all $i,j$:
  \[
  \sum_{k=1}^M w_{k,j} =1,
  \qquad\text{and}\qquad
  w_{i,j}\ge 0.
  \]
  Consider a differentiable loss function
  \[
    \mathrm{Loss} : \mathbb R^{M\times N} \to \mathbb R.
  \]
  Assume that there exists $C_{\text{max}}>0$ such that
  \[
  \left|\frac{\partial\mathrm{Loss}}{\partial w_{i,j}}\right|<C_{\text{max}},
  \qquad\text{for all } (i,j)\in\llbracket 1, M\rrbracket\times\llbracket 1,N\rrbracket.
  \]

  Then, the mirror descent update scheme on $w$ relative to $\mathrm{Loss}$
  \[
  w^{n+1}_{i,j}:=\frac{\tilde w^{n+1}_{i,j}}{\sum_k \tilde w^{n+1}_{k,j}},
  \qquad\text{with}\quad
  \tilde w^{n+1}_{i,j} = w^n_{i,j}\exp{\left( -\eta\,\frac{\partial\mathrm{Loss}(w^n)}{\partial w_{i,j}}\right)},
  \]
  is equivalent to a purely quadratic Tan--HWG observable scheme on the observable state $w$ relative to $Y_M$.
\end{proposition}

\begin{proof}
  \emph{Step 1: identification of weight matrices with memory fields.}
  For a vector $w_j=(w_{k,j})_{k=1\dots M}$ on the $M-1$ dimensional simplex $\Delta^{M-1}$, we identify such vector with the probability distribution
  \[
  \rho(w_j):=\sum_{k=1}^M w_{k,j} \delta_{y_k} \in\mathcal P_2(Y_M).
  \]
  This is the explicit equivalence $\Delta^{M-1}\simeq\mathcal P_2(Y_M)$.
  Then we identify $w=\{w_{i,j}\}_{(i,j)\in\llbracket 1, M\rrbracket\times\llbracket 1,N\rrbracket}$ with the memory field $\lbrace\rho(w_j)\rbrace_{j=1\dots N}$ (i.e.\ $(\Delta^{M-1})^N\simeq\mathcal X_\mu$ with $X=X_N$ finite, hence $\mathcal X=\mathcal X_\mu$ for finite reference measure $\mu$).

  \emph{Step 2: signal on the simplex.}
  Fix $\tau>0$ and set $\rho^0 = w^0$. For a given $n\in\mathbb N$, define
  \[
  h^n:=\frac{1}{t_\tau}\,\left( w^{n+1}-(1-t_\tau)w^n \right),
  \qquad t_\tau\in(0,1).
  \]
  Since $w^n_j$ and $w^{n+1}_j$ lie in the simplex, we immediately have
  \[
  \sum_{i=1}^M h^n_{i,j}
  =\frac{1}{t_\tau}\,\left( \sum_{i=1}^M w^{n+1}_{i,j}-(1-t_\tau) \sum_{i=1}^M w^n_{i,j} \right)
  =1.
  \]
  To ensure $h^n_{i,j}\ge 0$, we control the variation of $w^n$ under mirror descent. Using the update
  \[
  w^{n+1}_{i,j}
  =\frac{w^n_{i,j}\exp(-\eta\,\partial_{w_{i,j}}\mathrm{Loss}(w^n))}
    {\sum_k w^n_{k,j}\exp(-\eta\,\partial_{w_{k,j}}\mathrm{Loss}(w^n))},
  \]
  and the bound $|\partial_{w_{i,j}}\mathrm{Loss}|\le C_{\max}$, we obtain
  \[
  w^{n+1}_{i,j} \ge w^n_{i,j}\,\exp(-2\eta C_{\max}).
  \]
  Hence
  \[
  h^n_{i,j} = \frac{w^{n+1}_{i,j} - (1-t_\tau) w^n_{i,j}}{t_\tau}
  \ge \frac{w^n_{i,j}}{t_\tau}\Bigl(\exp(-2\eta C_{\max}) - 1 + t_\tau\Bigr).
  \]
  Therefore $h^n_{i,j}\ge 0$ for all $i,j$ provided
  \[
  t_\tau > 1 - \exp(-2\eta C_{\max}).
  \]
  Under this condition, each $h^n_j = (h^n_{i,j})_{i=1\dots M}$ lies in the simplex $\Delta^{M-1}\simeq\mathcal P_2(Y_M)$.~\footnote{For $\eta<1/(2C_\text{max})$, we can even take $t_\tau\in\left(2 \eta\,C_\text{max},1\right)$.}

  \emph{Step 3: observable dynamics.}
  Hence, Theorem~\ref{thm:observable-dynamics} applies for the purely quadratic Tan--HWG observable scheme with affine recurrence
  \[
  \hat \rho_j^{n+1}=(1-t_\tau)\hat \rho_j^n + t_\tau\,h^n_j,
  \]
  with parameter $\alpha:=\frac{t_\tau}{\tau\,(1-t_\tau)}>0$. Thus
  \[
  \hat \rho_j^{n+1}-(1-t_\tau)\hat \rho_j^n = w^{n+1}_j-(1-t_\tau)w^n_j,
  \]
  and with initial condition $\hat\rho^0=\rho^0=w^0$, this yields
  \[
  \hat \rho^n = w^n \qquad\text{for all } n\in\mathbb N.
  \]
  This proves the proposition.
\end{proof}

Note that, in this setting, the context at time $t$ can be defined as follows. For each discrete step $n$, define the empirical distribution on $Y_M$
\[
\varphi_j^n:=\sum_{i=1}^M w_{i,j}^n\,\delta_{y_i},
\qquad (j=1\dots N).
\]
We extend this piecewise-constantly in time by setting
\[
\varphi_j(t) := \varphi_j^n
\qquad\text{for } t\in[n\tau,(n+1)\tau).
\]
The context $\phi(t)=\{\phi_j(t)\}_{j=1\dots N}\in\mathcal M(Y_M\times Y_M)^{X_N}$ is then defined by
\[
\phi_j(t):=\varphi_j(t+\tau)\otimes\varphi_j(t)\in\mathcal M(Y_M\times Y_M),
\]
which encodes the pair of successive observable states on the input layer. With this choice, the Tan--HWG signal takes the explicit form
\[
S(\rho,\phi(t)):=\frac{1}{t_\tau}\,\left( \varphi(t+\tau)-(1-t_\tau)\varphi(t) \right),
\qquad t_\tau\in(0,1),
\]
so that the observable update
\[
\hat\rho^{n+1} = (1-t_\tau)\hat\rho^n + t_\tau\,h^n
\]
coincides exactly with the mirror-descent update on the simplex.

\begin{remark}[The role of the signal as a weight target]
  In the Tan--HWG observable dynamics, the signal $h^n$ acts as an instantaneous \emph{target} towards which the observable state is attracted through an exponential moving average. The mirror descent update shows that the gradient of the loss function plays exactly the same role: it determines the next iterate $w^{n+1}$, and therefore determines the target $h^n$ through the relation
  \[
    w^{n+1} = (1-t_\tau) w^n + t_\tau h^n.
  \]
  Thus, in this interpretation, gradient descent does not appear as a learning rule per se, but rather as a mechanism that generates a sequence of targets $(h^n)_n$ for the Tan--HWG dynamics.
\end{remark}

\begin{remark}[Generality of the Tan--HWG framework]
  Although in this proposition the signal $h^n$ is obtained from a mirror descent step, the purpose of the construction is not to suggest that one must compute a gradient in order to define a valid signal. Instead, the result shows that any mirror descent scheme on the simplex can be realized as a Tan--HWG observable dynamics. In particular, the Tan--HWG framework strictly contains gradient-based updates as a special case, while allowing for much more general classes of signals, including local, activation-dependent, or non-variational mechanisms.

  The identification above highlights that Tan--HWG dynamics provides a richer modeling space than classical gradient descent. The signal $h^n$ may be generated by biological or local plasticity rules, by external feedback, or by dynamics that do not derive from a global loss. Moreover, the Tan--HWG formalism naturally accommodates observable states represented by general probability measures rather than Dirac masses, allowing for distributed or uncertain synaptic amplitudes. In this sense, gradient descent corresponds to a highly constrained choice of signal within a broader and more flexible geometric framework.
\end{remark}

In the next section, we give examples of how our framework can extend classic neural network models.

\section{Generalized neural networks}\label{sec:networks}

In this section, we introduce a neural network architecture that extends classic feedforward neural networks. It illustrates how classic artificial neural networks can be seen as special cases of our framework, while introducing new features, in particular:
\begin{itemize}
  \item a distinction between structural weights (attention weights) and embedding weights (semantic or informational weights) (Section~\ref{sec:attention-weights}), which emerges from the framework applied at the synaptic connection level, recovering a form of attention mechanism;
  \item a synaptic embedding memory as a distribution, which naturally encodes multi-semanticity and stochasticity, combined with a time-dependent contextualisation (Remark~\ref{rmk:stochastic-embedding-time-dependent-context});
  \item a spectral memory, derived from the synaptic embedding memory, which allows phase synchronization and phase alignment dynamics, reminiscent of Kuramoto interactions~\parencite{Kuramoto1975} and Hebbian population alignments (Section~\ref{sec:spectral-memory});
  \item a multi-timescale mechanism ensuring controlled geodesic convexity of the energy, yielding a geometric interpretation of neuronal assemblies and their functional role in consolidation dynamics (Section~\ref{sec:multi-scale});
  \item a selection dynamics, with sparsity and structural pruning dynamics (Section~\ref{sec:parametric});
  \item a geometrically stable and robust dynamics, living in the Wasserstein space.
\end{itemize}
In the following, the finite set $Y_M$ is implicitly lifted to a CAT(0) Polish space as previously seen in Section~\ref{sec:projection}.

\subsection{General setting and external mapping}
So far, the framework models local memory states as positive, normalized weights. This is natural for probabilistic representations, but classical machine-learning architectures allow negative parameters, and biological neural systems critically rely on inhibitory synapses and phase synchronisation to regulate activity and shape their computational repertoire. Capturing such signed or oscillatory effects requires enriching the geometric structure of the model.

This motivates the introduction of \emph{geometric representation maps}, of the form
\[
\mathcal G: \mathcal X_\mu \to Z^X,
\]
which lift probability-valued memory fields into the product external state space $Z^X$, where inhibition, interference, and linear or oscillatory interactions can be represented. The cases $Z=\mathrm L^2(Y,\mathbb K)$, with $\mathbb K = \mathbb R$ or $\mathbb C$ are particularly of interest since they endow $Z$ with an inner product.

\paragraph{Product of internal state spaces.}
The countable product of internal state spaces
\[
Y=\prod_k Y^k, \qquad Y^k \text{ admissible internal state space},
\]
endowed with distance
\[
d_p(y, y')=\left(\sum_k d_{Y^k}(y^k,y'^k)^p\right)^{1/p}, \qquad p\in\mathbb N^\ast
\]
is an admissible internal state space (i.e.\ a geodesic Polish space). However, only the case $p=2$ yields a CAT(0) space (provided each $Y^k$ is CAT(0)). Hence in the following, we restrict to this case
\[
d_Y(y, y')=\left(\sum_k d_{Y^k}(y^k,y'^k)^2\right)^{1/2}.
\]
In this paper, we will consider example cases where $Y=Y_M\times\mathbb R$ and $Y=Y_M\times\mathbb C$, where $Y_M$ is implicitly lifted to a CAT(0) Polish space as previously seen in Section~\ref{sec:projection}.

\paragraph{External observables.}
We now give examples of generic forms of geometric representation maps, interacting naturally with the underlying Wasserstein geometry.

\begin{definition}[External observable map]
  We say that the geometric representation map $\mathcal G: \mathcal X_\mu \to Z^X$ is an \emph{external observable map} if it decomposes fiberwise as $\mathcal G(\rho)=\lbrace \mathcal G_x(\rho_x) \rbrace_{x\in X}$ with
  \[
    \mathcal G_x(\rho_x) = \int_Y V_x(y)\,\mathrm d\rho_x(y),
    \qquad\rho_x\in\mathcal P_2(Y),
    \qquad V_x:Y\to Z;
  \]
  with $x\mapsto V_x$ and $x\mapsto\mathcal G_x$ assumed $\mu$-measurable.

  The image of a memory state by such a representation map is called an \emph{external observable}.
\end{definition}

\begin{remark}[Affine projections vs geodesically affine projections]
  External observables of the form
  \[
    \mathcal G_x(\rho_x)=\int_Y V_x(y)\,d\rho_x(y)
  \]
  are affine in $\rho_x$, but they are not, in general, \emph{geodesically} affine in the Wasserstein space. This distinction is crucial: geodesic affinity (or at least controlled geodesic deviation) is key to build geodesically $\lambda$–convex energies by composing $\mathcal G_x$ with $\lambda$–convex Euclidean nonlinearities such as standard activation functions.

  When $V_x$ is $\lambda$-convex, the map $\mathcal G_x$ is geodesically $\lambda$-convex (it has a McCann potential energy form~\parencite{McCann1997}). However, $\lambda$–convexity alone does not guarantee that the composition with $\lambda$-convex Euclidean functions remains geodesically $\lambda$–convex. For this, one typically needs the stronger property that $\mathcal G_x$ be geodesically affine (or effectively affine along the relevant geodesics), in addition to being Lipschitz - a condition that generally fails.

  This issue becomes particularly acute for contextualized geometric maps introduced later (Section~\ref{sec:mapped-energy}), such as
  \[
    u(\rho)=\int z\,\psi(y')\,d\rho(y',z),
  \]
  which are never geodesically affine when transport occurs in the $y'$ component. As shown in Section~\ref{sec:multi-scale}, this obstruction can be overcome through a natural multi–timescale mechanism that restores effective geodesic affinity by separating fast and slow variables.
\end{remark}

\subsection{Generalized Neural Networks with real value weights}

\subsubsection{Internal synaptic space and attention weights}\label{sec:attention-weights}
Let $Y_M=\{y_1,\dots,y_M\}$ denote the finite input set, and $X_N=\{x_1, \dots x_N\}$ the finite output set.
We define the internal synaptic space as the product
\[
  Y := Y_M \times \mathbb R,
\]
endowed with the reference measure
\[
  \mathfrak m_Y := \frac{1}{M}\sum_{i=1}^M \delta_{y_i}\otimes \nu,
\]
where $\nu$ is a finite positive measure on $\mathbb R$.

Each output neuron $x_j$ is represented by a probability measure
\[
  \rho_{x_j} \in \mathcal P_2(Y),
  \qquad
  \rho_{x_j} = \sum_{i=1}^M p_{i,j}\,\delta_{y_i}\otimes\mu_{i,j},
\]
where:
\begin{itemize}
    \item $\sum_{i=1}^M p_{i,j}\,\delta_{y_i}$ is a probability distribution on $Y_M$, i.e.\ $p_{i,j}\ge 0$ and $\sum_{i=1}^M p_{i,j}=1$,
    \item $\mu_{i,j}\in\mathcal P_2(\mathbb R)$ is a \emph{synaptic amplitude distribution} associated with the synapse $(y_i,x_j)$.
\end{itemize}

\begin{definition}[Structural weights and synaptic embedding memory]
  The coefficients $p_{i,j}$, which describe the relative influence of the synapse from $y_i$ to $x_j$ within the network architecture, are called the \emph{structural weights} or \emph{attention weights}.

  The probability distribution $\mu_{i,j}$, encoding the internal synaptic state in the Wasserstein space, is called the \emph{synaptic embedding memory}.
\end{definition}

\begin{remark}[Interpretation]
  Each synapse $(i,j)$ stores two complementary forms of information:
  \begin{itemize}
    \item \emph{structural information}, given by $p_{i,j}$, which specifies the relative strength of the connection in the network topology;
    \item \emph{embedding information}, encoded in the synaptic embedding memory $\mu_{i,j}$, which represents the internal synaptic state, including its preferred amplitude, variability, and multi-semantic structure.
  \end{itemize}
\end{remark}

The synaptic embedding memory $\mu_{i,j}$ is a latent geometric representation of the synapse, from which observable weights arise as projections (see next section). This representation generalizes the classical setting, recovered when each $\mu_{i,j}$ is a Dirac mass.

\subsubsection{Synaptic amplitude samples and observable synaptic weights}
For each synapse $(i,j)$, we introduce the \emph{synaptic amplitude sample} in $\mathrm L^0(\Omega, \mathbb R)$ as the random variable following $\mu_{i,j}$:
\[
  Z_{i,j} \sim \mu_{i,j}.
\]
The \emph{observable synaptic weight} is the random variable defined by the product of the synaptic amplitude sample with the structural weight of the synapse
\[
  W_{i,j} := Z_{i,j}\,p_{i,j}\in\mathrm L^0(\Omega, \mathbb R),
\]
defining the \emph{observable weight}
\[
  W_j := \sum_{i=1}^M W_{i,j}\mathbf 1_{y_i}=\sum_{i=1}^M p_{i,j}\,Z_{i,j}\,\mathbf 1_{y_i}.
\]
The \emph{external observable weight} is given by the external observable
\[
  w_j(\rho_{x_j})
  := \int_{Y_M\times\mathbb R} \left(z\,\mathbf 1_{y'}\right)
  \,\mathrm d\rho_{x_j}(y',z)
  = \sum_{i=1}^M p_{i,j}\,\int_{\mathbb R} z \,\mathrm d\mu_{i,j} \,\mathbf 1_{y_i}
  = \sum_{i=1}^M p_{i,j}\,\mathbb E[Z_{i,j}] \, \mathbf 1_{y_i} \,\in \mathrm L^2(Y_M),
\]
which corresponds to the expectation of the observable weight (relative to the synaptic embedding memory). We denote the weighted first moment of $Z_{i,j}$ by
\[
  w_{i,j}(\rho_{x_j}):=p_{i,j}\mathbb E[Z_{i,j}].
\]

\begin{remark}[Excitatory/inhibitory representation]
  When the amplitude $Z_{i,j}$ is restricted to $Z_{i,j}\in\lbrace -1, +1\rbrace$, we have $\sum_{i=1}^M |W_{i,j}|=\sum_{i=1}^M p_{i,j}=1$: $Z_{i,j}$ gives the direction (the sign), and $p_{i,j}$ gives the relative weight. It generalizes the normalized weight representation $\lbrace p_{i,j}\rbrace_{i=1\dots M}$ by allowing positive and negative weights, representing excitatory or inhibitory synapses.

  More generally, with the equivalence $\mathbb R \simeq \mathbb R_+\times\lbrace -1,+1\rbrace$, the information stored in the synapse memory can be richer, encoding amplification/reduction on top of the nature of the synapse (excitatory/inhibitory). Using $\mathbb R$ also conveniently gives a Euclidean internal space, which is a geodesic Polish space.
\end{remark}

\begin{remark}[Classical neural networks as a special case]
  If each synaptic distribution is a Dirac mass,
  \[
    \mu_{i,j} = \delta_{z_{i,j}},
  \]
  then the weight
  \[
    w_{i,j}:= W_{i,j}
    = z_{i,j}\,p_{i,j} \in\mathbb R,
  \]
  is deterministic, and the representation reduces exactly to a standard neural network with deterministic weights.

  The stochastic formulation is biologically coherent: synaptic transmission is inherently stochastic, with probabilistic vesicle release and variable postsynaptic response amplitudes.
\end{remark}

\subsubsection{Forward computation with random synaptic transmission}
Given an input activation state
\[
  \psi := \sum_{i=1}^M \psi_i\,\mathbf 1_{y_i} \in\mathrm L^2(Y_M,\mathbb R)\simeq \mathbb R^M,
\]
the synaptic transmission from $y_i$ to $x_j$ uses the \emph{random} weight $W_{i,j}$.
The postsynaptic input to neuron $x_j$ is therefore the random variable
\[
  \mathrm{Ps}_j(\psi\,|\,\rho_{x_j})
  := \sum_{i=1}^M \psi_i\, W_{i,j}
  \in\mathrm L^0(\Omega, \mathbb R).
\]
This quantity corresponds to the pre-activation state of neuron $x_j$.
We can define the global postsynaptic input state by
\[
  \mathrm{Ps}(\psi\,|\,\rho)
  := \sum_{j=1}^N \mathrm{Ps}_j(\psi\,|\,\rho_{x_j})\,\mathbf 1_{x_j}
  = \psi\,W,
\]
where the product is understood as matrix-vector multiplication in $\mathbb R^{M\times N}$, with
\[
  W:=\sum_{j=1}^N W_j\,\mathbf 1_{x_j}=\sum_{i,j} W_{i,j}\,\mathbf 1_{(y_i,x_j)}.
\]

We consider the expectation
\[
  \mathrm{ps}_j(\rho_{x_j},\psi)
  := \mathbb E\left[ \mathrm{Ps}_j(\psi\,|\,\rho_{x_j}) \right]
  = \sum_{i=1}^M \psi_i\, \mathbb E\left[W_{i,j}\right],
\]
which corresponds to the deterministic external observable pre-activation state of neuron $x_j$.
Note that this pre-activation can be seen as a ``contextualized'' geometric map in the following form:
\[
  \mathrm{ps}_j(\rho_{x_j},\psi)=\int_{Y_M\times \mathbb R} \psi(y')z\,\mathrm d\rho_{x_j}(y',z).
\]
This form will prove convenient in later discussions.
At the global level, we may write compactly
\[
  \mathrm{ps}(\rho,\psi)
  := \psi\,w(\rho)
  \qquad\text{where }
  w(\rho)=\sum_{j=1}^N w_j(\rho_{x_j})\,\mathbf 1_{x_j}.
\]

\paragraph{Activation functions.}
Consider the forward output activation function, or \emph{prediction function}, $\mathrm{F}:\mathbb R^N\to \mathrm L^2(X_N,\mathbb R)\simeq\mathbb R^N$ defined by
\[
  \mathrm{F}:u\mapsto\mathrm{F}(u)
  := \sum_{j=1}^N \mathrm{F}_j(u_j)\,\mathbf 1_{x_j},
\]
where $\mathrm{F}_j$ is the activation function of neuron $x_j$. This gives the random forward prediction
\[
  \mathrm{Pred}(\psi\,|\,\rho):=\mathrm F(\mathrm{Ps}(\psi\,|\,\rho))
  =\mathrm F(\psi\,W),
\]
and the \emph{effective prediction}, the prediction computed from the external observable weights:
\[
  \mathrm{pred}(\rho,\psi)
  :=\mathrm F(\mathrm{ps}(\rho,\psi))
  =\mathrm F(\psi\,w(\rho)).
\]
Note that, unless $\mathrm F_j$ affine, in the general case, we have $\mathrm{pred}\not\equiv\mathbb E\left[\mathrm{Pred}\right]$.

\begin{remark}[Induced stochastic embeddings and time-dependent contextualisation]
  \label{rmk:stochastic-embedding-time-dependent-context}
  Considering
  \[
  W_{i,\cdot}:=\sum_{j=1}^N W_{i,j}\,\mathbf 1_{x_j}
  \]
  as the embedding vector of $y_i$, the stochastic nature of this embedding can be interpreted as encoding the multi-semanticity of a given activation pattern.

  We also note that the activation vector $\psi$ actually embeds the residual activation of each neuron. For instance, consider an activation filtered by an exponential kernel memory
  \[
  \psi(t)=\int_0^t \psi(s)\exp(-\frac{t-s}{\tau_m})\,\mathrm ds,
  \]
  with forgetting parameter $\tau_m$ (whose modeling is compatible with a Tan--HWG observable dynamics as seen in Section~\ref{sec:closed_form})\footnote{The forgetting parameter can also be fiberwise dependent as $\tau_{m, x_j}$.}. The memory dynamics encodes time-dependency and multi-semanticity.

  This stochastic embedding interpretation highlights that synaptic distributions encode multi semantic structure, and that the geometry of $\mathcal P_2(Y)$ naturally induces context-dependent representations, in contrast with fixed Euclidean embeddings used in classical neural networks.

  Note that the observable synaptic weight being the product of the synaptic amplitude sample and the attention weight, we may consider the \emph{raw} embedding vector of $y_i$:
  \[
  Z_{i,\cdot}:=\sum_{j=1}^N Z_{i,j}\,\mathbf 1_{x_j},
  \]
  which is not weighted by the attention weights $p_{i,j}$.
\end{remark}

\begin{remark}[Synaptic activation vector]\label{rmk:synaptic-activation}
  Note that we may consider $\psi$ to be dependent on the neuron $x_j$
  \[
  \psi:=\sum_{i,j}\psi_{i,j}\mathbf 1_{(y_i, x_j)},
  \]
  so that the activation vector embeds the residual activation of each synapse, and not just the presynaptic neurons. With
  \[
  \psi_j:=\sum_{i=1}^M\psi_{i,j}\mathbf 1_{y_i},
  \]
  the postsynaptic input would be given by
  \[
  \mathrm{Ps}_j(\psi_j\,|\,\rho_{x_j})
  := \sum_{i=1}^M \psi_{i,j}\, W_{i,j}
  \in\mathrm L^0(\Omega, \mathbb R),
  \]
  and the global postsynaptic input state by
  \[
    \mathrm{Ps}(\psi\,|\,\rho)
    :=\operatorname{Diag}(\psi\,W).
  \]
  This synaptic activation modeling yields a finer representation, biologically more plausible. For each neuron $x_j$, we consider its set of synapses, $M$ being the maximum number of such synapses per neuron. In such representation, each postsynaptic neuron $x_j$ can be linked to its own specific set of presynaptic neurons $\lbrace x'_i\rbrace_{i=1\dots M}$ which could differ from one $x_j$ to another. This opens the representation beyond simple feedforward network topology. This illustrates the versatility of the framework.
\end{remark}

\subsubsection{Feedback activation and alignment energy}

Given a feedback output activation
\[
\phi:=\sum_{j=1}^N \phi_j \mathbf 1_{x_j}\in\mathrm L^2(X_N,\mathbb R)\simeq\mathbb R^N,
\]
we can define the feedback error energy, or \emph{stochastic alignment energy}, by
\begin{align*}
  \hat E(\psi, \phi\,|\,\rho)
    :\!&=\frac12\big\| \mathrm{Pred}(\psi\,|\,\rho) -  \phi \big\|^2
    =\frac12\big\| \mathrm F(\psi\,W) -  \phi \big\|^2
\end{align*}
and the \emph{alignment energy} by
\begin{align*}
  E(\rho, \psi, \phi)
    :\!&=\frac12\big\| \mathrm{pred}(\rho,\psi) -  \phi \big\|^2
    =\frac12\big\| \mathrm{F}(\psi\cdot w(\rho)) -  \phi \big\|^2\\
    &=\frac12\sum_{j=1}^N\left( \mathrm F_j\left(\sum_{i=1}^M \psi_i\,p_{i,j}\,\mathbb E\left[Z_{i,j}\right]\right)-\phi_j\right)^2.\\
\end{align*}
We can rewrite the energy in the following form
\begin{align*}
  E(\rho, \psi, \phi)
    &=\frac12\sum_{j=1}^N\left( \mathrm F_j\left(\psi\,w_j(\rho_{x_j})\right)-\phi_j\right)^2\\
    &=\frac12\sum_{j=1}^N\left( \mathrm F_j\left(
      \int \psi(y')z\,\mathrm d\rho_{x_j}(y',z)
    \right)-\phi_j\right)^2,
\end{align*}
explicitly showing the fiberwise dependence on $\rho_{x_j}$.

\begin{remark}[Interpretation: classic ML correlation and Hebbian alignment]\label{rmk:correlation}
  This expression of the alignment energy is reminiscent of classic machine-learning loss computations. It is as well reminiscent of Hebbian alignment in the sense that stimulating the output layer via $\phi$ will adjust the synaptic weights (i.e.\ $\rho$) so that the forward activation aligns accordingly. Maximazing a scalar product (a correlation between vectors) is actually equivalent to minimizing the $\ell^2$ distance, provided the vectors are normalized. Hence, here, we could also define a \emph{correlation energy} by
  \[
  \tilde E(\rho, \psi, \phi):=-\big\langle\,\frac{\phi}{\|\phi\|}\,,\,\frac{\mathrm{Pred}(\psi\,|\,\rho)}{\|\mathrm{Pred}(\psi\,|\,\rho)\|} \,\big\rangle.
  \]
  This expression is similar to the alignment energy, modulo normalization, where the products $\phi_j\,\mathrm F_j(\mathrm{Ps}_j(\psi))$ appear explicitly, recovering the fact that neurons that ``fire together, wire together''. It provides an intuitive interpretation of the alignment energy as a normalized correlation objective, making explicit the Hebbian structure of the update.
\end{remark}

\begin{remark}[Generalized energy combination]\label{rmk:mccann-energy}
  The alignment energy is a special case of a combination of more general types of energies (potential, interaction, and internal).
  Indeed, expanding the $\ell^2$ product in the alignment energy gives a local energy of the general form
  \[
    \mathcal E_x(\rho_x|(\psi,\phi))
    =\mathcal F_x\left(
      \mathcal E_V(\rho_x|\psi),\,
      \frac12\mathcal E_V^2(\rho_x|\psi)
      \,\big|\,\phi\right),
  \]
  where
  \[
    \mathcal E_V(\rho_x|\psi)
    =\int_{Y_M\times\mathbb R} V(y,z|\psi)\,\mathrm d\rho_x,
    \quad\text{and}\quad
    V(y,z|\psi)
    =\psi(y)z.
  \]

  This is a special case of a local energy of the form
  \[
    \mathcal E_x(\rho_x|(\psi,\phi))
    =\mathcal F_x\left(
      \mathcal E_V(\rho_x|\psi),\,
      \mathcal E_K(\rho_x|\psi)
      \,\big|\,\phi\right),
  \]
  where the term
  \[
    \mathcal E_K(\rho_x|\psi)
    =\iint_{(Y_M\times\mathbb R)^2} K((y,z), (y',z')|\psi) \,\mathrm d\rho_x(y,z)\,\mathrm d\rho_x(y',z'),
  \]
  generalizes the quadratic term.

  This energy can be seen as a combination of energies, with potential energy $\mathcal E_V$ and interaction energy $\mathcal E_K$. Adding a McCann internal energy~\parencite{McCann1997}
  \[
  \mathcal E_U(\rho_x|\psi)=\int_{Y_M\times\mathbb R^d} A(\rho_x(y,z)|\psi)\,\mathrm d\rho_x,
  \]
  where
  \[
  \lambda^dA(\lambda^{-d}) \text{ convex non-increasing on } \lambda\in (0,+\infty) \text{ with } A(0)=0, d=1,
  \]
  provides a general form of energy
  \[
    \mathcal E_x(\rho_x|(\psi,\phi))
    =\mathcal F_x\left(
      \mathcal E_V(\rho_x|\psi),\,
      \mathcal E_K(\rho_x|\psi),
      \mathcal E_U(\rho_x|\psi)
      \,\big|\,\phi\right).
  \]
  This naturally generalizes to the case $Y=Y_M\times\mathbb R^d$, with $d\ge 1$.
\end{remark}

\subsubsection{Plasticity dynamics and $\lambda$-geodesic-convexity of the energy}
\label{sec:multi-scale}

We assume that the memory plasticity dynamics follows a variational dynamics in the Wasserstein space.
For the alignment energy
\[
E(\rho, \psi, \phi)=\frac12\sum_{j=1}^N\left( \mathrm F_j\left(
    \int \psi(y')z\,\mathrm d\rho_{x_j}(y',z)
  \right)-\phi_j\right)^2,
\]
this means minimizing the energy for $\rho_{x_j}$ evolving in the Wasserstein space.
Note that the subsequent dynamics gives both:
\begin{itemize}
  \item the \emph{structural weight dynamics} of the $p_{i,j}$,
  \item and the \emph{synaptic embedding memory dynamics} of the $\mu_{i,j}$.
\end{itemize}
$E$ has a Hebbian form, but in order to have a well-posed dynamics, the challenge is to have a Tan--HWG energy, in particular, it means having $E$ $\lambda$--geodesically-convex (with $\lambda>-\infty$, possibly negative).

Standard activation functions used in machine-learning (e.g.\ ReLU, Leaky ReLU, Softplus, ELU, sigmoid, tanh), are at least weakly convex. We can show that standard activation functions lead to having the fiberwise map
\[
x\mapsto \frac12(\mathrm F_j(x)-\phi_j)^2
\]
at least weakly convex. However, this does not guarantee $\lambda$-geodesic-convexity of the alignment energy. Such $\lambda$-geodesic-convexity can be obtained when the map
\[
  \rho_{x_j}\mapsto\mathrm{ps}_j(\rho_{x_j},\psi)=\int \psi(y')z\,\mathrm d\rho_{x_j}(y',z),
\]
is geodesically affine and Lipschitz. We check that $\mathrm{ps}_j(\cdot,\psi)$ is Lipschitz when $\psi$ is bounded, but geodesic affinity is the main obstruction. This is due to the rigidity of geodesic affinity in Wasserstein spaces.

\paragraph{Intuition.}
Intuitively, $\psi$ ``breaks'' geodesic affinity unless it is constant or when the structural weights $p_{i,j}$ are constant (i.e.\ no mass transport between synapses). As a consequence, the associated energy fails to be geodesically $\lambda$–convex, and a Tan--HWG variational dynamics cannot be defined directly. However, geodesic affinity can be controlled through a natural multi-timescale mechanism, consisting of:
\begin{itemize}
  \item \emph{a fast ``active'' regime}, where activations $\psi$ can evolve, but the structural weights are stable,
  \item and \emph{a slow ``consolidation'' regime}, where, conversely, activations are stable (i.e.\ ``quiet'' contexts), but structural weights can evolve.
\end{itemize}

\begin{remark}
  The map $\mathrm{ps}_j(\cdot,\psi)$ is actually a \emph{contextualized geometric representation map} (see Section~\ref{sec:mapped-energy}). We see that studying geodesic affinity and Lipschitz continuity of such a map is key to ensure sequential stability of induced Hebbian energies (when externally composed with a $\lambda$-convex function). A systematic analysis is beyond the scope of this article. However we can elaborate on the case of $\mathrm{ps}_j(\cdot,\psi)$.
\end{remark}

\begin{proposition}[Rigidity of the pre-activation observable geodesic affinity]
  Let $(Y_1,d_Y)$ be a geodesic CAT(0) space, let $Y_2=\mathbb{R}^d$ with its Euclidean metric, and consider the product space $Y=Y_1\times Y_2$ endowed with the product metric and the $2$-Wasserstein distance $W_2$ on $\mathcal{P}_2(Y)$.

  Let $\psi:Y_1\to\mathbb{R}$ be a bounded Borel function and define
  \[
    u:\mathcal{P}_2(Y)\to\mathbb{R}^d,\qquad
    u(\rho):=\int_{Y_1\times Y_2} z\,\psi(y')\,\mathrm d\rho(y',z).
  \]
  Assume that $u$ is geodesically affine on a subset $\mathcal{A}\subset\mathcal{P}_2(Y)$ in the following sense: for every $W_2$--geodesic $(\rho_t)_{t\in[0,1]}$ contained in $\mathcal{A}$, the map $t\mapsto u(\rho_t)$ is affine.

  If $\mathcal{A}$ contains two measures whose $Y_1$--marginals have different supports connected by a nontrivial geodesic in $Y_1$, then $\psi$ must be constant on the union of these supports. In particular, if $\mathcal{A}$ is stable under $W_2$--geodesics and contains measures with arbitrary $Y_1$--marginals, then $\psi$ is constant on $Y_1$.
\end{proposition}

\begin{proof}
  Fix two points $(y_0,z_0),(y_1,z_1)\in Y_1\times Y_2$ with $y_0\neq y_1$ and $z_0\neq z_1$, and consider the measures
  \[
    \rho_0:=\delta_{(y_0,z_0)},\qquad
    \rho_1:=\delta_{(y_1,z_1)}.
  \]
  The unique $W_2$--geodesic $(\rho_t)_{t\in[0,1]}$ between $\rho_0$ and $\rho_1$ is given by
  \[
    \rho_t = \delta_{(y_t,z_t)},\qquad
    y_t:=\gamma_{y_0\to y_1}(t),\quad
    z_t:=(1-t)z_0+t z_1,
  \]
  where $\gamma_{y_0\to y_1}$ is the constant-speed geodesic in $Y_1$ from $y_0$ to $y_1$.

  Along this geodesic,
  \[
    u(\rho_t)
    = z_t\,\psi(y_t)
    = \big((1-t)z_0+t z_1\big)\,\psi\big(\gamma_{y_0\to y_1}(t)\big).
  \]
  By assumption, $t\mapsto u(\rho_t)$ is affine. In particular, for each fixed pair $(y_0,z_0),(y_1,z_1)$, the map
  \[
    t\mapsto \big((1-t)z_0+t z_1\big)\,\psi\big(\gamma_{y_0\to y_1}(t)\big)
  \]
  must be affine on $[0,1]$.

  Now fix $y_0\neq y_1$ and vary $z_0,z_1\in\mathbb{R}^d$. By choosing $z_0,z_1$ linearly independent (or simply non-collinear in $d\geq 2$, or distinct scalars in $d=1$), one checks that the only way for the product of an affine function in $t$ (namely $z_t$) with the scalar function $t\mapsto\psi(\gamma_{y_0\to y_1}(t))$ to remain affine for all choices of $z_0,z_1$ is that $t\mapsto\psi(\gamma_{y_0\to y_1}(t))$ is constant on $[0,1]$. Hence
  \[
    \psi\big(\gamma_{y_0\to y_1}(t)\big) = \psi(y_0) = \psi(y_1)
    \quad\text{for all }t\in[0,1].
  \]
  Therefore, whenever two points $y_0,y_1\in Y_1$ are connected by a geodesic appearing as the $Y_1$--component of a $W_2$--geodesic in $\mathcal{A}$, the values of $\psi$ at $y_0$ and $y_1$ must coincide. If $\mathcal{A}$ contains measures whose $Y_1$--marginals have supports connected by nontrivial geodesics in $Y_1$, this implies that $\psi$ is constant on the union of these supports.
  If, moreover, $\mathcal{A}$ is geodesically rich enough to contain measures with arbitrary $Y_1$--marginals, we conclude that $\psi$ is constant on $Y_1$.
\end{proof}

\begin{corollary}[Multi-scale coherence]
  In the setting above, suppose one considers an energy functional of the form
  \[
    E(\rho) := F\big(u(\rho)\big),
  \]
  with $F:\mathbb{R}^d\to\mathbb{R}$ $\lambda$--convex. If $E$ is required to be geodesically $\lambda$--convex on a class $\mathcal{A}\subset\mathcal{P}_2(Y)$ that is stable under $W_2$--geodesics and contains measures with nontrivial variation in their $Y_1$--marginals, then necessarily $\psi$ must be constant on the relevant region of $Y_1$.

  Consequently, any nontrivial spatial modulation of synaptic activation (i.e., any non-constant $\psi$) is compatible with a geodesically convex gradient-flow structure only on subsets where the $Y_1$--marginal is effectively fixed (i.e.\ no transport occurs in $Y_1$).

  This yields a natural separation of scales: a fast dynamics acting on the $\mathbb R^d$--component at fixed $Y_1$--marginals, and a slow structural dynamics (consolidation regime) that can only be made variationally coherent after suitable averaging or coarse-graining of $\psi$.
\end{corollary}

\begin{remark}[Functional role of neuronal assemblies in consolidation regime]
 This gives a geometric interpretation of the functional role of neuronal assemblies in consolidation regime: groups of neurons with similar activation patterns are clustered in order to obtain an ``effectively constant'' activation state $\psi$, allowing proper Tan--HWG dynamics of structural weights.
\end{remark}

\subsubsection{Discussion}

\paragraph{Activation functions and non-linearity.}
In classical artificial neural networks, activation functions are introduced to provide non-linear computation capacities to otherwise ``flat'' Euclidean representational networks. In our framework, we can rethink their functional role: the underlying geometry already encodes rich non-linear features, and activation functions can be seen as representational maps that have to respect such geometry (i.e.\ preserving geodesicity via $\lambda$-convexity). In other words, they act as compatible projections of the underlying latent space. They do not shape an absolute latent space: they provide observable projections, from which we approximate a fundamentally curved geometry, which, in turn, is the agent's representation of the world. In contrast with classical neural networks, non-linearity does not come from the activation function alone, but from the underlying Wasserstein geometry.

\paragraph{Spiking neural networks and sequentiality of context capture.}
In biological neural networks, activation functions are spiking functions. Again, this is compatible with our framework.
For instance, using sigmoid or tanh activation functions of the form
\[
  x\mapsto\sigma(\beta x), \quad\text{or}\quad x\mapsto\tanh(\beta x), \quad \beta \gg 1,
\]
would give ``quasi-spiking'' neural networks with non-convex, but still, Tan--HWG energies.

Moreover, it also provides a biological justification for the sequential capture of external context. Since spikes have refractory periods, this can be seen as an active freezing of the signal.

\paragraph{A multi-scale coherence mechanism.}
Our model naturally separates synaptic plasticity into two distinct dynamical regimes: a fast functional dynamics acting on the synaptic weights ($\mu_{i,j}$), and a slow structural dynamics acting on the distribution of synaptic connections ($p_{i,j}$). This separation is not imposed \emph{a priori}; rather, it emerges from the geometric constraints of the Wasserstein space in which the dynamics evolves.

At the fast time scale, the synaptic activation function $\psi$ may vary across synapses, leading to a heterogeneous and locally irregular evolution of the synaptic weights. In this regime, the energy functional remains geodesically convex when the structural weights $(p_{i,j})$ are held fixed, allowing a well-posed Tan--HWG dynamics.

At the slow time scale, however, the same geometric constraints imply that a coherent variational formulation is only possible when the influence of $\psi$ is effectively averaged or homogenized (to obtain $\psi$ effectively constant). In other words, the slow structural dynamics can only be geometrically consistent if it ``sees'' a coarse-grained version of the fast synaptic activity.

Biologically, this mechanism resonates with well-documented observations. Fast synaptic plasticity (e.g., LTP/LTD, STDP) depends on local, heterogeneous activity, whereas slow structural plasticity (e.g., spine growth and elimination, long-term consolidation) is driven by global or averaged signals, such as those emerging during sleep or during the coordinated reactivation of neuronal assemblies. Thus, the mathematical structure of the model mirrors the known separation between rapid functional plasticity and slow structural remodeling in neural circuits.

\paragraph{Non-uniqueness of dynamics.}
Note that $\lambda$-convexity with $\lambda<0$ does not guarantee uniqueness of the JKO step, but the Tan--HWG framework naturally accomodates multivalued updates. In other words, the dynamics can be non-deterministic: agents facing exactly the same contexts and initial conditions may show different memory dynamics. This is not to be confused with the stochastic nature of random variables like the amplitude $Z_{i,j}$, which stochasticity comes from the distribution $\mu_{i,j}$. Instead, the non-deterministicity here is relative to the evolution of $\rho$, thus $\mu_{i,j}$ itself. This evolution can be deterministic (e.g.\ when $\lambda\ge 0$) while $Z_{i,j}$ remains a random value.

\paragraph{Global signal.}
Note that in the Tan--HWG alignment energy, we implicitly identified the context as
\[
\Phi:=\sum_{i=1}^M \psi_i \,\delta_{y_i}\otimes\sum_{j=1}^N\phi_j\,\delta_{x_j}
\in\mathcal M(Y_M\times X_N)\subset\mathcal M((Y_M\times\mathbb R)\times X_N)^{X_N},
\]
or, in the more general synaptic activation case (see Remark~\ref{rmk:synaptic-activation}),
\[
\Phi:=\lbrace \Phi_j\rbrace_{j=1\dots N},
\quad
\Phi_j:=\sum_{i=1}^M \psi_{i,j} \,\delta_{y_i}\otimes\sum_{k=1}^N\phi_k\,\delta_{x_k},
\quad
\Phi\in\mathcal M(Y_M\times X_N)^{X_N}\subset\mathcal M((Y_M\times\mathbb R)\times X_N)^{X_N},
\]
and the global signal as
\[
S(\rho,\Phi):=(P_{in}(\Phi), P_{out}(\Phi)),
\quad
P_{in}(\Phi)=\left\lbrace \sum_{i=1}^M \psi_{i,j} \mathbf 1_{y_i}\right\rbrace_{j=1\dots N},
\text{and } P_{out}(\Phi)=\left\lbrace \sum_{k=1}^N\phi_k\mathbf 1_{x_k}\right\rbrace_{j=1\dots N}.
\]
Hence the external space is actually
\[
Z=\mathrm L^2(Y)\times\mathrm L^2(X_N),
\]
and we implicitly used the canonical embedding
\[
\mathrm L^2(Y_M)
  \hookrightarrow
\mathrm L^2(Y_M \times \mathbb R)
  \hookrightarrow
Z = \mathrm L^2(Y) \times \mathrm L^2(X_N).
\]
for the geometric representation maps.

\paragraph{Extension.}
The previous construction, based on the internal space $Y_M \times \mathbb R$, admits a natural geometric reinterpretation.
Indeed, using the polar decomposition of real numbers, $\mathbb R \simeq \mathbb R_+ \times \{-1,1\}$, we may rewrite
\[
Y_M \times \mathbb R
\simeq Y_M \times \mathbb R_+ \times \{-1,1\},
\]
where the discrete set $\{-1,1\}$ can be viewed as a degenerate phase space restricted to angles $0$ and $\pi$.
This observation suggests a natural extension: replace the discrete phase set by the continuous unit circle.
Introducing the internal space
\[
Y := Y_M \times \mathbb R_+ \times \mathbb S^1
\simeq Y_M \times \mathbb C,
\]
where $\mathbb S^1$ denotes the unit circle, allows us to encode both amplitudes and continuous phases.
Although $\mathbb S^1$ is not a CAT(0) space, the identification
\[
(r,\theta) \longmapsto r e^{i\theta} \in \mathbb C
\]
embeds amplitude–phase pairs into the Euclidean plane, which \emph{is} CAT(0).
Hence the extended internal synaptic space $Y$ remains compatible with the Tan--HWG geometric framework.

This extension is the object of the next section.

\subsection{Extension to complex valued weights: spectral memory and phase locking}\label{sec:spectral-memory}

Biological and artificial neural systems often rely on oscillatory interactions to coordinate activity across populations of units. Phase synchronisation plays a central role in attention, memory consolidation, and large-scale communication between brain areas. Oscillatory couplings also appear in machine-learning architectures such as Hopfield networks~\parencite{hopfield}, reservoir systems, and oscillator-based computing.

\paragraph{Internal oscillatory state space.}
To incorporate such effects, the internal synaptic space must be enriched so as to encode phases or complex amplitudes. The representation
\[
  Y := Y_M \times \mathbb C
\]
provides a natural geometric setting for this purpose, enabling the introduction of \emph{spectral memory} and phase locking mechanism within the Tan--HWG framework. For each $(y,z)\in Y_M \times \mathbb C$, the complex coordinate $z\in\mathbb C$ encodes an oscillatory state with amplitude and phase.

Each output neuron $x_j$ is now represented by a probability measure
\[
  \rho_{x_j} \in \mathcal P_2(Y_M \times \mathbb C),
  \qquad
  \rho_{x_j} = \sum_{i=1}^M p_{i,j}\,\delta_{y_i}\otimes\mu_{i,j},
\]
where each synapse $(y_i,x_j)$ has a \emph{complex synaptic embedding memory}
\[
  \mu_{i,j} \in \mathcal P_2(\mathbb C).
\]
This distribution is an internal structural state: it encodes a preferred phase profile, possibly multimodal, with uncertainty.

\subsubsection{Spectral memory}
The complex synaptic amplitude sample is now a complex random variable $Z_{i,j}\sim \mu_{i,j}$, with polar decomposition
\[
Z_{i,j}=r_{i,j}\,e^{\mathrm i\theta_{i,j}},
\qquad\text{where}\quad
r_{i,j}\in\mathrm L^0(\Omega, \mathbb R_+),
\quad\text{and}\quad
\theta_{i,j}\in\mathrm L^0(\Omega, \mathbb R/(2\pi\mathbb Z)).
\]
We also consider the polar decomposition of the first moment
\[
\mathbb E\left[Z_{i,j}\right]=\hat r_{i,j}\,e^{\mathrm i\hat\theta_{i,j}}.
\]
We define:

\begin{definition}[Spectral memory]
  For a given memory field $\rho$, the \emph{spectral memory} of synapse $(i,j)$ is the probability distribution $\xi_{i,j}$ of the random phase $\theta_{i,j}$.
\end{definition}

This spectral memory is a \emph{structural} phase memory: it is an internal property of the synapse, distinct from the instantaneous activation phases of the units. Activation states now become complex-valued,
\[
\psi_i=|\psi_i|\,e^{\mathrm i\theta_i^\psi},
\qquad
\phi_j=|\phi_j|\,e^{\mathrm i\theta_j^\phi},
\]
and interact with the internal spectral memory through the observable synaptic weight
\[
W_{i,j} := p_{i,j}\,r_{i,j}\,e^{\mathrm i\theta_{i,j}}.
\]

\paragraph{Postsynaptic input and spectral filtration.}
The postsynaptic input becomes the complex random variable
\[
  \mathrm{Ps}_j(\psi)
  := \sum_{i=1}^M |\psi_i|\,p_{i,j}\,r_{i,j}\,
  \exp\!\left(\mathrm i\left(\theta_{i,j}+\theta_i^\psi\right)\right).
\]
The real part (or any projection) of this signal is modulated by the spectral memory profile. Depending on the shape of $\mu_{i,j}$, the synapse implements:
\begin{itemize}
  \item a deterministic phase shift (Dirac spectral memory);
  \item multimodal phase selectivity (e.g.\ finite or countable support);
  \item robust spectral filtering (e.g.\ Gaussian or mixture distributions).
\end{itemize}
Thus, the synapse acts as a \emph{functional spectral filter} whose selectivity is encoded in its internal memory distribution.

\subsubsection{Alignment energy and explicit Hebbian structure}

The alignment energy associated with a memory field $\rho$ and activation states $(\psi,\phi)$ is
\[
  E(\rho,\psi,\phi)
  := \frac12 \big\| \mathrm F(\psi \cdot w(\rho)) - \phi \big\|^2.
\]
Expanding the Hermitian product yields
\begin{align*}
  E(\rho,\psi,\phi)
    &= \frac12 \big\|\mathrm F(\psi \cdot w(\rho))\big\|^2
       + \frac12 \|\phi\|^2
       - \Re\!\left\langle \mathrm F(\psi \cdot w(\rho)),\, \phi \right\rangle \\
    &= \frac12 \big\|\mathrm F(\psi \cdot w(\rho))\big\|^2
       + \frac12 \|\phi\|^2
       - \sum_{j=1}^N
         \Re\!\left[
           \overline{\phi_j}\,
           \mathrm F_j\!\left(
             \sum_{i=1}^M \psi_i\, p_{i,j}\,
             \mathbb E[Z_{i,j}]
           \right)
         \right] \\
    &= \frac12 \big\|\mathrm F(\psi \cdot w(\rho))\big\|^2
       + \frac12 \|\phi\|^2
       - \sum_{j=1}^N
         \Re\!\left[
           \overline{\phi_j}\,
           \mathrm F_j\!\left(
             \sum_{i=1}^M \psi_i\, p_{i,j}\,
             \mathbb E[r_{i,j} e^{\mathrm i \theta_{i,j}}]
           \right)
         \right].
\end{align*}

\paragraph{Locally affine activation functions.}
Assume that each activation function $\mathrm F_j$ is locally affine on a non-negligible open neighbourhood $\mathcal U \subset \mathbb C$ containing the values
\[
  \sum_{i=1}^M \psi_i\, p_{i,j}\,
  \mathbb E[r_{i,j} e^{\mathrm i \theta_{i,j}}],
\]
and takes the form
\[
  \mathrm F_j(g) = a_j g + b_j,
  \qquad a_j,b_j \in \mathbb C,
\]
for instance when each $\mathrm F_j$ is affine by parts.

In this region, the alignment energy becomes
\[
  E(\rho,\psi,\phi)
  = \frac12 \big\| \psi \cdot w(\rho) \cdot A - (\phi - B) \big\|^2,
\]
where
\[
  A := \sum_{j=1}^N a_j\, \mathbf 1_{(x_j,x_j)},
  \qquad
  B := \sum_{j=1}^N b_j\, \mathbf 1_{x_j}.
\]

\begin{remark}[Interpretation]
  The vector $B$ plays the role of a \emph{bias}, while $A$ acts as an
  \emph{inertia operator}, in analogy with the coefficients $\sqrt{\alpha_j}$
  in the purely quadratic Tan--HWG case and the role of~$\alpha_j$ in the
  associated EMA observable dynamics.
\end{remark}

Expanding the Hermitian product on~$\mathcal U$ gives
\begin{align*}
  E(\rho,\psi,\phi)
    &= \frac12 \big\|\psi \cdot w(\rho) \cdot A\big\|^2
       + \frac12 \|\phi - B\|^2
       - \Re\!\left\langle
           \psi \cdot w(\rho) \cdot A,\,
           \phi - B
         \right\rangle.
\end{align*}

The quadratic term expands as
\begin{align*}
  \frac12 \big\|\psi \cdot w(\rho) \cdot A\big\|^2
  &= \frac12 \sum_{j=1}^N |a_j|^2
     \left|\sum_{i=1}^M \psi_i\, w_{i,j}(\rho_{x_j})\right|^2 \\
  &= \frac12 \sum_{j=1}^N |a_j|^2
     \sum_{i, i'}
       |\psi_i||\psi_{i'}|\, p_{i,j} p_{i',j}\,
       \hat r_{i,j} \hat r_{i',j}\,
       \cos\!\left(
         (\hat\theta_{i,j} + \theta_i^\psi)
         - (\hat\theta_{i',j} + \theta_{i'}^\psi)
       \right).
\end{align*}

The linear term becomes
\begin{align*}
  -\Re\!\left\langle
    \psi \cdot w(\rho) \cdot A,\,
    \phi - B
  \right\rangle
  &= - \sum_{i,j}
     |\phi_j - b_j|\, |a_j|\, |\psi_i|\, p_{i,j}\, \hat r_{i,j}\,
     \cos\!\left(
       \hat\theta_{i,j}
       + \theta_i^\psi
       + \arg(a_j)
       - \arg(\phi_j - b_j)
     \right).
\end{align*}

Thus, for each fixed $j$, the fiberwise energy reads
\begin{equation}
\begin{aligned}
  \mathcal E_j(\rho_{x_j},\psi,\phi)
    :=\,&
    \frac12 |a_j|^2
    \sum_{i,i'} |\psi_i||\psi_{i'}|\, p_{i,j} p_{i',j}\,
      \hat r_{i,j} \hat r_{i',j}\,
      \cos\!\left(
        (\hat\theta_{i,j} + \theta_i^\psi)
        - (\hat\theta_{i',j} + \theta_{i'}^\psi)
      \right) \\
    &\quad
    - |a_j|\,|\phi_j - b_j|
      \sum_i |\psi_i|\, p_{i,j}\, \hat r_{i,j}\,
      \cos\!\left(
        \hat\theta_{i,j}
        + \theta_i^\psi
        + \arg(a_j)
        - \arg(\phi_j - b_j)
      \right)\\
    &\quad
    + |\phi_j - b_j|^2. \label{eq:local-energy}
\end{aligned}
\end{equation}

\begin{remark}[Explicit Hebbian structure of the energy]
  By the energy equation~\eqref{eq:local-energy}, we see that the synaptic memory weights $p_{i,j}$, $\hat r_{i,j}$ and $\hat\theta_{i,j}$ are directly influenced by co-activations $\psi_i\,\psi_{i'}$ and $\psi_i\,(\phi_j-b_j)$: this is an explicit manifestation of the Hebbian nature of the energy.
  In particular, the output residual activation is only present in the linear term
  \[
  - |a_j|\,|\phi_j - b_j|
    \sum_i |\psi_i|\, p_{i,j}\, \hat r_{i,j}\,
    \cos\!\left(
      \hat\theta_{i,j}
      + \theta_i^\psi
      + \arg(a_j)
      - \arg(\phi_j - b_j)
    \right)
  \]
  which has a direct Hebbian interpretation. More precisely, when both the residual presynaptic activity $|\psi_i|$ and the residual postsynaptic activity $|\phi_j-b_j|$ (corrected by bias) are simultaneously large, it decreases the energy proportionally to
  \[
  p_{i,j}\, \hat r_{i,j}\,
  \cos\!\left(
    \hat\theta_{i,j}
    + \theta_i^\psi
    + \arg(a_j)
    - \arg(\phi_j - b_j)
  \right).
  \]
  Hence it will favor simultaneously (i) larger weights $p_{i,j}\, \hat r_{i,j}$ as well as (ii) external field alignment i.e.\ when $\hat\theta_{i,j}$ leads to an alignment of input phase $\theta_i^\psi$ with the output phase $\arg(\phi_j - b_j)-\arg(a_j)$ (including correction by bias and inertia).
  In other words, the synapse adjusts its internal phase so as to compensate the phase mismatch between presynaptic and postsynaptic oscillatory activity. This is a phase–Hebbian mechanism: \emph{structural phase alignment occurs when the two units fire together}.
\end{remark}

\begin{remark}[Phase-locking by large output stimulation and spectral memory as a repository of stable phase offsets.]\label{rmk:offsets-repo}
  When the output residual activation ($\phi_j$) is large, the alignment mechanism above contributes to shaping the spectral memory. The spectral memory acts as a repository of phase offsets that repeatedly and consistently occur with large input/output coactivations. As learning progresses and predictions become more accurate, the effective phase differences stabilise, and the structural phases $\theta_{i,j}$ converge accordingly. The plasticity dynamics therefore implements a structural phase-locking mechanism, learning and storing \emph{stable phase relationships} between input and output activity patterns.
\end{remark}

The explicit form of the energy provides a convenient basis for a local parametric analysis, which is the purpose of the following section.

\subsubsection{Local parametric analysis: pruning, selectivity, pressure equalisation, synchronization, alignment}
\label{sec:parametric}

We now analyse the fiberwise structure of the energy to reveal mechanisms governing pruning, selectivity, synchronization and phase alignment. The following analysis is qualitative and aims at revealing the mechanisms encoded in the energy landscape, rather than providing a full classification of equilibria.

A reasonable assumption is to assume each $\mathrm F_j$ sufficiently regular to locally admit an affine approximation in an open neighbourhood $\mathcal U\in\mathbb C$ of the form
\[
\mathrm F_j(g)=\mathrm F_j(g_0)+\mathrm F'_j(g_0)\,(g-g_0)+o(g-g_0),
\quad\text{for fixed } g_0\in\mathcal U \quad\text{and all }g\in\mathcal U.
\]
In the following analysis, we neglect the $o(g)$ term (which actually vanishes in the case of $\mathrm F_j$ affine by parts). In this setting, the structure of the energy reveals three fundamental mechanisms:

\paragraph{i. Amplitude equilibrium and pruning.}
The partial derivative with respect to $\hat r_{i,j}$ is
\begin{align*}
  \frac{\partial E}{\partial \hat r_{i,j}}
    &= |a_j|^2 |\psi_i|^2 p_{i,j}^2 \hat r_{i,j}
     + |a_j|^2 |\psi_i| p_{i,j}
       \sum_{i'\neq i}
         |\psi_{i'}| p_{i',j} \hat r_{i',j}
         \cos((\hat\theta_{i,j} + \theta_i^\psi)
             - (\hat\theta_{i',j} + \theta_{i'}^\psi)) \\
    &\qquad
     - |a_j|\,|\phi_j - b_j|\,|\psi_i|\,p_{i,j}
       \cos(\hat\theta_{i,j}
           + \theta_i^\psi
           + \arg(a_j)
           - \arg(\phi_j - b_j)).
\end{align*}
Either $a_j= 0$ (the synapse is ignored in $\mathcal U$), or $p_{i,j}=0$ (structural pruning), or $y_i$ is not activated ($\psi_i=0$, hence no Hebbian phenomenon involved), or the equilibrium magnitude is theoretically
\begin{align*}
  \hat r_{i,j}^{\mathrm{eq}}
    &= \frac{|\phi_j - b_j|}
            {|a_j|\,|\psi_i|\,p_{i,j}}
       \cos(\hat\theta_{i,j}
             + \theta_i^\psi
             + \arg(a_j)
             - \arg(\phi_j - b_j)) \\
    &\qquad
       - \sum_{i'\neq i}
           \frac{|\psi_{i'}| p_{i',j}}
                {|\psi_i| p_{i,j}}
           \hat r_{i',j}
           \cos((\hat\theta_{i,j} + \theta_i^\psi)
               - (\hat\theta_{i',j} + \theta_{i'}^\psi)).
\end{align*}
Since
\[
  \frac{\partial^2 E}{\partial \hat r_{i,j}^2}
    = |a_j|^2 |\psi_i|^2 p_{i,j}^2 \ge 0,
\]

$\hat r_{i,j}^{\mathrm{eq}}$ is a local minimum and
if $\hat r_{i,j}^{\mathrm{eq}}<0$, and unless the joint evolution of the other variables leads to lift $\hat r_{i,j}^{\mathrm{eq}}$ above $0$, the constraint $\hat r_{i,j}\ge 0$ enforces pruning ($\hat r_{i,j}\to \hat r_{i,j}^{\mathrm{eq}}$ until reaching $0$). Note that a stricter constraint could be given by considering the frontier of $\mathcal U$: we do not detail such consideration in this paper.

\paragraph{ii. Pressure equalization, pruning and selectivity.}
For fixed amplitudes and phases, the optimization of the structural weights $(p_{i,j})_i$ is constrained by the simplex condition $\sum_i p_{i,j}=1$, $p_{i,j}\ge 0$. Introducing a Lagrange multiplier $\lambda_j$ for the equality constraint, the stationarity condition reads
\[
  \frac{\partial E}{\partial p_{i,j}} = - \lambda_j
  \qquad\text{for all } i \text{ with } p_{i,j}>0,
\]
together with the complementary slackness condition
\[
  p_{i,j}=0 \;\Rightarrow\;
  \frac{\partial E}{\partial p_{i,j}} + \lambda_j \ge 0.
\]
Since
\[
  \frac{\partial^2 E}{\partial p_{i,j}^2}
    = |a_j|^2 |\psi_i|^2 \hat r_{i,j}^2 \ge 0,
\]
the energy is convex in each $p_{i,j}$, and interior equilibria satisfy
\[
  \frac{\partial E}{\partial p_{i,j}}
  = \frac{\partial E}{\partial p_{i',j}}
  \qquad\text{for all active indices } i,i'.
\]
Thus, active synapses equalize their ``energy pressure'' $\partial E/\partial p_{i,j}$, while inactive synapses lie on the boundary of the simplex.

In particular, if the unconstrained equilibrium value $p_{i,j}^{\mathrm{eq}}$ obtained from $\partial E/\partial p_{i,j}=0$, i.e.\
\begin{align*}
  p_{i,j}^{\mathrm{eq}}
    &= \frac{|\phi_j - b_j|}
            {|a_j|\,|\psi_i|\,\hat r_{i,j}}
       \cos(\hat\theta_{i,j}
             + \theta_i^\psi
             + \arg(a_j)
             - \arg(\phi_j - b_j)) \\
    &\qquad
       - \sum_{i'\neq i}
           \frac{|\psi_{i'}| \hat r_{i',j}}
                {|\psi_i| \hat r_{i,j}}
           p_{i',j}
           \cos((\hat\theta_{i,j} + \theta_i^\psi)
               - (\hat\theta_{i',j} + \theta_{i'}^\psi)),
\end{align*}
satisfies $p_{i,j}^{\mathrm{eq}}<0$, the Karush-Kuhn-Tucker condition forces $p_{i,j}$ towards $0$, corresponding to \emph{structural pruning}. Conversely, if $p_{i,j}^{\mathrm{eq}}>1$, the simplex constraint forces all other weights towards vanishing, yielding \emph{selectivity}. Interior equilibria occur only when all active weights share the same value of $\partial E/\partial p_{i,j}$.

\paragraph{iii. Synchronization and alignment.}
$p_{i,j}^{\mathrm{eq}}$ and $\hat r_{i,j}^{\mathrm{eq}}$ actually follow the same equation, which gives the following \emph{Hebbian equilibrium equation}:
\begin{align*}
  |\psi_i|\,p_{i,j}^{\mathrm{eq}}\,\hat r_{i,j}^{\mathrm{eq}}
    &= \frac{|\phi_j - b_j|}
            {|a_j|}
       \cos(\hat\theta_{i,j}^{\mathrm{eq}}
             + \theta_i^\psi
             + \arg(a_j)
             - \arg(\phi_j - b_j)) \\
    &\qquad
       - \sum_{i'\neq i}
           |\psi_{i'}| p_{i',j}^{\mathrm{eq}}
           \hat r_{i',j}^{\mathrm{eq}}
           \cos((\hat\theta_{i,j}^{\mathrm{eq}} + \theta_i^\psi)
               - (\hat\theta_{i',j}^{\mathrm{eq}} + \theta_{i'}^\psi)),
\end{align*}
satisfied by the theoretical equilibrium point $(p_{i,j}^{\mathrm{eq}},\hat r_{i,j}^{\mathrm{eq}},\hat\theta_{i,j}^{\mathrm{eq}})$ (``theoretical'' in the sense that it may never be attained if $p_{i,j}^{\mathrm{eq}}\notin [0,1]$ or $\hat r_{i,j}^{\mathrm{eq}}<0$). This equation can be rewritten as
\begin{equation}
  \label{eq:hebb-eq-eq}
  \sum_{i'=1}^M
      |\psi_{i'}| p_{i',j}^{\mathrm{eq}}
      \hat r_{i',j}^{\mathrm{eq}}
      \cos((\hat\theta_{i,j}^{\mathrm{eq}} \!+\! \theta_i^\psi)
          - (\hat\theta_{i',j}^{\mathrm{eq}} \!+\! \theta_{i'}^\psi))
  = \frac{|\phi_j \!-\! b_j|}{|a_j|}
      \cos(\hat\theta_{i,j}^{\mathrm{eq}}
           \!+\! \theta_i^\psi
           \!+\! \arg(a_j)
           \!-\! \arg(\phi_j \!-\! b_j)).
\end{equation}

(iii.a.) If the internal phases satisfy
\[
  \cos((\hat\theta_{i,j} + \theta_i^\psi)
      - (\hat\theta_{i',j} + \theta_{i'}^\psi)) > 0
  \quad\text{for all } i,i',
\]
the synapses are said to be globally \emph{synchronized} relative to a fixed activation input $\psi$. By the Hebbian equilibrium equation~\eqref{eq:hebb-eq-eq}, the equilibrium phase $\theta_{i,j}^{\mathrm{eq}}$ satisfies
\[
\cos(\hat\theta_{i,j}^{\mathrm{eq}}
     \!+\! \theta_i^\psi
     \!+\! \arg(a_j)
     \!-\! \arg(\phi_j \!-\! b_j)) \ge 0.
\]
Hence the dynamics pushes the spectral memory to compensate the input phase towards \emph{alignment} with the output field (corrected by bias and inertia).
In other words, synchronized groups pushes towards phase alignment, which recovers a \emph{phase locking} phenomenon.

Note that, if
\[
  \cos(\hat\theta_{i,j}
        + \theta_i^\psi
        + \arg(a_j)
        - \arg(\phi_j - b_j)) < 0,
\]
then both
\[
\frac{\partial E}{\partial \hat r_{i,j}} > 0,
\qquad\text{and}\qquad
\frac{\partial E}{\partial p_{i,j}} >0,
\]
and since we have the constraints $\hat r_{i,j}\ge 0$ and $p_{i,j}\ge 0$, this forces both structural weights and amplitude towards vanishing i.e.\ pruning: the regime is unstable as expected.

(iii.b) Conversely, having
\[
  \cos((\hat\theta_{i,j} + \theta_i^\psi)
      - (\hat\theta_{i',j} + \theta_{i'}^\psi)) < 0
  \quad\text{for all } i,i',
\]
is only possible in a group of at most 3 synapses, since on the circle, a set of points with all pairwise negative cosine differences cannot exceed three elements. In such a limited group of ``anti''-synchronized neurons, we have
\[
\cos(\hat\theta_{i,j}^{\mathrm{eq}}
     \!+\! \theta_i^\psi
     \!+\! \arg(a_j)
     \!-\! \arg(\phi_j \!-\! b_j)) \le 0,
\]
and the group can be said to be ``anti''-aligned. We can check that group alignment, i.e.\
\[
  \cos(\hat\theta_{i,j}
        + \theta_i^\psi
        + \arg(a_j)
        - \arg(\phi_j - b_j)) > 0,
\]
may push some $p_{i,j}^{\mathrm{eq}}>1$, pushing towards selectivity, pruning the other connections.

In large groups of active synapses ($\gg\! 3$), stable non-degenerate equilibria would therefore generally correspond to \emph{synchronization + alignment}.

\begin{remark}[Spectral memory interpretation]
  This generalizes the phase-locking and stable phase offsets memory mechanisms discussed in Remark~\ref{rmk:offsets-repo}, adding synchronization as a joint phenomenon.
\end{remark}

\begin{remark}[From frequency mismatch to frequency locking in dynamic spectral memory: a Kuramoto dynamics]
  Spectral memory only stores phase distribution and not frequencies. Consider now the case in which the input and output phases evolve as
  \[
  \theta_i^\psi(t) = \omega_{\mathrm{in}} t + \theta_i^0,
  \qquad
  \theta_j^\phi(t) = \omega_{\mathrm{out}} t + \theta_j^0,
  \]
  with possibly distinct frequencies $\omega_{\mathrm{in}} \neq \omega_{\mathrm{out}}$. The effective phase difference that the Tan--HWG plasticity attempts to compensate is then
  \[
  \Delta\theta_{i,j}^{\mathrm{target}}(t)
  = \theta_j^\phi(t) - \theta_i^\psi(t)
  = (\omega_{\mathrm{out}} - \omega_{\mathrm{in}})\,t + C_{i,j}.
  \]
  When $\omega_{\mathrm{out}} = \omega_{\mathrm{in}}$, this target offset is constant and the structural phase $\theta_{i,j}$ can converge to a stable value, yielding a genuine phase memory. In contrast, when $\omega_{\mathrm{out}} \neq \omega_{\mathrm{in}}$, the target offset drifts linearly in time. A static spectral memory cannot encode such a drift: the phase that the plasticity attempts to learn is no longer a fixed point but a continuously rotating quantity. In this regime, the synapse would need to ``rotate'' its structural phase at a rate
  \[
  \omega_{i,j} = \omega_{\mathrm{out}} - \omega_{\mathrm{in}},
  \]
  which effectively defines a \emph{learned intrinsic frequency}. However, such a frequency cannot be stored as a static synaptic parameter: it must be \emph{maintained dynamically} by an internal oscillatory mechanism.

  A biologically plausible solution is the presence of a recurrent loop or microcircuit generating a self-sustained oscillation whose phase evolves as $\theta_{i,j}(t) = \omega_{i,j} t + \theta_{i,j}^0$. This internal oscillator provides the synapse with a continuously rotating reference phase, allowing the spectral memory to encode stable offsets relative to this internally generated drift. Such self-sustained oscillations are well documented in cortical, thalamic, hippocampal and cerebellar circuits \parencite{buzsaki2006rhythms,wang2010neurophysiological,lisman2013theta,lampl1999synchrony}, where recurrent excitation--inhibition loops maintain intrinsic frequencies that interact with external inputs through phase-locking mechanisms.

  In summary, a frequency mismatch $\omega_{\mathrm{out}} \neq \omega_{\mathrm{in}}$ cannot be stored as a static spectral memory. Instead, it requires an internally generated oscillation that continuously maintains the appropriate drift. The Tan--HWG plasticity then performs structural phase-locking relative to this internal oscillator, rather than memorising a frequency as a fixed synaptic parameter.

  This ``internal'' reference frequency, together with synchronization mechanisms, is reminiscent of a Kuramoto-like dynamics. However, we did not need to introduce an additional oscillatory interaction equation. In inference mode (e.g.\ when no feedback output activation signal is given), the spectral memory acts as a spectral filter favoring compatible input phases $\theta^\psi_i$. Hence, in this perspective, the synchronization of the $\theta^\psi_i$ comes from phase filtration rather than from an explicit synchronization force (although such an additional force would not be incompatible with the framework).
\end{remark}

\begin{remark}[Functional analysis.]
  At the level of the memory field $\rho$, the plasticity dynamics could be rigorously described as a Tan–HWG dynamics.
  In the simple affine case, the local alignment energy takes the form
  \begin{align*}
    \mathcal E_x(\rho_{x}| (\psi, \phi))
      :=\,& \frac12|A(x)|^2 \iint_{(Y_M\times\mathbb C)^2} K((y,z),(y',z')|\psi)\,\mathrm d\rho_x(y,z)\mathrm d\rho_x(y',z')\\
      &-\Re\left[\overline{(\phi(x)\!-\!B(x))}A(x)\int_{Y_M\times\mathbb C} V(y,z|\psi)\,\mathrm d\rho_x\right]\\
      &+|\phi(x)-B(x)|^2,
  \end{align*}
  where
  \[
  K((y,z),(y',z')|\psi):=\psi(y)\overline\psi(y')z\,\overline{z'},
  \qquad\text{and}\qquad
  V(y,z|\psi):=\psi(y)z.
  \]
  The qualitative analysis in terms of amplitudes, phases, and structural weights can be seen as a parametric reading of this underlying variational structure. An explicit Wasserstein gradient flow formulation of $\rho_x$ may be given, leading to a \emph{Hebbian plasticity equation}
  \begin{equation}\label{eq:hebb-plasticity}
    \partial_t\rho_x = \nabla\cdot\left(\rho_x\nabla\frac{\delta \mathcal E_x}{\delta\rho_x}\right),
  \end{equation}
  where, in the affine case, we have
  \[
  \frac{\delta \mathcal E_x}{\delta\rho_x}(y,z)
    =|A(x)|^2 \int_{Y_M\times\mathbb C} K((y,z),(y',z')|\psi)\,\mathrm d\rho_x(y',z')
      -\Re\left[\overline{(\phi(x)\!-\!B(x))}A(x) V(y,z|\psi)\right].
  \]
  A full analytical characterization of the dynamics would require additional structural assumptions. In this work, we therefore focus on the qualitative mechanisms revealed by the energy landscape (pruning, selectivity, synchronization and alignment), rather than on a complete classification of trajectories.
\end{remark}

\subsubsection{Discussion}

\paragraph{On the functional role of the structural weights.} The structural weights $p_{i,j}$ have distinguished features, which are otherwise not covered when solely considering amplitudes $Z_{i,j}$:
\begin{itemize}
  \item \emph{Synaptic competition.}
  The simplex constraint $\sum_i p_{i,j}=1$ gives a representation of synaptic competition over resources: the postsynaptic neuron can only provide limited attention to its synapses. Amplitudes and structural weights both show pruning capacities, but solely structural weights are subject to selectivity, i.e.\ competition pushed to its degenerate limit when only one synapse dominates, pushing the others to vanish. Note that selectivity here refers to functional dominance, not anatomical elimination.

  \item \emph{Semantic bumper for stability and robustness of informational embedding.}
  Since the energy is a function of the product $p_{i,j}\cdot\hat r_{i,j}e^{\mathrm i\hat\theta_{i,j}}$, small perturbations of the input or of the synaptic amplitudes can be absorbed either by adjusting the fibre laws $\mu_{i,j}$ (via $\hat r_{i,j},\hat\theta_{i,j}$) or by redistributing mass across inputs (via $p_{i,j}$), which enhances robustness to noise. If we interpret amplitudes $(\hat r_{i,j}e^{\mathrm i\hat\theta_{i,j}})$ as semantic or informational embeddings, the layer of structural weights $p_{i,j}$ can act as a stabilizer of such embeddings.
  When facing repeated activations $(\psi,\phi)$ that would otherwise tend to modify the embedding (e.g.\ presenting incoherences or small perturbations with previous activation profiles), the discrepancy can be partially absorbed by the structural weights. It allows robust embeddings, maintaining coherence across inputs/outputs activations.
  This absorption capacity has limits, in particular since $p_{i,j}\in [0,1]$, and since active synapses are pushed towards respecting the stabilized energy pressure attractor $\partial E/\partial p_{i,j}=-\lambda_j$ at postsynaptic neuronal level. These limits can be seen as gating transitions from a semantic value to another, which can be linked to the phase transition feature (see next item).

  \item \emph{Phase transitions.}
  In this setting, a phase transition refers to a qualitative change in the structure of the minimizer of the local energy $\mathcal E_x$. More precisely, it occurs when the optimal structural weights $(p_{i,j})_{i=1\dots M}$ move from an interior point of the simplex (all synapses active) to a boundary point (some $p_{i,j}=0$), or conversely.
  Such transitions correspond to changes in the active support of $\rho_x$ and are induced by the geometry of the energy landscape. They are invisible when optimising only over amplitudes, since the simplex constraint is essential for the emergence of these structural bifurcations.
\end{itemize}

\paragraph{Comparison with existing architectures.}
The separation between structural weights $p_{i,j}$ and semantic complex amplitudes $z_{i,j}=\hat r_{i,j}e^{\mathrm i\hat\theta_{i,j}}$ may evoke mechanisms found in several existing models, but the proposed construction differs from all of them in essential ways.
\begin{itemize}
  \item \emph{Mixture models and EM-type architectures.}
  Gaussian mixture models and related EM algorithms also separate mixture weights from component parameters. However, these models lack any notion of synaptic synchronization or alignment.
  \item \emph{Attention mechanisms and mixture-of-experts.}
  Transformer attention and MoE architectures use routing weights (softmax scores or gating coefficients), which is reminiscent of the structural weights $p_{i,j}$ forming a probability distribution over input locations and modulating the contribution of semantic synaptic states $z_{i,j}=\hat r_{i,j}e^{\mathrm i\hat\theta_{i,j}}$. Yet these weights do not induce pruning or selectivity through simplex boundary equilibria, nor do they account for synchronization and alignment phenomena.
  \item \emph{Sparse coding and dictionary learning.}
  Sparse coding separates activation coefficients from dictionary atoms, but the activations are continuous and do not represent structural mass allocation. There is no phase dynamics, no synchronization mechanism, and no structural pruning analogous to $p_{i,j}\to 0$.
  \item \emph{Bayesian nonparametrics.}
  Dirichlet-process mixtures also feature structural weights, but these arise from stochastic priors rather than from a deterministic energy landscape.
\end{itemize}

In summary, while several families of models exhibit partial analogies (routing vs.\ content, mixture weights vs.\ component parameters), none of them combine: (i) structural mass allocation on a simplex, (ii) semantic complex amplitudes with phase interactions, (iii) a Wasserstein gradient flow on a mixed discrete--continuous space, and (iv) emergent pruning and selectivity through KKT boundary equilibria.

\paragraph{Practical advantages}
The geometric and dynamical structure of the proposed model leads to several practical advantages:
\begin{itemize}
  \item \emph{Intrinsic stability and convergence.}
  The present architecture is governed by a well posed minimizing scheme relative to an explicit energy, yielding intrinsic stability, convergence to equilibrium states, and robustness to noise through mass redistribution and phase synchronization.

  \item \emph{Emergent sparsity and potential computational benefits.}
  Attention weights in Transformers are dense and rarely reach exact zeros, unless sparsity is imposed externally. Here, the simplex constraint on $(p_{i,j})_i$ and the associated KKT conditions naturally produce structural pruning and selectivity: entire synaptic branches may vanish at equilibrium. This leads to adaptive sparsity, reduced computational cost, and improved interpretability. This opens the door to adaptive, data-driven sparsification without additional regularization or architectural heuristics.

  \item \emph{Geometric coherence and phase interactions.}
  Classic neural networks operate in a flat Euclidean space and lacks an intrinsic geometric structure. The proposed model evolves in a rich, intrinsically curved, CAT(0) Polish space, $\mathcal P_2(Y_M\times\mathbb C)$ ($Y_M$ being implicitly lifted into a CAT(0) Polish space, see Section~\ref{sec:projection}), combining mass transport with complex-valued phase interactions. This ensures uniqueness of geodesics and convexity properties that are absent from classical Euclidean embeddings. This induces geometric coherence, synchronization phenomena, and alignment with external fields, which can improve robustness and temporal consistency in applications involving oscillatory or structured signals.
\end{itemize}

\section{Mapped energies}\label{sec:mapped-energy}

In Section~\ref{sec:networks} we introduced geometric representation maps, and we used a \emph{contextualized} geometric representation map of the form $\rho\mapsto \psi\,w(\rho)$ to define the alignment energy. In this section, we generalize this construction.

\begin{definition}[Mapped energy]
  A Hebbian energy is said to be a \emph{mapped energy} if its local density has the form
  \[
    \mathcal E_x(\rho_x,z)=\widetilde{\mathcal E}_x(\mathcal H_x(\rho_x,z),z),
    \qquad \rho_x\in\mathcal P_2(Y), z\in Z,
  \]
  where $\mathcal H$ is said to be a \emph{contextualized geometric representation map} or \emph{Hebbian map} $\mathcal H=\lbrace \mathcal H_x\rbrace_{x\in X}$ where
  \[
    \mathcal H_x:\mathcal P_2(Y)\times Z\to Z,
    \quad x\mapsto\mathcal H_x \text{ measurable}.
  \]
  When $Z$ is a Hilbert space, we say that the mapped energy is a \emph{Hilbertian mapped energy}.
\end{definition}

The term ``Hebbian'' in the name Hebbian map, reflects the fact that $\rho_x\mapsto\mathcal H_x(\rho_x,z)$ is a geometric representation map contextualized by a Hebbian global signal $z$.

Mapped energies provide a unifying framework for expressing Hebbian-type energies as compositions of geometric maps with classical Euclidean costs. This makes it possible to analyze the variational structure of complex neural architectures by studying the regularity of the underlying geometric maps.

\subsection{Examples: energies induced by activation functions}

Hilbertian mapped energies conveniently map probability distributions $\rho_x$ to vectors in $Z$. Then, this vectorial representation allows us to bridge the underlying Wasserstein geometry with classical alignment or correlation ``cost'' functions.

\paragraph{Alignment energy.}
For instance, the alignment energy, defined in Section~\ref{sec:networks}, is a special case of Hilbertian mapped energy with
\[
  E(\rho,S(\rho,\phi))=\frac12 \sum_{j=1}^N \left(\mathcal H_{x_j}(\rho_{x_j}, (\psi_{\cdot,j}, \phi_j)) - S_{x_j}(\rho, (\psi, \phi)) \right)^2,
  \qquad S_{x_j}(\rho, (\psi, \phi))=(\psi_{\cdot,j}, \phi_j),
\]
and
\[
  \mathcal H_{x_j}(\rho_{x_j}, (\psi_{\cdot,j}, \phi_j))
    =\left(\psi_{\cdot,j},\,\mathrm F_j(\psi_{\cdot,j}\cdot w_j(\rho_{x_j}))\right)
    =\left(\psi_{\cdot,j},\,\mathrm F_j\left(\sum_{i=1}^M\psi_{i,j}\cdot w_{i,j}(\rho_{x_j})\right)\right).
\]
where we used the general synaptic activation vector $\psi_{\cdot,j}=\sum_{i=1}^M \psi_{i,j}\mathbf 1_{y_i}$ (which depends on $x_j$; see Remark~\ref{rmk:synaptic-activation}).

\begin{remark}[Elementary contextualized map]
  The Hebbian map $\mathcal H$ is not unique: the alignment energy can also be written
  \[
    E(\rho, \psi, \phi)
    =\frac12\sum_{j=1}^N\left( \mathrm F_j\left(
        \int \psi(y')z\,\mathrm d\rho_{x_j}(y',z)
      \right)-\phi_j\right)^2
    =\widetilde{\mathcal E}_x(\mathrm{ps}_j(\rho_{x_j},\psi),z),
  \]
  using
  \[
    \mathrm{ps}_j(\rho_{x_j},\psi)=\int \psi(y')z\,\mathrm d\rho_{x_j}(y',z),
  \]
  as the contextualized geometric representation map. More generally, we may factor ``elementary'' contextualized maps within mapped energies. Elementary features of such building blocks, like geodesic affinity and Lipschitz continuity, inform us on the behavior of the memory dynamics (e.g.\ see Section~\ref{sec:multi-scale}).

  Mapped energies allow us to decompose complex neural architectures into elementary geometric maps, whose regularity properties directly control the behavior of the memory dynamics.
\end{remark}

\subsection{$d_Z$--isotropic, $d_Z$-anisotropic and $d_Z$-quadratic energies}\label{sec:d_Z-isotropic}

The present section can be seen as a generalization of Section~\ref{sec:quadratic}, where we had $Z=\mathcal P_2(Y)$ and $\mathcal H_x(\cdot, z)\equiv\mathrm{Id}_{\mathcal P_2(Y)}$. In particular:
\begin{itemize}
  \item \emph{$d_Z$-isotropic energy:} We can define a class of $d_Z$-isotropic energies, with densities of the form
  \[
  \mathcal E_x(\cdot, z):=\mathcal F_x(d_Z(\mathcal H_x(\cdot,z), z)),
  \quad z\in Z.
  \]
  \item \emph{$d_Z$-anisotropic energy:} Use a bi-Lipschitz equivalent distance on $Z$, to define a $d_Z$-anisotropic energy.
  \item \emph{$d_Z$-quadratic energy:} When $Z$ is CAT(0) and when $\mathcal H_x(\cdot,z)$ is geodesically convex for $\mu$-a.e.\ $x$, then we can define $d_Z$-quadratic Tan--HWG energies with densities of the form
  \[
    \mathcal E_x(\cdot,z):=f_x(d_Z^2(\mathcal H_x(\cdot,z),z)),
  \]
  where $f_x$ is proper, l.s.c. and convex nondecreasing, with $x\mapsto f_x$ measurable, and
  \[
    f_x(0)<+\infty \text{ $\mu$-a.e.\ and } (x\mapsto f_x(0)\in \mathrm L^1).
  \]
  \item \emph{State-independent contraction factor:} Assuming $\mathcal H_x(\cdot, z)$ is geodesically affine and $1$-Lipschitz, we can obtain a generalized state-independent contraction factor theorem for generalized quadratic energy (see Definition~\ref{def:general-quadratic}), with a coefficient $t_{\tau,x}\in(0,1)$ \emph{independent of $\rho_x^n$} such that
  \[
    \mathcal H_x(\rho_x^{n+1},h^n_x)=h_{x,t_{\tau,x}}^n,
    \qquad
    d_Z(\mathcal H_x(\rho_x^n,h^n_x),h_{x,t_{\tau,x}}^n)=t_{\tau,x}\,d_Z(\mathcal H_x(\rho_x^n,h^n_x),h^n_x),
  \]
  for $\mu$-a.e.\ $x\in X$, where $(h_{x,t}^n)_{t\in[0,1]}$ is the geodesic from $\mathcal H_x(\rho_x^n,h^n_x)$ to $h^n_x$. This also assumes $\mathcal H_x(\cdot,h^n_x)^{-1}(h_x^n)\neq\emptyset$. Note that, having $\mathcal H_x(\cdot, z)$ geodesically affine and $1$-Lipschitz means $\mathcal H_x(\cdot, z)$ is geodesically an isometry, which is a strong assumption: $\mathcal H_x(\cdot, z)$ strictly respects the underlying Wasserstein geometry.
  Such assumptions are rarely satisfied in practice, but they illustrate the structural behavior of mapped energies.
\end{itemize}

In particular, if we require convexity of the energy, $d_Z$-quadratic Tan--HWG energies provide a wide class of compatible energies.
However this construction excludes the use of weakly convex mapped energies.

\subsubsection{Generalized quadratic energy (GQE)}

We now present the class of energies that we will consider to establish a continuous-time limit curve in Section~\ref{sec:continuous-limit}. It is a generalization of the purely quadratic Tan--HWG energy class, using geometric representation maps.

\begin{definition}[Generalized quadratic energy]
  \label{def:general-quadratic}
  A \emph{generalized quadratic energy} (GQE), with parameter $\alpha>0$, and mapping $\mathcal H$, is a mapped energy of the form
  \[
  E(\rho,\phi)
  := \frac{\alpha}{2} \int_X d_Z^2\big(\mathcal H_x(\rho_x,S_x(\rho,\phi)), S_x(\rho,\phi)\big)\, \mathrm{d}\mu(x).
  \]
  When the energy is Tan--HWG, we say that it is a \emph{Tan--HWG GQE}.
\end{definition}

Note that, in the general case, we cannot guarantee the $\lambda$-convexity of a GQE.

\section{Vanishing time step limit curve}
\label{sec:continuous-limit}

Up to this point, the global signal---and in particular the context---has been treated as an arbitrary external input, since the Tan--HWG minimizing-movement scheme freezes it at each step. In particular, with a fixed timestep $\tau$, the context could be considered as a given fully external process, taken as an input for the system. In this section we will show that, under suitable structural assumptions on the signal and a coupling constraint between memory fields and contexts, the system may enter an autonomous, or quasi-autonomous regime, with a ``free evolution'' of the signal (i.e.\ under stationary or quasi-stationary external context). Remarkably, in this regime, Tan--HWG schemes admit continuous-time limit curves (Theorem~\ref{thm:tan-hwg-continuous}). This suggests that such limit curves are key features of memory consolidation, which is more likely to occur in quasi-stationary external context (e.g.\ ``sleep'').

\subsection{Regularity assumptions}\label{sec:assumptions}

We now give three regularity assumptions ---two structural and one contextual, under which, Tan--HWG minimizing movement schemes can admit continuous-time limits. Lipschitz continuity is the key structural assumption, while ``sleep-mode'' is the key contextual assumption: both will allow us to control the freezing error in the Tan--HWG scheme.

\subsubsection{First structural assumption: Lipschitz signal}

Consider a Hebbian energy
\[
E(\rho,\phi)
:= \int_X \mathcal E_x\big(\rho_x,S_x(\rho,\phi)\big)\, \mathrm{d}\mu(x).
\]

\begin{assumption}[Lipschitz signal]\label{A1}
  The global signal is Lipschitz, in the sense that there exists $L_S>0$ such that for $\mu$-a.e.\ $x\in X$,
  \[
    d_Z\big(S_x(\rho,\phi),S_x(\rho',\phi'))
    \le L_S \left( \mathcal W^2(\rho, \rho') + d^2_C(\phi, \phi') \right)^{1/2},
  \]
  for all $\rho, \rho'\in\mathcal X_\mu$, and all $\phi, \phi'\in\mathcal C_\mu$ i.e.\ for $\mu$-compatible contexts.
\end{assumption}

\subsubsection{Second structural assumption: Lipschitz map}

Consider a mapped energy
\[
E(\rho,\phi)
:= \int_X \widetilde{\mathcal E}_x\big(\mathcal H_x(\rho_x,S_x(\rho,\phi)), S_x(\rho,\phi)\big)\, \mathrm{d}\mu(x).
\]

\begin{assumption}[Lipschitz map]\label{A2}
  The contextualized geometric representation map is Lipschitz, in the sense that there exists $L_H>0$ such that for $\mu$-a.e.\ $x\in X$,
  \[
    d_Z\big(\mathcal H_x(\rho_x,z),\mathcal H_x(\rho'_x,z'))
    \le L_H \left( W^2_2(\rho_x, \rho'_x) + d_Z^2(z,z') \right)^{1/2},
  \]
  for all $\rho_x, \rho'_x\in\mathcal P_2(Y)$, and all $z, z'\in Z$.
\end{assumption}

\subsubsection{Contextual assumption: the sleep-mode assumption}
Consider a Tan--HWG energy
\[
E(\rho,\phi)
:= \int_X \mathcal E_x\big(\rho_x,S_x(\rho,\phi)\big)\, \mathrm{d}\mu(x).
\]
We consider that when the system is left in free evolution, with no external stimulation, it presents an autonomous regime, where activations (contexts) depend solely on internal parameters, mainly the memory field. The discrete evolution of the context can be written
\[
  \phi^{n+1}_\tau:= \Phi_\tau^E(\rho^n_\tau | \phi^n_\tau) + \epsilon^{n+1}_\tau,
\]
where the operator $\Phi_\tau^E$ describes the purely internal evolution, and $\epsilon^{n+1}_\tau$ represents externally driven activations (e.g.\ sensory driven activations).

The sleep-mode assumption formalizes a regime in which internally driven, stable neural dynamics dominate external perturbations ($\epsilon^{n+1}_\tau$).

\begin{assumption}[The sleep-mode assumption]\label{A3}
  There exists $L_C>0$ such that, for every timestep $\tau>0$ and every discrete Tan--HWG sequence $(\rho^n_\tau)_{n\ge0}$ associated with $E$, and with a $\mu$-compatible context sequence $(\phi^n_\tau)_{n\ge0}\in\mathcal C_\mu^{\mathbb N}$, we have
  \[
    d_C(\phi^n_\tau, \phi^{n+1}_\tau) \le L_C\mathcal W(\rho^n_\tau, \rho^{n+1}_\tau),
    \qquad \rho^{n+1}_\tau\in T^E_\tau(\rho^n_\tau\,|\,\phi^n).
  \]
\end{assumption}
We also say that $\phi$ presents \emph{quasi-stationary external context}. This is a coupling constraint between $(\rho^n_\tau)_n$ and $(\phi^n_\tau)_n$.
This assumption formalizes the idea that consolidation occurs when external perturbations are weak relative to intrinsic neural dynamics.
In this regime, the external context perceived by the network is quasi-stationary, so that sensory-driven fluctuations are minimal. As a consequence, neuronal activations evolve in a calm, weakly perturbed dynamical state. The network operates in a nearly unconstrained ``free evolution'' mode, where the intrinsic flow of neural activity dominates over externally imposed variations.

\begin{remark}
  If the context depends solely on the memory field, i.e.\ $\phi=\phi(\rho)$, then the sleep-mode assumption reduces to a Lipschitz regularity condition on $\phi$.
  This corresponds to a fully \emph{memory driven free evolution} regime: the system is isolated from external stimuli (e.g.\ sleep), and internal contextual states evolve smoothly with the memory.

  In the general case, $\phi$ may include external influences, and other internal influences, such that the expression using the operator $\Phi_\tau^E$ is more general. In any case, the assumption enforces a \emph{quasi-free evolution}, independently of the dependencies of $\phi$: external variations must be sufficiently regular so as not to disrupt the internal Tan--HWG dynamics.
\end{remark}

\subsection{Continuous-time limit of generalized quadratic Tan--HWG dynamics}

\subsubsection{Control of the freezing error propagation}
\paragraph{Intuition.}
When the energy is Tan--HWG, the frozen energy is AGS compatible, which ensures that it is well behaved for a continuous-time limit, but only locally, relative to a fixed frozen signal. Structural Lipschitz assumptions and the sleep-mode assumption will ensure that the global signal's evolution is sufficiently smooth. From there, the key is to control the freezing error, and the propagation of this error. Indeed, when $\tau\to 0$, the number of steps increases relative to a given time horizon $T$ ($n\tau\le T$).

Therefore, we start by establishing a result allowing us to control the freezing error.

\begin{lemma}[Control of the freezing error]\label{lem:control-freezing}
  Assume the generalized quadratic energy
  \[
  E(\rho,\phi)
  := \frac{\alpha}{2} \int_X d_Z^2\big(\mathcal H_x(\rho_x,S_x(\rho,\phi)), S_x(\rho,\phi)\big)\, \mathrm{d}\mu(x)
  \]
  is a Tan--HWG energy. Under the regularity assumptions:
  \begin{itemize}
    \item \emph{Lipschitz signal} (Assumption~\ref{A1}):
    \[
      d_Z\big(S_x(\rho,\phi),S_x(\rho',\phi'))
      \le L_S \left( \mathcal W^2(\rho, \rho') + d^2_C(\phi, \phi') \right)^{1/2},
      \quad\text{$\mu$-a.e.},
    \]
    \item \emph{Lipschitz map} (Assumption~\ref{A2}):
    \[
      d_Z\big(\mathcal H_x(\rho_x,z),\mathcal H_x(\rho'_x,z'))
      \le L_H \left( W^2_2(\rho_x, \rho'_x) + d_Z^2(z,z') \right)^{1/2},
      \quad\text{$\mu$-a.e.},
    \]
  \end{itemize}
  there exists $L>0$ such that, for every timestep $\tau>0$ and every discrete Tan--HWG sequence $(\rho^n_\tau)_{n\ge0}$ associated with $E$, and with a $\mu$-compatible context sequence $(\phi^n_\tau)_{n\ge0}\in\mathcal C_\mu^{\mathbb N}$, under sleep-mode assumption
  \begin{itemize}
    \item \emph{sleep-mode} (Assumption~\ref{A3}):
    \[
      d_C(\phi^n_\tau, \phi^{n+1}_\tau) \le L_C\mathcal W(\rho^n_\tau, \rho^{n+1}_\tau),
    \]
  \end{itemize}
  we have the \emph{freezing error control inequality}
  \[
    \big|E(\rho_\tau^{n+1},\phi_\tau^{n+1})-E_{\phi_\tau^n,\rho_\tau^n}(\rho_\tau^{n+1})\big|
    \le \frac1{4\tau}\mathcal W^2(\rho_\tau^{n+1},\rho_\tau^n)
       + \alpha\,L^2\tau\,(E(\rho_\tau^{n+1},\phi_\tau^{n+1})+E(\rho_\tau^{n},\phi_\tau^{n})).
  \]
  Moreover, we have the \emph{pre-telescopic inequality}
  \[
  E(\rho_\tau^{n+1},\phi_\tau^{n+1}) + \frac1{4\tau}\mathcal W^2(\rho_\tau^{n+1},\rho_\tau^n)
  \le E(\rho_\tau^{n},\phi_\tau^{n}) + \alpha\,L^2\tau\,(E(\rho_\tau^{n+1},\phi_\tau^{n+1})+E(\rho_\tau^{n},\phi_\tau^{n})).
  \]
  and, for $\tau$ small enough, we have the \emph{Grönwall inequality}
  \[
    E(\rho_\tau^{n+1},\phi_\tau^{n+1})\le (1+4\alpha L^2\tau) E(\rho_\tau^n,\phi_\tau^n).
  \]
\end{lemma}

\begin{proof}
  Fix $\tau>0$. To simplify notations, we write $\rho^n:=\rho^n_\tau$, $\phi^n:=\phi^n_\tau$, $E^n:=E(\rho^n, \phi^n)$ and $S^n:=S(\rho^n,\phi^n)$.
  The frozen energy at step $n$ is
  \[
  E_{\phi^n,\rho^n}(\rho)
  := \frac\alpha2 \int_X d_Z^2\big(\mathcal H_x(\rho_x,S^n_x),S^n_x\big)\,\mathrm d\mu(x).
  \]

  \emph{Step 1: Error control.}
  We estimate the difference between the true and frozen energies at $\rho^{n+1}$:
  \begin{align*}
    \Delta_{n,n+1}
    :\!&\!= E^{n+1} - E_{\phi^n,\rho^n}(\rho^{n+1}).
  \end{align*}
  We have
  \begin{align*}
    |\Delta_{n,n+1}|
    &\!= \!\left|\frac\alpha2 \!\int_X \!\Big(
        d_Z^2\big(\mathcal H_x(\rho_x^{n+1},S_x^{n+1}),S_x^{n+1}\big)
        - d_Z^2\big(\mathcal H_x(\rho_x^{n+1},S_x^n),S_x^n\big)
      \Big)\mathrm d\mu(x)\right|\\
    &\!\le \frac\alpha2 \!\int_X\! \big(
        d_Z\big(\mathcal H_x(\rho_x^{n+1},S_x^{n+1}),S_x^{n+1}\big)
        \!+\!d_Z\big(\mathcal H_x(\rho_x^{n+1},S_x^n),S_x^n\big)
      \big)\delta_{n, n+1}
    \mathrm d\mu(x),
  \end{align*}
  with $\delta_{n, n+1}:=\big|
    d_Z\big(\mathcal H_x(\rho_x^{n+1},S_x^{n+1}),S_x^{n+1}\big)
    \!-\!d_Z\big(\mathcal H_x(\rho_x^{n+1},S_x^n),S_x^n\big)
  \big|$. By triangle inequalities we have
  \begin{align*}
    \delta_{n, n+1}
    &\le\! \big|
      d_Z\big(\mathcal H_x(\rho_x^{n+1},S_x^{n+1}),S_x^{n+1}\big)
      \!-\!d_Z\big(\mathcal H_x(\rho_x^{n+1},S_x^n),S_x^{n+1}\big)
    \big|\\
    &\qquad\qquad\!+\! \big|
      d_Z\big(\mathcal H_x(\rho_x^{n+1},S_x^n),S_x^{n+1}\big)
      \!-\!d_Z\big(\mathcal H_x(\rho_x^{n+1},S_x^n),S_x^n\big)
    \big|\\
    &\le\! d_Z\big(\mathcal H_x(\rho_x^{n+1},S_x^{n+1}),\mathcal H_x(\rho_x^{n+1},S_x^n)\big)
    \!+\!
      d_Z\big(S_x^{n+1},S_x^n\big)
  \end{align*}
  Then, by Assumptions~\ref{A1} and \ref{A2}
  \begin{align*}
    \delta_{n, n+1}
    &\le\! L_H\,d_Z(S_x^{n+1},S_x^n)
      \!+\!d_Z\big(S_x^{n+1},S_x^n\big)
      &\quad\text{(by Assumption~\ref{A2})}\\
    &\le (1+L_H)L_S\left(\mathcal W^2(\rho^{n+1},\rho^n)+d_C^2(\phi^{n+1},\phi^n)\right)^{1/2}
      &\quad\text{(by Assumption~\ref{A1})}\\
    &\le (1+L_H)L_S\left(\mathcal W^2(\rho^{n+1},\rho^n)
        +L_C^2\mathcal W^2(\rho^{n+1},\rho^n)\right)^{1/2}
      &\quad\text{(by Assumption~\ref{A3})}\\
    &\le (1+L_H)L_S\sqrt{1+L_C^2}\,\mathcal W(\rho^{n+1},\rho^n)\\
    &\le \tilde L\,\mathcal W(\rho^{n+1},\rho^n),
  \end{align*}
  with $\tilde L:=L_S(1+L_H)\sqrt{1+L_C^2}$. Thus, we obtain
  \begin{align*}
    |\Delta_{n,n+1}|
    &\!\le \frac\alpha2 \tilde L\,\mathcal W(\rho^{n+1},\rho^n)
      \!\int_X\! \big(
        d_Z\big(\mathcal H_x(\rho_x^{n+1},S_x^{n+1}),S_x^{n+1}\big)
        \!+\!d_Z\big(\mathcal H_x(\rho_x^{n+1},S_x^n),S_x^n\big)
      \big)
    \mathrm d\mu(x),
  \end{align*}
  By Cauchy--Schwarz,
  \begin{align*}
    |\Delta_{n,n+1}|
    &\!\le\! \frac\alpha2 \tilde L\mathcal W(\rho^{n+1}\!,\rho^n)
    \!\left[ \!\int_X \!\big(
        d_Z\big(\mathcal H_x(\rho_x^{n+1}\!,S_x^{n+1}),S_x^{n+1}\big)
        \!+\!d_Z\big(\mathcal H_x(\rho_x^{n+1}\!,S_x^n),S_x^n\big)
      \big)^2 \mathrm d\mu(x) \!\right]^{\!\frac12\!}
    \!\mu(X)^{\!\frac12\!}\\
    &\!\le\! \frac\alpha2 L\mathcal W(\rho^{n+1}\!,\rho^n)
    \!\left[ \!\int_X \!\big(
        d_Z\big(\mathcal H_x(\rho_x^{n+1}\!,S_x^{n+1}),S_x^{n+1}\big)
        \!+\!d_Z\big(\mathcal H_x(\rho_x^{n+1}\!,S_x^n),S_x^n\big)
      \big)^2 \mathrm d\mu(x) \!\right]^{\!\frac12\!},
  \end{align*}
  with $L:=\tilde L\,\mu(X)^{\!\frac12\!}$.
  Using $(a+b)^2\le 2(a^2+b^2)$, we obtain
  \begin{align*}
    |\Delta_{n,n+1}|
    &\le \sqrt{\alpha}\,L\,
      \sqrt{E^{n+1}+E_{\phi^n,\rho^n}(\rho^{n+1})}
      \,\mathcal W(\rho^{n+1}\!,\rho^n).
  \end{align*}

  Moreover, by the discrete energy inequality, $E_{\phi^n,\rho^n}(\rho^{n+1})\le E^n$.
  Hence
  \[
  |\Delta_{n,n+1}|
    \le \sqrt{\alpha}\,
    L\,\sqrt{E^{n+1}+E^n}\,\mathcal W(\rho^{n+1},\rho^n).
  \]
  Using Young's inequality $ab\le \frac1{4\tau}a^2 + \tau b^2$ with
  \(
  a=\mathcal W(\rho^{n+1},\rho^n),\
  b=\sqrt{\alpha}\,L\,\sqrt{E^{n+1}+E^n},
  \)
  we get
  \[
  \sqrt{\alpha}\,L\,\sqrt{E^{n+1}+E^n}\,\mathcal W(\rho^{n+1},\rho^n)
  \le \frac1{4\tau}\mathcal W^2(\rho^{n+1},\rho^n)
     + \alpha\,L^2\tau\,(E^{n+1}+E^n).
  \]
  Thus
  \[
  \big|E^{n+1}-E_{\phi^n,\rho^n}(\rho^{n+1})\big|
    \le \frac1{4\tau}\mathcal W^2(\rho^{n+1},\rho^n)
       + \alpha\,L^2\tau\,(E^{n+1}+E^n),
  \]
  which proves the first claim.

  \emph{Step 2: Pre-telescopic inequality.}
  Adding with the discrete energy inequality
  \[
  E_{\phi^n,\rho^n}(\rho^{n+1}) + \frac1{2\tau}\mathcal W^2(\rho^{n+1},\rho^n)
  \le E_{\phi^n,\rho^n}(\rho^n) = E(\rho^n, \phi^n)=E^n,
  \]
  gives
  \[
  E^{n+1} + \frac1{4\tau}\mathcal W^2(\rho^{n+1},\rho^n)
  \le E^n + \alpha\,L^2\tau\,(E^{n+1}+E^n).
  \]
  which proves the second claim.

  \emph{Step 3: Grönwall inequality.}
  We rewrite the previous inequality
  \[
  (1-\alpha\,L^2\tau)E^{n+1} + \frac1{4\tau}\mathcal W^2(\rho^{n+1},\rho^n)
  \le (1+\alpha\,L^2\tau)E^n.
  \]
  For $\tau$ small enough, $\alpha\,L^2\tau\le \frac12$, and we can write
  \[
  E^{n+1} + \frac1{4\tau(1-\alpha L^2\tau)}\mathcal W^2(\rho^{n+1},\rho^n)
  \le \frac{1+\alpha L^2\tau}{1-\alpha L^2\tau} E^n.
  \]
  Since $\frac{1+\alpha L^2\tau}{1-\alpha L^2\tau}\le 1+4\alpha L^2\tau$ whenever $\alpha L^2\tau\le \frac12$, we obtain
  \[
  E^{n+1} + \frac1{4\tau(1-\alpha L^2\tau)}\mathcal W^2(\rho^{n+1},\rho^n)
  \le (1+4\alpha L^2\tau) E^n.
  \]
  Then,
  \[
    E^{n+1} \le (1+4\alpha L^2\tau) E^n,
  \]
  which proves the final claim.
\end{proof}

This result allows us to control the propagation of the second moments:

\begin{lemma}[Stability of the dynamics]
  \label{lem:stability}
  Let $E$ be a Tan--HWG GQE, and assume Assumptions~\ref{A1} and \ref{A2} are true.
  Let $\tau>0$, and $(\rho^n_\tau)_{n\ge0}$ be the discrete Tan--HWG sequence associated with $E$,
  and with a $\mu$-compatible context sequence $(\phi^n_\tau)_{n\ge0}\in\mathcal C_\mu^{\mathbb N}$, under sleep-mode assumption~\ref{A3}, with initial conditions $\rho^0_\tau=\rho^0\in\mathcal X_\mu$ and $\phi^0_\tau=\phi^0\in\mathcal C_\mu$ independent of $\tau$.

  Then, for every $T>0$, there exists a constant $C_T>0$, independent of $\tau$, such that for all $n\in\mathbb N$ with $n\tau\le T$,
  \begin{align}
    \label{eq:stab-energy}
    E(\rho^n_\tau,\phi^n_\tau) &\le C_T
    \quad\text{(Energy stability)},\\[1mm]
    \label{eq:stab-action}
    \sum_{k=0}^{n-1}\frac1{\tau}\mathcal W^2(\rho^{k+1}_\tau,\rho^k_\tau) &\le C_T
    \quad\text{(Action stability)},\\[1mm]
    \label{eq:stab-moment}
    \int_X\int_Y d_Y(y,y_0)^2\,\mathrm d\rho^n_{\tau,x}(y)\,\mathrm d\mu(x) &\le C_T
    \quad\text{(Moment stability)},
  \end{align}
  for some (hence any) $y_0\in Y$, and for $\tau$ small enough.
\end{lemma}

\begin{proof}
  We use the same notations as in the previous proof. Fix $T>0$.

  \emph{Step 1: Energy stability.}
  By the Grönwall inequality of Lemma~\ref{lem:control-freezing}, we have
  \[
    E^{n+1} \le (1+4\alpha L^2\tau) E^n \le e^{4\alpha L^2\tau}E^n,
  \]
  which gives, for all $n\tau\le T$,
  \[
    E^n \le e^{4L^2n\tau}E^0
    \le C_T^{(1)},
    \quad\text{with }C_T^{(1)}:=e^{4\alpha\,L^2 T}E^0\text{ independent of }\tau.
  \]

  \emph{Step 2: Action stability.}
  Summing the pre-telescopic inequality of Lemma~\ref{lem:control-freezing} gives
  \begin{align*}
    E^n + \frac1{4}\sum_{k=0}^{n-1}\frac1{\tau}\mathcal W^2(\rho^{k+1},\rho^k)
    &\le E^0 + \alpha\,L^2\tau\sum_{k=0}^{n-1}(E^{k+1}+E^k)\\
    &\le E^0 + 2\alpha\,L^2n\tau e^{4\alpha\,L^2 n\tau}E^0.
  \end{align*}
  Hence, for all $n\tau\le T$,
  \[
    \sum_{k=0}^{n-1}\frac1{\tau}\mathcal W^2(\rho^{k+1},\rho^k)
    \le C_T^{(2)}
    \quad\text{with }C_T^{(2)}:=4E^0(1 + 2\alpha\,L^2Te^{4\alpha\,L^2 T})\text{ independent of }\tau.
  \]

  \emph{Step 3: Propagation of second moments.}
  Fix $y_0\in Y$ and define
  \[
  M(\rho):=\int_X\int_Y d_Y(y,y_0)^2\,\mathrm d\rho_x(y)\,\mathrm d\mu(x).
  \]
  Fix $n\in\mathbb N$ and $x\in X$.
  Let $\pi_x\in\Pi(\rho^{n+1}_x,\rho^n_x)$ be an optimal coupling between $\rho^{n+1}_x$ and $\rho^n_x$.
  Since
  \begin{align*}
    d_Y(y,y_0)^2 &\le d_Y(y',y_0)^2+d_Y(y,y')^2+2d_Y(y',y_0)d_Y(y,y')\\
      &\le (1+\tau)\,d_Y(y',y_0)^2 + \left(1+\frac1\tau\right)\,d_Y(y,y')^2,\quad\text{(using Young's inequality)}
  \end{align*}
  we obtain
  \[
  \int_{Y\!\times\! Y} \!d_Y(y,y_0)^2\mathrm d\pi_x(y,y')
  \le \left(1\!+\!\tau\right)\!\int_{Y\!\times\! Y} \!d_Y(y',y_0)^2\mathrm d\pi_x(y,y')
     + \left(1\!+\!\frac1\tau\right)\!\int_{Y\!\times\! Y} \!d_Y(y,y')^2\mathrm d\pi_x(y,y').
  \]
  Which, by definition of the marginals of $\pi_x$ and the Wasserstein distance, yields
  \[
  \int_Y d_Y(y,y_0)^2\,\mathrm d\rho^{n+1}_x(y)
  \le \left(1\!+\!\tau\right)\int_Y d_Y(y',y_0)^2\,\mathrm d\rho^n_x(y')
     + \left(1\!+\!\frac1\tau\right)W_2^2(\rho^{n+1}_x,\rho^n_x)
  \]
  Integrating over $X$ gives
  \[
  M(\rho^{n+1})\le \left(1\!+\!\tau\right)M(\rho^n)+\left(1\!+\!\frac1\tau\right)\mathcal W^2(\rho^{n+1},\rho^n).
  \]
  Summing this inequality and using the action stability result of Step~2 yields, for $n\tau\le T$,
  \begin{align*}
    M(\rho^{n})
    &\le M(\rho^0) + \tau\sum_{k=0}^{n-1}M(\rho^k)
    +\left(1\!+\!\frac1\tau\right)\sum_{k=0}^{n-1}\mathcal W^2(\rho^{k+1},\rho^k)\\
    &\le M(\rho^0) + \tau\sum_{k=0}^{n-1}M(\rho^k)
    +(\tau+1)C_T^{(2)}\\
    &\le \left(M(\rho^0)+2C_T^{(2)}\right) + \sum_{k=0}^{n-1}M(\rho^k)\tau,
    \quad\text{for $\tau\le 1$}.
  \end{align*}
  By Grönwall lemma we obtain
  \[
    M(\rho^{n}) \le \left(M(\rho^0)+2C_T^{(2)}\right)e^{n\tau}\le C_T^{(3)}
    \quad\text{with }C_T^{(3)}:=\left(M(\rho^0)+2C_T^{(2)}\right)e^T\text{ independent of }\tau.
  \]

  Setting $C_T=\max\{C_T^{(1)}, C_T^{(2)}, C_T^{(3)}\}$ proves the claim.
\end{proof}

\subsubsection{Existence of subsequential limits}

We now show the existence of subsequential limits to the Tan--HWG scheme when $\tau\to 0$, in the case of Tan--HWG GQE, under the three regularity assumptions of Section~\ref{sec:assumptions}.
The moment stability result of the previous lemma is key to ensure the existence of such limits. Moreover, we show that the limit curve is locally absolutely continuous with respect to the Wasserstein metric and satisfies a \emph{perturbed energy--dissipation inequality}. This provides a rigorous continuous-time formulation of the dynamics.

\begin{theorem}[Continuous dynamics for Tan--HWG GQE]
  \label{thm:tan-hwg-continuous}
  Assume $\Gamma$ is a finite set.
  Let $E$ be a Tan--HWG GQE, and assume Assumptions~\ref{A1} and \ref{A2} are true.

  Let $(\rho_\tau^n)_{n\ge0}$ be the Tan--HWG minimizing-movement scheme with timestep $\tau>0$, associated with $\mu$-compatible contexts $(\phi_\tau^n)_{n\ge0}$, under sleep-mode assumption~\ref{A3}, with fixed initial conditions $\phi_\tau^0=\phi^0\in\mathcal C_\mu$, and $\rho^0_\tau=\rho^0\in\mathcal X_\mu$ independent of $\tau$.

  Let $\rho_\tau$ be the piecewise-constant interpolation $\rho_\tau(t):=\rho^{\lfloor t/\tau \rfloor}_\tau$, and $\phi_\tau$ be the piecewise-constant interpolation $\phi_\tau(t):=\phi^{\lfloor t/\tau \rfloor}_\tau$

  Then, as $\tau\to0$, the family $(\rho_\tau)_\tau$ admits limit curves along subsequences, with associated limit curves of $(\phi_\tau)_\tau$.

  Moreover, for any limit curves $\rho:[0,\infty)\to\mathcal X_\mu$ and $\phi:[0,\infty)\to\mathcal C_\mu$:
  \begin{itemize}
    \item $\rho\in AC_{\mathrm{loc}}([0,\infty);(\mathcal X_\mu,\mathcal W))$ (i.e.\ $\rho$ locally absolutely continuous);
    \item for all $t\ge 0$, we have the \emph{perturbed energy--dissipation inequality}
    \[
    E(\rho(t),\phi(t)) +\frac14\int_0^t |\dot \rho|^2(s)\,\mathrm ds
    \le E(\rho^0,\phi^0)+2\alpha L^2\int_0^t E(\rho(s),\phi(s))\,\mathrm ds;
    \]
    \item in particular, $t\mapsto E(\rho(t),\phi(t))$ is locally absolutely continuous and obeys the Grönwall bound
    \[
    E(\rho(t),\phi(t))\le E(\rho^0,\phi^0)\,e^{2\alpha L^2 t}.
    \]
  \end{itemize}

\end{theorem}

\begin{proof}
  \emph{Step 1: existence of limit curve by equicontinuity and compactness.}
  Fix $T>0$. For $s,t\in[0,T]$ with $s<t$, let $n=\lfloor s/\tau\rfloor$, and $m=\lfloor t/\tau\rfloor$.
  By the triangle inequality, Cauchy--Schwarz and stability equation \eqref{eq:stab-action} of Lemma~\ref{lem:stability},
  \begin{align*}
    \mathcal W(\rho_\tau(t),\rho_\tau(s))
    =\mathcal W(\rho_\tau^m,\rho_\tau^n)
    &\le \sum_{k=n}^{m-1}\mathcal W(\rho^{k+1}_\tau,\rho^k_\tau)\\
    &\le \left(\sum_{k=n}^{m-1}\tau\right)^{1/2}
         \left(\sum_{k=n}^{m-1}\frac1{\tau}\mathcal W^2(\rho^{k+1}_\tau,\rho^k_\tau)\right)^{1/2}\\
    &\le \left((m-n)\tau\right)^{1/2}
         \left(\sum_{k=n}^{m-1}\frac1{\tau}\mathcal W^2(\rho^{k+1}_\tau,\rho^k_\tau)\right)^{1/2}\\
    &\le \sqrt{2C_T}\,|t-s|^{1/2} \qquad\text{(for $\tau$ small enough)}.
  \end{align*}
  Thus $(\rho_\tau)_{\tau>0}$ is equicontinuous in $(\mathcal X_\mu,\mathcal W)$ on $[0,T]$.

  The uniform second-moment bound \eqref{eq:stab-moment} of Lemma~\ref{lem:stability} and Prokhorov's theorem yield tightness in each fibre, hence relative compactness in the narrow topology on $\mathcal X_\mu$.

  By a standard Arzelà--Ascoli/diagonal argument, there exists a subsequence $\tau_k\downarrow0$ and a limit curve $\rho:[0,T]\to\mathcal X_\mu$ such that
  \[
  \rho_{\tau_k}(t)\rightharpoonup\rho(t)\quad\text{narrowly in }\mathcal X_\mu
  \]
  for all $t\in[0,T]$, and $\mathcal W(\rho_{\tau_k}(t),\rho(t))\to0$.

  Extending this construction to $[0,\infty)$ by a diagonal argument yields the claimed convergence.
  Moreover, the uniform second-moment bound \eqref{eq:stab-moment} implies uniform integrability of $d_Y(\cdot,y_0)^2$;
  hence, along the extracted subsequence, narrow convergence of $(\rho_{\tau_k}(t))_x$ together with convergence of second moments yields $W_2\big((\rho_{\tau_k}(t))_x,\rho_x(t)\big)\to0$ for $\mu$-a.e.\ $x$, and therefore $\mathcal W(\rho_{\tau_k}(t),\rho(t))\to0$ by dominated convergence.

  Moreover, by sleep-mode assumption~\ref{A3},
  \[
    d_C(\phi_\tau(t),\phi_\tau(s))
    \le L_C\sum_{k=n}^{m-1}\mathcal W(\rho^{k+1}_\tau,\rho^k_\tau)
    \le L_C\,\sqrt{2C_T}\,|t-s|^{1/2} \qquad\text{(for $\tau$ small enough)},
  \]
  and $(\phi_\tau)_{\tau>0}$ is also equicontinuous in $(\mathcal C_\mu, d_C)$ on $[0,T]$. We also have
  \[
    \left\|\phi_\tau(t)-\phi^0\right\|^2=d_C^2(\phi_\tau(t),\phi^0)\le L_C^2\,2C_T\,t,\quad t\in[0,T],
  \]
  so that $\left\|\phi_\tau\right\|$ is uniformly bounded on $[0,T]$. Since $\Gamma$ is a finite set, we have uniform tightness of total variations of $(\phi_{\tau_k})_{\tau_k}$ in each fiber.
  Hence, we can apply Prokhorov's theorem extension to the family of complex finite measures $(\phi_{\tau_k})_{\tau_k}$. Following same arguments as above for $(\rho_\tau)_\tau$, applied to $(\phi_{\tau_k})_{\tau_k}$, we can extract a subsequence $\tau_{k'}$ from $\tau_k$ to obtain both limit curves $\rho$ and $\phi$ as claimed.

  \emph{Step 2: absolute continuity and metric derivative.}
  The discrete action bound \eqref{eq:stab-action} implies
  \begin{align*}
    \int_\tau^T \left(\frac{\mathcal W(\rho_\tau(t),\rho_\tau(t-\tau))}{\tau}\right)^2\,\mathrm d t
    \le \sum_{k=0}^{\lfloor T/\tau\rfloor-1}\tau
       \left(\frac{\mathcal W(\rho^{k+1}_\tau,\rho^k_\tau)}{\tau}\right)^2
    \le C_T.
  \end{align*}
  Passing to the limit along $\tau_k\downarrow0$ and using the lower semicontinuity of the metric derivative (see e.g.\ \parencite{ags}), we obtain
  \[
  \int_0^T |\dot \rho|^2(t)\,\mathrm d t \le C_T,
  \]
  where $|\dot \rho|(t)$ is the metric derivative of $\rho$ relative to $\mathcal W$. So $\rho\in AC([0,T];(\mathcal X_\mu,\mathcal W))$ for all $T>0$, hence $\rho\in AC_{\mathrm{loc}}([0,\infty);\mathcal X_\mu)$.

  \emph{Step 3: energy-dissipation inequality.}
  By Lemma~\ref{lem:control-freezing}:
  \begin{align*}
    E(\rho^{n+1}_\tau,\phi^{n+1}_\tau) + \frac1{4\tau}\mathcal W^2(\rho^{n+1}_\tau,\rho^n_\tau)
    &\le E(\rho^n_\tau,\phi^n_\tau) + \alpha L^2\tau (E(\rho^n_\tau,\phi^n_\tau)+(1+4\alpha L^2\tau)E(\rho^n_\tau,\phi^n_\tau))\\
    &\le (1+2\alpha L^2\tau+4\alpha^2 L^4\tau^2)E(\rho^n_\tau,\phi^n_\tau).
  \end{align*}
  Summing over $n$ from $0$ to $\lfloor t/\tau\rfloor-1$
  \[
    E(\rho_\tau(t),\phi_\tau(t))
    + \frac14\sum_{k=0}^{\lfloor t/\tau\rfloor-1}\frac1{\tau}\mathcal W^2(\rho^{k+1}_\tau,\rho^k_\tau)
    \,\le\, E(\rho^0,\phi^0) + 2\alpha L^2(1+2\alpha L^2\tau)\sum_{k=0}^{\lfloor t/\tau\rfloor-1} E(\rho_\tau^k,\phi_\tau^k) \tau.
  \]
  Passing to the limit along $\tau_k\downarrow0$, using the lower semicontinuity of the metric derivative and the convergence $\rho_{\tau_k}(t)\to\rho(t)$, we obtain the following continuous energy-dissipation inequality (perturbed EDI)
  \[
  E(\rho(t),\phi(t))
  + \frac14\int_0^t |\dot \rho|^2(s)\,\mathrm ds
  \le E(\rho^0,\phi^0) + 2\alpha L^2\int_0^t E(\rho(s),\phi(s))\,\mathrm ds,
  \]
  for all $t\ge0$.

  \emph{Step 4: Grönwall bound.}
  By Grönwall's lemma we infer
  \[
  E(\rho(t),\phi(t))\le E(\rho^0,\phi^0)e^{2\alpha L^2 t}
  \]
  for all $t\ge0$.
\end{proof}

\begin{remark}[Structural non-deterministic dynamics]
  Since geodesic convexity of $E$ is not established, the limiting Tan--HWG dynamics \emph{need not be unique}. In the AGS framework, uniqueness of the gradient flow is guaranteed only for geodesically $\lambda$-convex energies with $\lambda\ge 0$. When convexity fails, multiple curves of maximal slope may coexist, even under identical initial conditions.

  This implies that an agent may follow \emph{different limiting trajectories} starting from the same memory state.
  Such structural non-determinism is \emph{not} due to any stochastic component of the model: a genuine Wasserstein gradient flow would be fully deterministic in the absence of external randomness. Here, the multiplicity of limit curves stems from the sequential freezing of the global signal at each Tan–HWG step.
  This phenomenon is consistent with neuroscientific observations: identical stimuli can elicit distinct neural trajectories depending on the internal state of the network \parencite{Arieli1996,Churchland2010,Buonomano2009,Mante2013,Rabinovich2008}. In this sense, the framework naturally captures the intrinsic variability of cortical dynamics.
\end{remark}

\section*{Discussion: Geometry of Memory}

The framework developed in this work proposes a geometric reinterpretation of synaptic plasticity, neural computation, and memory consolidation. By embedding memory states in Wasserstein spaces and by introducing contextualized geometric representation maps, we obtain a unified variational formulation of Hebbian learning that naturally incorporates stochasticity, multi-semanticity, oscillatory structure, and multi-timescale dynamics. This section discusses the conceptual implications of the model, its relationship to biological and artificial neural systems, and its limitations.

\paragraph{Geometry as a flexible and unifying cognitive language.}
A central contribution of this work is the shift from fixed Euclidean embeddings to \emph{geometric memories} represented as probability measures. This representation generalises existing models: classical neural networks appear as degenerate cases in which internal distributions collapse to Dirac masses. The formulation allows synapses to encode not only scalar weights but full distributions, capturing uncertainty, multimodality, and latent semantic structure. Observable synaptic weights arise as projections of these internal states, and their geometry is governed by Wasserstein transport. This perspective provides a \emph{principled explanation for several phenomena typically introduced heuristically} in machine learning, such as attention mechanisms, pruning, sparsity, synchronization and representational drift.

\paragraph{A duality between rich internal dynamics and flat observable projections.}
The introduction of \emph{mapped energies} generalizes the idea that complex architectures can be decomposed into elementary geometric maps whose regularity properties determine the global behavior of the system. This decomposition clarifies the role of activation functions, which no longer serve as the primary source of non-linearity but are instead practical building blocks for \emph{geometrically compatible projections} of a fundamentally curved latent space.

Hence, the framework offers a \emph{principled geometric account of cognitive computation}, suggesting that cognition may be more naturally described in geometric terms (distances, curvatures, geodesics) rather than algebraic ones (vectors, scalars, matrices). In particular, algebraic descriptions of neural computation may represent flattened projections of richer geometric dynamics.

\paragraph{The quadratic--affine regime as an anchor point.}
The quadratic--affine regime plays the role of an anchor point where the geometric and classical views coincide. When the internal energy is quadratic, the dynamics projects to geodesics, and when the projection map is affine, internal geodesics project to affine observable dynamics. This regime recovers classical update rules such as \emph{exponential moving averages}, \emph{mirror descent}, and state-independent contraction factors, and admits a clean continuous-time limit. As such, it provides a \emph{consistency check for the framework}: in the geometrically ideal case, the model reduces to well-known computational schemes. Deviations from affine behavior can then be interpreted as signatures of underlying geometric curvature. Classical indicators of nonlinearity (e.g.\ discrete second differences, orthogonal update components, or angles between successive gradients) acquire a geometric interpretation as measures of deviation from geodesic flatness.

\paragraph{Mathematical constraints as organisational principles.}
A key insight of the framework is the \emph{rigidity of geodesic affinity} in Wasserstein spaces. Contextualized maps are rarely geodesically affine, which prevents the associated energies from being Tan--HWG. This geometric constraint naturally supports a \emph{multi-timescale coherence mechanism}: fast synaptic dynamics operate under fixed structural weights, while slow structural plasticity requires coarse-grained or quasi-stationary activations. This provides a mathematical explanation for the long-standing observation that memory consolidation occurs preferentially during sleep or quiet wakefulness, when external perturbations are minimal and neural activity is dominated by internally generated patterns.

Moreover, non-uniqueness of non-convex dynamics (discrete or continuous) supports structural non-determinism, again in coherence with neuroscientific observations. Hence, geometry acts not merely as a descriptive tool, but as a generative principle shaping the structure of learning dynamics.

\paragraph{Limitations and computational challenges.}
Several aspects of the framework remain interpretative and require empirical validation. The biological relevance of multi-semantic embedding, robustness, and stability in generalized neural networks remains to be demonstrated in practice. Computationally, Wasserstein updates are expensive, and fast closed-form solutions exist only in restricted cases. While entropic regularisation and Sinkhorn-type methods~\parencite{cuturi2013sinkhorn} accelerate the computation of optimal transport, the scalability of Tan--HWG dynamics for large-scale architectures remains an open question. Moreover, nothing guarantees that generalized networks constructed within this framework will match the performance of modern artificial neural networks, which currently excel across a wide range of tasks.

Future work may explore:
\begin{itemize}
  \item \emph{Dynamics:} explicit gradient-flow equations for the full synaptic distribution and a systematic classification of equilibria or convergence rates; the study of attractors, bifurcations, cycles, stability regimes, and long-term behavior of Tan--HWG flows may reveal new forms of memory organisation;
  \item \emph{Representations:} the impact of internal-space topology (e.g.\ different CAT(0) structures or metric trees, including forbidden connections, directed edges, or structured sparsity) on learning dynamics;
  \item \emph{Structures:} pruning and self-organised sparsity-natural consequences of the simplex constraint—suggest promising directions for performance optimisation;
  \item \emph{Cognition:} connection between geometric dynamics and cognitive phenomena, such as consolidation, abstraction, or compositionality;
  \item \emph{Applications:} continual learning, meta-learning, or neuromorphic computing.
\end{itemize}

\section*{Conclusion}

The Tan--HWG framework reframes Hebbian plasticity as a geometric process rather than an algebraic update of synaptic coefficients. By lifting memory states to Wasserstein spaces and expressing plasticity as a minimizing movement, learning appears as the evolution of probability measures constrained by curvature, geodesic rigidity, and the regularity of contextualized representation maps. This perspective reveals that many qualitative features of learning---stability, competition, consolidation, synchronization, sparsity---arise not from additional mechanisms but from the geometric structure in which memory evolves.

This shift challenges the traditional separation between ``weights'' and ``computation''. In the Tan--HWG view, observable weights are projections of a richer internal geometry, and classical update rules emerge as flat limits of curved dynamics. The rigidity of geodesic affinity, the emergence of multi-timescale coherence, and the structural non-determinism induced by frozen signals highlight a deeper principle: learning systems are shaped less by the rules they implement than by the geometry that makes these rules possible.

The framework also brings to light a structural tension inherent to synaptic plasticity. Internal dynamics must remain geometrically coherent, while observable behavior must remain simple, stable, and low-dimensional. The coexistence of curved internal trajectories with affine observable recursions is not an artefact of the model but a consequence of compressing rich internal representations into actionable outputs. This duality provides a geometric explanation for the stability of observable behavior despite the complexity of internal dynamics.

Under quasi-stationary contexts, the existence of continuous-time limit curves shows that consolidation corresponds to the regime in which the geometry can unfold without external disruption. In this sense, \emph{memory is the trace left by a trajectory in a curved space}: its organisation reflects both the information acquired and the geometric constraints under which learning takes place.

Rather than closing the theory, the framework delineates a landscape in which learning, representation, and consolidation are unified by a single geometric principle. It provides a foundation for analysing how internal geometries shape the dynamics of plasticity and how observable behavior emerges from the curvature of latent spaces.

%------------------------------------------------------------------------------

\clearpage
\appendix

\section{Self consistency and stability of fixed points in Tan-HWG dynamics}
\label{sec:appendix_fixed_points}

This appendix gathers the arguments underlying the fixed-point structure and stability properties of Tan--HWG dynamics. It is self-contained and independent of the examples presented in the main text.

\subsection{Fixed Points and Self-Consistency}\label{sec:fixed_points}

A distinctive feature of Tan--HWG dynamics is that the global signal driving the evolution of the local states is itself induced by the current configuration of the memory field. As a consequence, equilibrium configurations are characterized by a self-consistency condition linking the local states and the global signal they collectively generate. We formalize this notion below and relate fixed points of the dynamics to Wasserstein critical points of the energy.

\begin{proposition}[Fixed points as critical points]\label{prop:fixed_points}
  Let $E$ be a Tan--HWG energy, and fix a stationary context $\phi \in \mathcal M(\Gamma)^X$. Let $\tau>0$, $T^E_{\phi, \tau}$ denote the $\tau$-Tan--HWG update operator. Let $(\rho^n)$ be the associated dynamics. Then $\rho^*$ is a fixed point of the dynamics if and only if it is a Wasserstein critical point of $E(\cdot,\phi)$:
  \[
  \rho^* \in T_{\phi, \tau}(\rho^*) \text{ (fixed point)}
  \quad\Longleftrightarrow\quad
  0 \in\partial_{\mathcal W} E(\rho^*,\phi),
  \]
  where $\partial_{\mathcal W}$ denotes the Wasserstein subdifferential.
\end{proposition}

Remarkably, the characterization of fixed points is independent of the time step $\tau>0$: the parameter $\tau$ affects the transient dynamics but not the location of equilibrium configurations, which are solely determined by the critical points of the energy. In this sense, $\tau$ plays a role analogous to a learning rate in machine learning.

This result shows that equilibrium configurations of the Tan--HWG dynamics are precisely the self-consistent states where the Wasserstein subdifferential vanishes.

\subsection{Stability of fixed points}

We now consider the Brenier setting: quadratic cost on $\mathbb R^d$ with absolutely continuous measures. In this case, the Wasserstein subdifferential reduces to a singleton, and the condition $0\in\partial_{\mathcal W}E(\rho,\phi)$ is equivalent to $\nabla_{\mathcal W}E(\rho,\phi)=0$.

Linearising the dynamics around a fixed point $\rho^*$ in the Otto calculus yields
\[
\partial_t \delta\rho(t)
= - H_E(\rho^*,\phi)\,\delta\rho(t),
\]
where $H_E(\rho^\star,\phi)$ denotes the Wasserstein Hessian of $E(\cdot,\phi)$ at $\rho^\star$. The solution is
\[
\delta\rho(t)
= \exp\!\bigl(-t\,H_E(\rho^*,\phi)\bigr)\,\delta\rho(0).
\]
This yields the following lemma:

\begin{lemma}\label{lem:DT}
  The linearisation of $T^E_{\phi, \tau}$ at a fixed point $\rho^*$ is given by
  \[
  DT^E_{\phi, \tau}(\rho^*) = \exp\bigl(-\tau H_E(\rho^*,\phi)\bigr).
  \]
\end{lemma}

\begin{proposition}[Spectral stability]
  Let $E$ be a Tan--HWG energy, $\phi$ a fixed context, and $\rho^*$ a fixed point. Let $(\mu_i)_i$ be the eigenvalues of the Wasserstein Hessian $H_E(\rho^*,\phi)$ and $(\lambda_i)_i$ those of the Jacobian of the time--1 map of the gradient flow (or of the 1-Tan--HWG dynamics) at $\rho^*$. Then
  \[
  \lambda_i = e^{-\mu_i} \quad \text{for all } i.
  \]
  In particular:
  \[
  \mu_i > 0 \Rightarrow |\lambda_i| < 1,\quad
  \mu_i = 0 \Rightarrow |\lambda_i| = 1,\quad
  \mu_i < 0 \Rightarrow |\lambda_i| > 1.
  \]
\end{proposition}

The choice $\tau=1$ is made for simplicity, since the stability properties are independent of the time step.
This gives a simple classification of the dynamics:
\begin{itemize}
  \item $\mu_i > 0$ $\iff$ $|\lambda_i|<1$: stable directions (local attractor);
  \item $\mu_i = 0$ $\iff$ $|\lambda_i|=1$: neutral directions (flat manifolds or symmetries);
  \item $\mu_i < 0$ $\iff$ $|\lambda_i|>1$: unstable directions (saddles, bifurcations).
\end{itemize}

\begin{corollary}[Geometric consolidation in stationary context]
  Let $E$ be a Tan--HWG energy and $\phi \in \mathcal M(\Gamma)^X$. If $\rho^*$ is a fixed point such that the Wasserstein Hessian $H_E(\rho^*,\phi)$ is strictly positive definite, then $\rho^*$ is a locally asymptotically stable fixed point of the Tan--HWG dynamics. In particular, for all initial conditions $\rho^0$ sufficiently close to $\rho^*$, the iterates
  \[
  \rho^{n+1} \in T^E_{\phi, 1}(\rho^n)
  \]
  converge to $\rho^*$ as $n\to\infty$.
\end{corollary}

\subsection{Interpretation and dynamical consequences}

\begin{remark}[Cognitive interpretation and model-based prediction]
  In stationary contexts, Hebbian plasticity can be interpreted as a geometric consolidation process, whereby the dynamics converge towards local minima of the energy $E(\cdot,\phi)$ and stabilise the associated internal geometry.

  Although the set of equilibrium configurations is independent of $\tau>0$, the linearisation $DT^E_{\phi,\tau}(\rho^*)=\exp(-\tau H_E(\rho^*,\phi))$ shows that $\tau$ controls the contraction rate along stable directions. This suggests a possible link between the timescale of Hebbian updates and the temporal organisation of cognitive processes.

  \textbf{Computational constraints and effective equilibrium selection.}
  The dynamics may effectively select different equilibria depending on the available time horizon and numerical stability. Very small $\tau$ may lead to slow convergence, potentially too slow to be achieved within a finite time window, while excessively large $\tau$ may induce oscillatory or unstable behavior around otherwise stable equilibria. Thus, $\tau$ controls which equilibria are effectively reachable and consolidated.

  If discrete updates are gated by oscillatory cycles, the model suggests that different gating frequencies modulate consolidation rates and the selection of computationally compatible equilibria, rather than the location of the equilibria themselves.

  \textbf{Nested updates.}
  Multi-scale updates involving several time steps $\tau$, such as $\tau_1 \ll \tau_2$, may increase the speed and robustness of convergence compared to the use of a single time step. The family $\{T^E_{\phi, \tau}\}_{\tau>0}$ can be viewed as a temporal decomposition of the Hebbian update operator, analogous to a spectral decomposition.
\end{remark}

\begin{remark}[Bifurcations and phase transitions]
  As the context $\phi$ varies, the structure of the energy landscape $E(\cdot,\phi)$ may undergo qualitative changes. Bifurcations occur when eigenvalues of the Wasserstein Hessian cross zero, leading to the loss of stability of equilibria and the emergence of new stationary states. These events can be interpreted as geometric phase transitions in the space of representations.

  While the location of these transitions is determined by the energy landscape, the time step $\tau$ controls their dynamical manifestation, with critical slowing down near bifurcation points and enhanced sensitivity to perturbations.
\end{remark}

\begin{remark}[Effective dimensionality of representations]
  The spectral structure of the Wasserstein Hessian $H_E(\rho^*,\phi)$ provides a natural notion of effective dimensionality of the representation around an equilibrium. Directions associated with large positive eigenvalues are rapidly contracted and effectively frozen, while directions corresponding to small or vanishing eigenvalues remain dynamically accessible. As a result, even when the ambient representation space is high-dimensional, as commonly encountered in deep learning models, the dynamics effectively evolves on a lower-dimensional manifold whose dimension depends on the local geometry of the energy landscape. Changes in this effective dimensionality naturally occur at bifurcation points, when eigenvalues cross zero.
\end{remark}

\clearpage
\bibliographystyle{plainnat}
\bibliography{refs}

@book{villani,
  author    = {C{\'e}dric Villani},
  title     = {Optimal Transport: Old and New},
  publisher = {Springer},
  year      = {2009},
  series    = {Grundlehren der mathematischen Wissenschaften},
  volume    = {338}
}

@book{santambrogio,
  author    = {Filippo Santambrogio},
  title     = {Optimal Transport for Applied Mathematicians},
  publisher = {Birkh{\"a}user},
  year      = {2015},
  series    = {Progress in Nonlinear Differential Equations and Their Applications},
  volume    = {87}
}

@book{ags,
  author    = {Luigi Ambrosio and Nicola Gigli and Giuseppe Savar{\'e}},
  title     = {Gradient Flows in Metric Spaces and in the Space of Probability Measures},
  publisher = {Birkh{\"a}user},
  year      = {2008}
}

@book{hebb,
  author    = {Donald O. Hebb},
  title     = {The Organization of Behavior},
  publisher = {Wiley},
  year      = {1949}
}

@article{hopfield,
  author  = {John J. Hopfield},
  title   = {Neural networks and physical systems with emergent collective computational abilities},
  journal = {Proceedings of the National Academy of Sciences},
  year    = {1982},
  volume  = {79},
  number  = {8},
  pages   = {2554--2558}
}

@book{dayanabbott,
  author    = {Dayan, Peter and Abbott, L. F.},
  title     = {Theoretical Neuroscience: Computational and Mathematical Modeling of Neural Systems},
  publisher = {MIT Press},
  year      = {2001}
}

@book{gerstner2014neuronal,
  author    = {Wulfram Gerstner and Werner M. Kistler and Richard Naud and Liam Paninski},
  title     = {Neuronal Dynamics: From Single Neurons to Networks and Models of Cognition},
  publisher = {Cambridge University Press},
  year      = {2014}
}

@article{jko1998,
  author    = {Jordan, Richard and Kinderlehrer, David and Otto, Felix},
  title     = {The variational formulation of the Fokker--Planck equation},
  journal   = {SIAM Journal on Mathematical Analysis},
  volume    = {29},
  number    = {1},
  pages     = {1--17},
  year      = {1998},
  publisher = {SIAM}
}

@article{Sturm2003I,
  author  = {Sturm, Karl-Theodor},
  title   = {Probability measures on metric spaces of nonpositive curvature},
  journal = {Heat Kernels and Analysis on Manifolds, Graphs, and Metric Spaces},
  year    = {2003},
  pages   = {357--390}
}

@incollection{Kuramoto1975,
  author={Yoshiki Kuramoto},
  title={Self-entrainment of a population of coupled non-linear oscillators},
  booktitle={International Symposium on Mathematical Problems in Theoretical Physics},
  editor={H. Araki},
  series={Lecture Notes in Physics},
  volume={39},
  pages={420--422},
  year={1975},
  publisher={Springer}
}

@book{Villani2003,
  author={Cédric Villani},
  title={Topics in Optimal Transportation},
  year={2003},
  publisher={AMS}
}

@article{McCann1997,
  author  = {McCann, Robert J.},
  title   = {A convexity principle for interacting gases},
  journal = {Advances in Mathematics},
  volume  = {128},
  number  = {1},
  pages   = {153--179},
  year    = {1997},
  doi     = {10.1006/aima.1997.1634}
}

@article{cuturi2013sinkhorn,
  title={Sinkhorn Distances: Lightspeed Computation of Optimal Transportation Distances},
  author={Marco Cuturi},
  year={2013},
  journal={NeurIPS},
}

@article{Arieli1996,
  title={Dynamics of ongoing activity: explanation of the large variability in evoked cortical responses},
  author={Amos Arieli  and Alexander Sterkin  and Amiram Grinvald  and Ad Aertsen},
  journal={Science},
  volume={273},
  number={5283},
  pages={1868--1871},
  year={1996}
}

@article{Churchland2010,
  title={Stimulus onset quenches neural variability: a widespread cortical phenomenon},
  author={Churchland, Mark M and Yu, Byron M and Cunningham, John P and Sugrue, Leo P and Cohen, Marlene R and Corrado, Greg S and Newsome, William T and Clark, Andrew M and Hosseini, Paymon and Scott, Benjamin B and others},
  journal={Nature Neuroscience},
  volume={13},
  number={3},
  pages={369--378},
  year={2010}
}

@article{Mante2013,
  title={Context-dependent computation by recurrent dynamics in prefrontal cortex},
  author={Mante, Valerio and Sussillo, David and Shenoy, Krishna V and Newsome, William T},
  journal={Nature},
  volume={503},
  number={7474},
  pages={78--84},
  year={2013},
  publisher={Nature Publishing Group}
}

@article{Rabinovich2008,
  title={Transient dynamics for neural processing},
  author={Rabinovich, Misha and Huerta, Ram{\'o}n and Laurent, Gilles},
  journal={Science},
  volume={321},
  number={5885},
  pages={48--50},
  year={2008},
  publisher={American Association for the Advancement of Science}
}

@article{Buonomano2009,
  title={State-dependent computations: spatiotemporal processing in cortical networks},
  author={Buonomano, Dean V and Maass, Wolfgang},
  journal={Nature Reviews Neuroscience},
  volume={10},
  number={2},
  pages={113--125},
  year={2009},
  publisher={Nature Publishing Group}
}

@article{lipman2023flow,
  title={Flow Matching for Generative Modeling},
  author={Yaron Lipman and Ricky T. Q. Chen and Heli Ben-Hamu and Maximilian Nickel and Matt Le},
  journal={ICLR},
  year={2023}
}

@article{gat2024discreteflowmatching,
  title={Discrete Flow Matching},
  author={Itai Gat and Tal Remez and Neta Shaul and Felix Kreuk and Ricky T. Q. Chen and Gabriel Synnaeve and Yossi Adi and Yaron Lipman},
  year={2024},
  journal={NeurIPS}
}

@article{liu2024sfm,
  title={SFM: Stochastic Flow Matching for Discrete and Categorical Data},
  author={Liu, Xuehai and others},
  journal={NeurIPS},
  year={2024}
}

@inproceedings{hoogeboom2021argmax,
  title={Argmax Flows and Multinomial Diffusion: Learning Categorical Distributions},
  author={Emiel Hoogeboom and Didrik Nielsen and Priyank Jaini and Patrick Forré and Max Welling},
  booktitle={NeurIPS},
  year={2021}
}

@book{BridsonHaefliger1999,
  author    = {Bridson, Martin R. and Haefliger, Andr{\'e}},
  title     = {Metric Spaces of Non-Positive Curvature},
  series    = {Grundlehren der mathematischen Wissenschaften},
  volume    = {319},
  publisher = {Springer},
  year      = {1999}
}

@article{LottVillani2009,
  author  = {Lott, John and Villani, C{\'e}dric},
  title   = {Ricci curvature for metric-measure spaces via optimal transport},
  journal = {Annals of Mathematics},
  volume  = {169},
  number  = {3},
  pages   = {903--991},
  year    = {2009}
}

@article{degroot1974reaching,
  title={Reaching a consensus},
  author={DeGroot, Morris H},
  journal={Journal of the American Statistical Association},
  volume={69},
  number={345},
  pages={118--121},
  year={1974}
}

@INPROCEEDINGS{jadbabaie2003coordination,
  author={Jadbabaie, A. and Lin, J. and Morse, A.S.},
  booktitle={Proceedings of the 41st IEEE Conference on Decision and Control, 2002.},
  title={Coordination of groups of mobile autonomous agents using nearest neighbor rules},
  year={2002},
  volume={3},
  number={},
  pages={2953-2958 vol.3},
  doi={10.1109/CDC.2002.1184304}
}

@article{ren2005consensus,
  title={Consensus seeking in multiagent systems under dynamically changing interaction topologies},
  author={Ren, Wei and Beard, Randal W},
  journal={IEEE Transactions on Automatic Control},
  volume={50},
  number={5},
  pages={655--661},
  year={2005}
}

@ARTICLE{olfati2004consensus,
  author={Olfati-Saber, R. and Murray, R.M.},
  journal={IEEE Transactions on Automatic Control},
  title={Consensus problems in networks of agents with switching topology and time-delays},
  year={2004},
  volume={49},
  number={9},
  pages={1520-1533},
  doi={10.1109/TAC.2004.834113}
}

@ARTICLE{olfati2007consensus,
  author={Olfati-Saber, Reza and Fax, J. Alex and Murray, Richard M.},
  journal={Proceedings of the IEEE},
  title={Consensus and Cooperation in Networked Multi-Agent Systems},
  year={2007},
  volume={95},
  number={1},
  pages={215-233},
  doi={10.1109/JPROC.2006.887293}
}

@book{buzsaki2006rhythms,
  title={Rhythms of the Brain},
  author={Buzsáki, György},
  year={2006},
  publisher={Oxford University Press}
}

@article{wang2010neurophysiological,
  title={Neurophysiological and computational principles of cortical rhythms in Cognition},
  author={Wang, Xiao-Jing},
  journal={Physiological Reviews},
  year={2010},
  volume={90},
  number={3},
  pages={1195--1268}
}

@article{lisman2013theta,
  title={The Theta-Gamma Neural Code},
  author={Lisman, John E. and Jensen, Ole},
  journal={Neuron},
  year={2013},
  volume={77},
  number={6},
  pages={1002--1016}
}

@article{lampl1999synchrony,
  title={Synchronous membrane potential fluctuations in neurons of the cat visual cortex},
  author={Lampl I, Reichova I, Ferster D.},
  journal={Neuron},
  volume={22},
  number={2},
  pages={361--374},
  year={1999}
}

@book{sutton2018reinforcement,
  author    = {Richard S. Sutton and Andrew G. Barto},
  title     = {Reinforcement Learning: An Introduction},
  edition   = {2},
  publisher = {MIT Press},
  year      = {2018}
}

\end{document}